\documentclass[11pt]{article}

\usepackage{mathrsfs}

\usepackage{soul}

\usepackage[margin=1in]{geometry}
\usepackage[utf8]{inputenc}
\usepackage{enumerate}
\usepackage{amsfonts}
\usepackage{amsmath}
\usepackage{amssymb}
\usepackage{hyperref}
\usepackage{amsthm}
\usepackage{xcolor}
\usepackage{tikz}
\usepackage{comment}
\usepackage{kotex}
\usepackage[numbers]{natbib}

\usepackage{colonequals}
\usepackage{theoremref}
\usepackage{enumitem}
\numberwithin{equation}{section}
\usepackage{appendix}

\newcommand{\sh}{\sinh}
\newcommand{\ch}{\cosh}
\newcommand{\hth}{\tanh}
\newcommand{\sch}{\mathrm{sech}}
\newcommand{\e}{\mathbb E}

\newcommand{\p}{\mathbb P}
\newcommand{\1}{\mymathbb{1}}
\newcommand{\0}{\mymathbb{0}}
\DeclareMathAlphabet{\mymathbb}{U}{BOONDOX-ds}{m}{n}
\newcommand{\eps}{\varepsilon}

\newcommand{\lla}{\langle\!\langle}
\newcommand{\rra}{\rangle\!\rangle}
\newcommand{\biglla}{\bigl\langle\!\bigl\langle}
\newcommand{\bigrra}{\bigr\rangle\!\bigr\rangle}
\newcommand{\Biglla}{\Bigl\langle\!\Bigl\langle}
\newcommand{\Bigrra}{\Bigr\rangle\!\Bigr\rangle}

\newcommand{\lee}{\preceq}
\newcommand{\gee}{\succeq}
\newcommand{\slee}{\prec}
\newcommand{\sgee}{\succ}

\newcommand{\B}{\mbox{\tiny VB}}

\newtheorem{theorem}{Theorem}[section]

\newtheorem{proposition}[theorem]{Proposition}
\newtheorem{lemma}[theorem]{Lemma}

\newtheorem{definition}[theorem]{Definition}

\theoremstyle{definition} 
\newtheorem*{ack}{Acknowledgments}
\newtheorem{remark}[theorem]{Remark}

\begin{document}
\title{On the de Almeida--Thouless Transition Surface in the Multi-Species SK Model with Centered Gaussian External Field}
\date{}

\author{ Heejune Kim \thanks{ Email: kim01154@umn.edu.} }

\maketitle

\begin{abstract}
We study the phase transition of the Parisi formula for the free energy in the multi-species Sherrington--Kirkpatrick model with a centered Gaussian external field and a positive-semidefinite variance profile matrix.
We show that in terms of the strength of the external field and the variance profile, the de Almeida--Thouless surface delineates the boundary between replica symmetric solutions and replica symmetry breaking solutions.

\end{abstract}

\tableofcontents

\section{Introduction}
The canonical Sherrington–Kirkpatrick (SK) model of spin glasses, introduced in \cite{Sherrington1975}, features homogeneous interactions. Using the physicists' replica method, Parisi \cite{Par80} conjectured a variational formula for the limiting free energy, which was later proved by Talagrand \cite{Tal06}, building on Guerra's replica symmetry breaking bound \cite{Guerra}. Generalizations of the SK model have been analyzed, including bipartite models and, more broadly, multi-species SK (MSK) models that capture inhomogeneous interactions and connect to neural networks \cite{ACM21, BGGPT14, BGF10, BGG11, BGGT14, Tal09}. For the MSK model with a positive-semidefinite interaction matrix, the Parisi formula was established by an upper bound in \cite{BCMT15} together with a matching lower bound in \cite{Pan15}.

Two central questions in Parisi's framework are the uniqueness of the minimizer of the Parisi formula, called a Parisi measure, and the characterization of when a phase transition occurs, namely, when the Parisi measure is nontrivial (not a Dirac measure). The uniqueness question for the SK model was completely resolved in \cite{AC15} and was recently extended to the MSK model with strictly convex interaction in \cite{CIM25}. The phase transition question was partly addressed by Toninelli \cite{Ton02}, who proved the existence of a transition above the conjectured de Almeida--Thouless (AT) line \cite{AT78} for the SK model. 
For the MSK model with a positive-definite  interaction matrix, Bates--Sloman--Sohn \cite{BSS19} developed a multi-species version of Toninelli's perturbative argument, and Dey–Wu \cite{DW21} confirmed the expected transition above the AT surface.\footnote{It's now a surface, rather than a line.}
Although the AT line is known to be inexact for some single-species models, such as the Ghatak--Sherrington model \cite{Pan05}, it is widely believed to be exact in the SK case. However, rigorously proving the absence of a transition inside the AT line for the SK model has remained elusive, with some  partial progress in \cite{BY22, JT17}. 
On the other hand,  Talagrand \cite[Chapter 13]{Talagrand2013vol2} derived another  characterization of the phase transition for the SK model with deterministic external fields, but it is not known if it analytically coincides with the AT line.

In contrast to a deterministic external field, allowing the external field to be centered Gaussian relaxes several technical difficulties, and Chen \cite{Che21} established the exactness of the AT line for the single-species SK model.\footnote{The proof in \cite{Che21} applies verbatim to external fields with symmetric decreasing density, but its extension to the MSK setting is not immediate, due to the technicalities in Section~\ref{sec: Minimizers of replica symmetric functional}.}
In this paper, we prove that for the MSK model  with a centered Gaussian external field and a positive-semidefinite interaction matrix, the AT surface is exact; in particular, there is no replica symmetry breaking inside the AT surface.




The MSK model is defined as follows.
For each $N\in\mathbb N$, denote the $N$-dimensional binary cube by $\{\pm 1\}^N$.
Each $\sigma =(\sigma_1,\dots,\sigma_N) \in \{\pm 1\}^N$ is called a configuration, where its $i$'th entry $\sigma_i$ is called the spin at site $i$.
Fix a finite set of species $\mathscr S$, and for each $N\in \mathbb N$, we partition the sites $\{1,\dots,N\}=   \dot \cup_{s\in \mathscr S}I_{s,N}$ into a finite disjoint union of species for some $ I_{s,N} \subseteq \{1,\dots ,N\}$.
For brevity, we will always denote $I_{s}$ instead of $I_{s,N}$, dropping the dependence on $N.$
For each $s\in\mathscr S$, we asymptotically fix the ratio $\lambda_{s,N}\colonequals |I_s|/N$ by assuming that the following limit exists \begin{equation*}
    \lim_{N\to\infty} \lambda_{s,N} \equalscolon \lambda_s \in (0,1].
\end{equation*} 
Denote the diagonal matrices $\Lambda =\text{diag}(\lambda_s: s\in\mathscr S)$ and $\Lambda_N =\text{diag}(\lambda_{s,N}: s\in\mathscr S)$.
For each $1\le i\le N$, we define $s(i)$ to be the species that site $i$ belongs to, that is, $i\in I_{s(i)}$.
Consider a symmetric variance matrix $\Delta^2=(\Delta^2_{st})_{s,t\in \mathscr S}$  and a vector $\tau^2= (\tau^2_s)_{s\in\mathscr S}$,  both entrywise non-negative.
Here, the square does not indicate a square of matrices, but we adopt this convention from the literature.
We consider the Hamiltonian of the  multi-species Sherrington--Kirkpatrick  model with centered Gaussian external field \begin{equation}\label{eq: MSK Hamiltonian}
    H_{N}(\sigma) = \frac{1}{\sqrt{N}}\sum_{1\le i,j\le N} J_{ij}  \sigma_i\sigma_j + \sum_{1\le i\le N}h_{i}\sigma_i ,\quad \sigma\in \{\pm 1\}^N, 
\end{equation}
 where $(J_{ij})_{1\le i,j\le N}$ and $(h_i)_{1\le i\le N}$ are  independent centered Gaussian random variables with inhomogeneous variance profile \begin{equation}\label{eq: variance of random variables}
 \begin{aligned}
     \e J_{ij}^2&= \Delta^2_{ss'} \quad \text{if $i\in I_s$, $j\in I_{s'}$ for $s,s' \in \mathscr S$},
     \\\e h_i^2 &=\tau_s^2\qquad \text{if $i\in I_s$ for $s\in \mathscr S$}.
 \end{aligned}
 \end{equation} 
The free energy associated to the Hamiltonian \eqref{eq: MSK Hamiltonian} is given by \begin{equation*}
    F_N=  \frac{1}{N} \log \sum_{\sigma\in\{\pm 1\}^N} \exp( H_N(\sigma)).
\end{equation*} 
For two configurations $\sigma^1$ and $\sigma^2$, we define the overlap vector by \[R_{12}= R(\sigma^1,\sigma^2)= \Bigl( \frac{1}{|I_s|}\sum_{i\in I_s}\sigma^1_i\sigma^2_i\Bigr)_{s\in\mathscr S}. \]
Without loss of generality, we may further assume that $\Delta^2$ is irreducible:  for any $a,b \in\mathscr S$, there exists a finite sequence $(s_1, \dots ,s_n)$  in $\mathscr S$ with $s_1=a$, $s_n=b$, and $\Delta^2_{s_i s_{i+1}}>0$ for all $1\le i\le n-1$.
Indeed, if this is not the case, the Hamiltonian \eqref{eq: MSK Hamiltonian} decomposes into a sum of Hamiltonians corresponding to ``noninteracting'' subsystems, which may then be treated separately as irreducible cases.


To prepare for the main result,  recall the Parisi formula of the MSK model  \cite{Pan15} for positive-semidefinite $\Delta^2$.
Given an integer $r\ge 1$, consider the parameter sequences \begin{align}
    0&=\zeta_{-1}\le  \zeta_0 < \dots <\zeta_{r-1}\le \zeta_{r}=1, \label{eq: zeta sequence}
    \\0&= q_0^s \le q_1^s \le \dots \le q_{r-1}^s\le q_{r}^s =1, \quad s \in \mathscr S, \label{eq: q sequence}
\end{align}  which we loosely denote by  $(\zeta_l)$ and $(q_l)$. 
We always denote the entrywise inequality for any vectors or matrices by $\lee$, the all-one vector by $\1$, and the all-zero vector or matrix $\0 $.
If strict inequalities hold for all entries,  we write $\slee.$
For two vectors $a$ and $b$,  we say $a\ne b$ if they differ in at least one entry.
For example, we can express \eqref{eq: q sequence} more succinctly as $\0\lee  q_1\lee \dots \lee q_{r-1}\lee \1$.
 With these parameters, for each $0\le l\le r$,  define \begin{equation} \label{eq: definition of Q}
     Q_l= q_l^{\intercal}\Lambda\Delta^2 \Lambda q_l  \quad \text{and} \quad  Q_l^s= 2(\Delta^2 \Lambda q_l)_s, \quad s \in \mathscr S.
 \end{equation} 
Let  $(z_i)_{1\le i\le r}$ be i.i.d. standard Gaussian random variables and $\e_{l}$ denote the expectation in $z_{l}$.
Let $(h_s)_{s\in\mathscr S}$ be i.i.d. centered Gaussian random variables with variance  $\e h^2_s =\tau^2_s$ for each $s\in\mathscr S.$
For each species $s\in \mathscr S$, we inductively define \begin{align}
    X^s_{l} &= \frac{1}{\zeta_{l}} \log  \e_{l+1}\exp(\zeta_{l} X^s_{l+1}), \quad 0\le l\le r-1, \label{eq: inductive definition}
\end{align} where the terminal random variable is \begin{equation}\label{eq: terminal X}
    X^s_{r} \colonequals \log \ch \Bigl(h_s + \sum_{1\le l \le r}z_{l}(Q_{l}^s-Q_{l-1}^s)^{1/2}\Bigr).
\end{equation} 
Note that $X_0^s$ is a $h_s$-measurable random variable.
The Parisi formula for the MSK model with (random) external field and positive-semidefinite $\Delta^2$ describes the limiting free energy as a variational formula,
\begin{equation*} 
    \lim_{N\to\infty} \e F_{N} = \inf_{r, (q_l), (\zeta_l)} \mathscr P ((q_l), (\zeta_l)),
\end{equation*} where the Parisi functional is given by \begin{equation}\label{eq: Parisi functional for MSK}
    \mathscr P((q_l), (\zeta_l)) \colonequals \log 2 +    \sum_{s \in \mathscr S} \lambda_s\e_h X_{0}^{s} - \frac{1}{2}\sum_{l=0}^{r-1}\zeta_l (Q_{l+1}-Q_l).
\end{equation} 
The Parisi functional can be extended to the compact space of probability measures on $[0,1]^{\mathscr S}$ with totally ordered support via  Lipschitz continuity (see Appendix~\ref{sec: Extension of Parisi functional}).
In particular, the minimum of the Parisi functional is attained  by some probability measures which we call Parisi measures. 
As mentioned in the introduction, Chen--Issa--Mourrat \cite[Theorem 1.2]{CIM25} confirm the uniqueness of Parisi measures for positive-definite $\Delta^2$.
For positive-semidefinite $\Delta^2$, multiple minimizers may exist, but uniqueness up to push-forward by $\Delta^2\Lambda$ is guaranteed, which implies that the supports of all minimizers have the same cardinality (see Appendix~\ref{sec: cardinality of supports of Parisi measures}).

Of particular interest is the case when a Parisi measure is supported on a singleton in $[0,1]^{\mathscr S}$, in which case  we say the model exhibits replica symmetric (RS) solution; otherwise, we say the model presents a replica symmetry breaking (RSB)  solution.
Alternatively, without referring to Parisi measures, one may define RS solutions to be the models with the limiting free energy being equal to the minimum of the RS functional which is the Parisi functional in the simplest case of $r=2$, $\zeta_0=0$, and $\zeta_1=1$ in \eqref{eq: zeta sequence},
\begin{equation}
    \mathsf{RS} (q) \colonequals \log 2 + \sum_{s \in \mathscr S} \lambda_s \e \log \ch\bigl(  z \sqrt{ \tau^2_s+ 2 (\Delta^2 \Lambda q)_s}\bigr) +\frac{1}{2}  (\1 -q)^{\intercal} \Lambda\Delta^2\Lambda (\1-q), \quad q\in [0,1]^{\mathscr S}. \label{eq: RS functional}
\end{equation} 
Associated to this RS functional is the following fixed-point equation
\begin{equation}
    q = \e \hth^2 \bigl(z \sqrt{\tau^2 + 2\Delta^2 \Lambda q}\bigr), 
    \quad q \in [0,1]^{\mathscr S}, \label{eq: fixed-point equation in RS}
\end{equation}
where $z$ is a standard Gaussian random variable, and the operations on the right-hand side are applied entrywise.
In Proposition~\ref{prop: uniqueness of fixed-point}, we show that if $\tau^2\ne \0$ then the equation \eqref{eq: fixed-point equation in RS} has a unique solution $q^* \in (0,1)^{\mathscr S}$.
If $\tau^2=\0$ and $\rho(\Delta^2\Lambda)\le 1/2$, then $q^*=\0$ is the unique solution.
If $\tau^2=\0$ and $\rho(\Delta^2\Lambda)> 1/2$, then  there are two solutions $\0$ and $q^*\in(0,1)^{\mathscr S}$.

Let us define a $\mathscr S \times \mathscr S$ diagonal matrix $\Gamma=\Gamma(\Delta^2,\tau^2)$ by 
 \begin{equation}\label{eq: diagonal operator}
    \Gamma_{ss} =  \e\sch^4 (z \sqrt{\tau^2_s+2 (\Delta^2 \Lambda q^*)_s}) \quad  \text{for all } s\in\mathscr S.
\end{equation}
Let $\rho(\cdot)$ be the spectral radius of a square matrix.
The AT surface is defined as the set of parameters $(\Delta^2,\tau^2)$ such that $\rho(\Gamma \Delta^2 \Lambda)=1/2$.
The following is our main result which characterizes the entire regime of RS solutions.


\begin{theorem}\label{thm: main theorem}  
For any irreducible positive-semidefinite $\Delta^2$  and for any $\tau^2 \gee\0 $,
the replica symmetric solution holds, that is, \[\lim_{N\to \infty}\e F_N  =\min_{q\in [0,1]^{\mathscr S}}\mathsf{RS}(q) \] if and only if $(\Delta^2,\tau^2)$ lies inside the AT surface,  $\rho(\Gamma\Delta^2\Lambda)\le 1/2.$ 
\end{theorem}
We remark that for $\tau^2=\0$, the condition $\rho(\Delta^2\Lambda)\le 1/2$ implies $q^*=\0$ and $\Gamma=I$, hence a RS solution which coincides with the annealed free energy.
However, it is an open question whether the converse holds, that is, whether $\tau^2=\0$ and $\rho(\Delta^2\Lambda)>1/2$ imply $\rho(\Gamma\Delta^2\Lambda)>1/2$ and hence a RSB solution; this nontrivial equivalence at zero external field indeed holds for the single-species SK model (see, e.g., \cite{ALR87, Che17, Com96}).

\paragraph{Related work}
In parallel with the Ising line of work, the spherical setting has seen significant recent progress, while rigorous results in the Ising case remain relatively scarce, mainly confined to the high temperature regime \cite{CL25, Liu21, LD24, Wu24}. 
Recent works on spherical models include the establishment of the Parisi formula and simultaneous RSB\footnote{See Appendix~\ref{sec: Proof of simultaneous RSB}.} \cite{BS22_CMP, BS22}, the development of the TAP approach for non-convex interactions \cite{Sub23, Sub23_AoP, Sub25},  results delineating algorithmic capabilities \cite{HS23, HS24, HS25}, fluctuations of the free energy \cite{BL20, CL25_AIHP}, and the complexity of critical points \cite{Kiv23, McK24}.
In the non-convex interaction setting, the limiting free energy is not known in general \cite{Che24, CM25, Mou21}, except in certain spherical cases.
We also note that an analogue of the AT line at zero temperature for the spherical mixed $p$-spin model was established in \cite{CS17}.

\subsection{Overview of the proof}
Under the assumption of  positive-definite $\Delta^2$ and nonzero external field, RSB outside the AT surface has already been achieved by Toninelli's perturbative  argument \cite{Ton02} adapted to the MSK model in \cite{BSS19, DW21}.
We supply a necessary adjustment to extend the result to irreducible positive-semidefinite $\Delta^2$ and $\tau^2\gee \0,$ see Lemma~\ref{lem: RSB refined}.


The main task of this paper is to establish RS solutions inside the AT surface. We now outline the key steps of our approach, focusing on the case when $\Delta^2$ is positive-definite, $\tau^2\sgee \0$, and $\rho(\Gamma \Delta^2\Lambda)<1/2$.

\vspace{0.5em}
\noindent{\bf Step 1. Uniqueness of the minimizer $q^*$ of the RS functional.}
A first routine but highly nontrivial step, as in the case of the SK model, is to analyze the RS functional, in particular the nagging question of uniqueness of its minimizer. 
Partial results in tackling this technical issue have appeared in \cite{ACM21, BSS19, DW21}; see Remark~\ref{rmk: uniqueness literature} below.  They are usually based on the Latala--Guerra argument or on the contraction argument that do not appear to work in the entire  regime we consider here.
 We resolve this problem by using two ingredients: the irreducibility of $\Delta^2$ and the concavity of the fixed-point map in \eqref{eq: fixed-point equation in RS}, which  critically relies on the assumption of the external field being centered Gaussian.\footnote{Unfortunately, concavity does not seem to hold for nonzero deterministic external fields.} In Lemma~\ref{lem: vector field property on the boundary} below, we show  that the irreducibility of $\Delta^2$ rules out the possibility that the minimizers of the RS functional occur on the boundary of $[0,1]^{\mathscr S}$,  except possibly at $\0$. Consequently, for irreducible positive-definite $\Delta^2$, any minimizer satisfies the fixed-point equation, which together with the concavity of the fixed-point map ensures the uniqueness of minimizers (see Proposition~\ref{prop: infimum of RS functional for positive-definite} below).


\vspace{0.5em}
\noindent{\bf Step 2. Setting up the interpolation.} The next major step is motivated by Talagrand's treatment \cite[Chapter 13]{Talagrand2013vol2} to show the replica symmetric solution in the SK model with deterministic external field. We begin by considering the interpolated Hamiltonian $H_{N,t}(\sigma)$ defined in \eqref{eq: interpolating Hamiltonian} associated to the variance profile \begin{equation*}
    ( \,t\Delta^2, \, \tau^2+(1-t)2\Delta^2\Lambda q^*\, )\quad \text{for}\quad  t\in [0,1].
\end{equation*}
Denote the corresponding free energy by $\e F_N(t)$, the Gibbs measure by $G_{N,t}(\sigma) \propto \exp(H_{N,t}(\sigma)),$ the Gibbs average by $\lla\cdot \rra_{N,t}$, and the RS functional evaluated at $q^*$ by $\mathsf{RS}^*(t)$ which is in \eqref{eq: interpolated approximating RS functional}. 
Noting that $H_{N,t}=H_N$ when $t=1,$ the replica symmetric solution holds at $(\Delta^2,\tau^2)$ if  $\mathsf{RS}^*(1)-\e F_N(1)\approx 0$. 
A standard computation shows that $\mathsf{RS}^*(0)-\e F_N(0)\approx 0$  and   $ \frac{d}{dt} ( \mathsf{RS}^*(t) -\e F_N(t))  \approx  2^{-1} \e \lla \| R_{12}-q^*\|^2\rra_{N, t}$,  where the norm is given by $\|v\|\colonequals \sqrt{v^\intercal \Lambda \Delta^2\Lambda v}$ for $v\in \mathbb R^{\mathscr S}$.
As a result,
\[ \mathsf{RS}^*(1)- \e F_N (1) \approx \frac{1}{2}\int_0^{1} \e \lla \| R_{12}-q^*\|^2\rra_{N,t} \, dt,  \]
and we shall see $\mathsf{RS}^*(1)\approx \e F_N (1)$ if  we can show that the quantity $\e \lla \| R_{12}-q^*\|^2\rra_{N,t}$ is small along the interpolation on the interval $[0,1]$. 
The idea to establish this limit is to show 
\begin{equation}\label{eq: overlap upper bound}
    \e \lla \| R_{12}-q^*\|^2\rra_{N,t} \lesssim \mathsf{RS}^*(t)- \e F_N (t), \quad \forall t\in [0,1],
\end{equation} which implies \[\frac{d}{dt}(\mathsf{RS}^*(t)- \e F_N (t)) \lesssim \mathsf{RS}^*(t)- \e F_N (t), \quad \forall t\in [0,1].\] 
The desired replica symmetric solution then follows by Gr\"onwall's inequality,
  \[0\le \mathsf{RS}^*(1) -\e F_{N} (1)  \lesssim  \mathsf{RS}^*(0) - \e F_{N} (0)  \approx 0.\] 

We emphasize  that while the above approach was implemented in \cite[Chapter 13]{Talagrand2013vol2} for the SK model with deterministic external field, the RS solution was established only on a certain regime of the temperature and external field parameters.
In particular, it does not validate the RS solution up to the AT line. Under our setting, we show that the control of \eqref{eq: overlap upper bound} is indeed  achievable up to the entire AT surface, relying on two additional steps below.

\vspace{0.5em}
\noindent{\bf Step 3. Free energy cost via Guerra--Talagrand 1-RSB bound.} 
Toward \eqref{eq: overlap upper bound}, note that it suffices to  control the Gibbs measure  \begin{equation}\label{eq: control Gibbs measure: overview}
    \e G_{N,t}^{\otimes 2} \bigl(\|R_{12}-q^*\|^2 \gtrsim   \mathsf{RS}^*(t) - \e F_N (t)\bigr) \approx 0.
\end{equation}
To this end, we first establish a multi-species extension of the Guerra--Talagrand (GT) 1-RSB bound in Lemma \ref{lem: constrained 1RSB bound}.
We then utilize this bound to show that by giving an $\eps>0$ room and denoting  $\delta(t) \colonequals C( \mathsf{RS}^*(t) - \lim_{N\to\infty}\e F_N (t)+ 2\eps )>0$ for some constant $C>0$ independent of $t$,  if $\rho(\Gamma \Delta^2 \Lambda)<1/2$, 
for large enough $N$,
\begin{equation}\label{eq: free energy cost: overview}
    \frac{1}{N} \e \log \sum_{\|R_{12}-q^*\| \ge \sqrt{\delta(t)},\, R_{12}\in \mathsf{T}}\exp \bigl (H_{N,t}(\sigma^1)+H_{N,t}(\sigma^2)\bigr)\leq 2 \e F_N(t)-\eps,
\end{equation} where $\mathsf{T}\colonequals \{q\in [0,\infty)^{\mathscr S}: q\lee q^* \; \text{or} \; q\gee q^*\}$.\footnote{We in fact establish  $N^{-1}\e \log \sum_{R_{12}=u}\exp(H_{N,t}(\sigma^1)+H_{N,t}(\sigma^2))\le 2\e F_N(t)-2\eps$ for any $u\in  \mathsf{T}$ with $\|u-q^*\|\ge  \sqrt{\delta(t)}$. Since $R_{12}$ can take at most $(2N+1)^{|\mathscr S|}$ many values due to Ising spins, this implies \eqref{eq: free energy cost: overview}.}
In view of Talagrand's positivity principle, we may add a harmless perturbation of a smaller order to the Hamiltonian and assume that any asymptotic overlap vector has non-negative entries.

Ideally, removing the constraint $R_{12}\in\mathsf{T}$ in \eqref{eq: free energy cost: overview} would then lead to \eqref{eq: control Gibbs measure: overview} by Gaussian concentration, thereby yielding \eqref{eq: overlap upper bound} as desired.
Unfortunately, such constraint is unavoidable because, roughly speaking, we can only establish the multi-species GT 1-RSB bound for totally ordered $q_1\lee q_2$; see Lemma~\ref{lem: constrained 1RSB bound}. 
In contrast, this monotone restriction in the GT bounds for the single species case does not pose a difficulty as there is only one direction in $[0,\infty)$, that is, $\mathsf{T}=[0,\infty)$, and the constraint is vacuous.

\vspace{0.5em}
\noindent{\bf Step 4. Aizenman--Sims--Starr scheme and a key property of Parisi measures.} To circumvent the difficulty sketched above, we first approximate the ratios of species  $(\lambda_s)_{s\in\mathscr S}$ by rational numbers $(k_s/k)_{s\in\mathscr S}$, where the integer $k$ can be interpreted as the level of approximation.
We then consider a modified Hamiltonian $\hat H_{N,t}^{\mathrm{pert}}$ in \eqref{eq: hat pert interpolating Hamiltonian} that differs from $H_{N,t}$ in distribution by a vanishing perturbative term as $N\to\infty.$ 
Denote by $\hat G_{N,t}^{\mathrm{pert}}$ the associated Gibbs measure. 
Since this perturbation is of a smaller order, it is standard to check that \eqref{eq: free energy cost: overview} corresponding to $\hat H_{N,t}^{\mathrm{pert}}$ remains valid so that by using the Gaussian concentration for the free energies, we obtain a  weaker variant of \eqref{eq: control Gibbs measure: overview}, \[ \e (\hat G_{N,t}^{\mathrm{pert}})^{\otimes 2} (\|R_{12}-q^*\| \ge \sqrt{\delta(t)},\, R_{12}\in \mathsf{T}) \approx 0. \]

The idea to remove the constraint $R_{12}\in \mathsf{T}$ is sketched as follows. 
By employing the Aizenman--Sims--Starr (ASS) scheme for the MSK model,  we may assume that any subsequential weak limit of $(\e (\hat G_{Nk,t}^{\mathrm{pert}})^{\otimes 2}(R_{12}\in\cdot ))_{N\geq 1}$ is a Parisi measure. 
Let $\mu$ denote one  such subsequential limit, which is known to be totally ordered \cite[Theorem 4]{Pan15}.
The preceding display implies \[\operatorname{supp}(\mu) \cap \operatorname{int}\bigl(  \mathsf{T}\setminus \mathsf{B}(q^*,\sqrt{\delta(t)}) \bigr)=\emptyset,\] where $\mathsf{B}(x,r)$ is an open  ball of radius $r$  centered at $x$ under the norm $\|\cdot\|$, and $\operatorname{int}(\cdot)$ denotes the interior of a set.
Now, we make a crucial observation (see
 Proposition~\ref{prop: min and max of support} below) that as $\tau^2\sgee\0$, \[ \{q^\mu_{\min},q^\mu_{\max}\}\subseteq \{q^*\} \cup \operatorname{int}( \mathsf{T}),\]
where  $q^\mu_{\min}$ and $q^\mu_{\max}$ are the smallest and largest points of $\operatorname{supp}(\mu)$, respectively.
The last two displays together imply $ \{q^\mu_{\min},  q^\mu_{\max}\} \subseteq \{q^*\}\cup \bigl(\operatorname{int}(\mathsf{T})\cap \operatorname{cl}(\mathsf{B}(q^*,\sqrt{\delta(t)}))\bigr)$, where $\operatorname{cl}(\cdot)$ is the closure of a set.
It might seem that this would readily imply $\mbox{supp}(\mu)\subseteq \operatorname{cl}(\mathsf{B}(q^*,\sqrt{\delta(t)})),$ but it is not necessarily the case. However, by enlarging the ball, we still have the following containment
\[\operatorname{supp}(\mu) \subseteq \mathsf{B}(q^*,C'\sqrt{\delta(t)})\] for some constant $C'>1$ that only depends on the geometry of the ball induced by the norm $\|\cdot \|$; see Figure~\ref{fig: two ellipses: overview} below for an illustration.
Since $\mu$ was arbitrary,  we can rewrite  this as \[ \e (\hat G_{Nk,t}^{\mathrm{pert}})^{\otimes 2} \bigl(\, \|R_{12}-q^*\| \ge C'\sqrt{\delta(t)}\,\bigr) \approx 0,\]
which is exactly \eqref{eq: control Gibbs measure: overview} for the Hamiltonian $\hat H_{Nk,t}^{\mathrm{pert}}$.
Now, the same argument in Step 2 starting from \eqref{eq: overlap upper bound} can still be carried through, and the  fact that the associated limiting free energy for the Hamiltonian  $\hat H_{Nk,t}^{\mathrm{pert}}$ is the same as the original Hamiltonian $H_{N,t}$ completes the proof.

The boundary cases where $\Delta^2$ is only positive-semidefinite, where $\rho(\Gamma \Delta^2\Lambda)=1/2$, and where $\tau_s^2=0$ for some $s\in\mathscr S$ can each be obtained by approximation from the case established above.
This requires extra effort in understanding certain continuity and monotonicity properties of $q^*$; see, e.g., Lemma~\ref{lem: q^* monotone and continuous} and Steps 2 and 3 in the proof of Theorem~\ref{thm: main theorem}.
We remark that many proofs will greatly simplify if we employ a stronger assumption  $\Delta^2\sgee \0$, in lieu of mere irreducibility.

\begin{figure}[h]
\centering
\begin{tikzpicture}[scale=1,>=stealth]

  \pgfmathsetmacro{\thetaDeg}{-15}
  \pgfmathsetmacro{\A}{2}
  \pgfmathsetmacro{\B}{sqrt(2)}

  \pgfmathsetmacro{\cosa}{cos(\thetaDeg)}
  \pgfmathsetmacro{\sina}{sin(\thetaDeg)}

  \pgfmathsetmacro{\Rx}{sqrt((\A*\cosa)^2 + (\B*\sina)^2)}
  \pgfmathsetmacro{\Ry}{sqrt((\A*\sina)^2 + (\B*\cosa)^2)}



  \pgfmathsetmacro{\scaleS}{1.556}

  \draw[gray!60] (-3,0)--(3,0);
  \draw[gray!60] (0,-2.22)--(0,2.22);

\begin{scope}
  \clip[rotate around={\thetaDeg:(0,0)}]
        (0,0) ellipse [x radius=\scaleS*\A, y radius=\scaleS*\B];

  \fill[cyan!20,opacity=.25] (0,0) rectangle (100,100);
  \fill[cyan!20,opacity=.25] (-100,-100) rectangle (0,0);
\end{scope}

\node[cyan!80!black] at (1.2*\Rx, 0.3*\Ry) {$\mathsf{T}$};
\node[cyan!80!black] at (-1.2*\Rx, -0.3*\Ry) {$\mathsf{T}$};


  \begin{scope}[rotate around={\thetaDeg:(0,0)}]
    \draw[line width=1.0pt] (0,0) ellipse [x radius=\A, y radius=\B];
    \draw[line width=1pt] (0,0) ellipse [x radius=\scaleS*\A, y radius=\scaleS*\B];
  \end{scope}

  \draw[dashed, line width=0.9pt] (-\Rx,-\Ry) rectangle (\Rx,\Ry);


\filldraw[black] (0,0) circle (1.2pt) node[below right] {$q^*$};
\filldraw[blue] (0.17,1.3) circle (1.5pt) node[below left] {$q^{\mu}_{\max}$};
\filldraw[blue] (-1.7,-0.25) circle (1.5pt) node[right] {$q^{\mu}_{\min}$};

  \filldraw[blue] (-1.6,1.2) circle (1.5pt) node[above left] {};

  \draw[->,thick,red,dashed] (0,0) -- ({2.5*cos(75)},{2.5*sin(75)});
\end{tikzpicture}
\caption{Illustration of Step 4 in two-species. Three blue dots indicate $\operatorname{supp}(\mu)$ which is totally ordered with respect to the partial order $\lee$. The smaller ellipse depicts $\mathsf{B}(q^*,\sqrt{\delta(t)})$ and  the larger ellipse depicts $\mathsf{B}(q^*, C' \sqrt{\delta(t)})$. 
The first and the third quadrants shaded in light blue depicts $\mathsf{T}= \{q\in [0,1]^{\mathscr S}: q\lee q^* \; \text{or} \; q\gee q^*\}$.
Although $ \{q^\mu_{\min},  q^\mu_{\max}\} \subseteq \{q^*\}\cup  \bigl(\operatorname{int}(\mathsf{T})\cap \operatorname{cl}(\mathsf{B}(q^*,\sqrt{\delta(t)}))\bigr)$, there remains  a blue dot  that lies outside $\operatorname{cl}(\mathsf{B}(q^*,\sqrt{\delta(t)}))$, which necessitates enlarging the ellipse to ensure containment. 
The ellipses are always elongated in this way, since the principal eigenvector (red dashed arrow) of $\Lambda\Delta^2\Lambda$ has positive entries by the Perron--Frobenius theorem.
}

\label{fig: two ellipses: overview}
\end{figure}
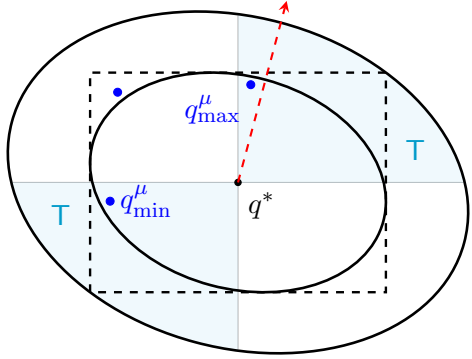

\subsection{Organization of the paper}
In Section~\ref{sec: Minimizers of replica symmetric functional}, we investigate the RS functional.
Section~\ref{sec: Support of Parisi measures} is devoted to confirming the aforementioned key property of the support of  Parisi measures.
In Section~\ref{sec: interpolation and  Aizenman-Sims-Starr cavity computation}, we make a connection between the asymptotic overlap and a Parisi measure  through the ASS scheme, and identify  the coupled free energy cost as a sufficient condition for RS solutions.
In turn, this cost  is established strictly inside the AT surface in Section~\ref{sec: Two-dimensional Guerra--Talagrand bound}, assuming the multi-species 1-RSB GT bound.
We carry out the proof of the main theorem in Section~\ref{sec: Proof of main theorem}.
Appendix~\ref{sec: Extension of Parisi functional}, in which the results also hold for deterministic external fields, provides the  directional derivative of the Parisi functional and derives an equation Parisi measures must satisfy.
Appendix~\ref{sec: Proof of continuity of Parisi formula in model parameters} establishes some continuity properties of the Parisi formula in its external parameters $(\Lambda,\Delta^2,\tau^2)$, which is used throughout the article.
Appendix~\ref{sec: cardinality of supports of Parisi measures} proves that the support cardinalities of Parisi measures are all equal.
In Appendix~\ref{sec: Proof of simultaneous RSB}, we confirm simultaneous RSB for irreducible, positive-definite $\Delta^2$ with deterministic external fields.


\begin{ack}
    \textrm{The author thanks his advisors, Wei-Kuo Chen and Arnab Sen, for suggesting this project and for many helpful discussions and comments. The author also acknowledges Qiang Wu for clarifying explanations regarding \cite{DW21}. This work was partially supported by NSF grant DMS-2246715.}
\end{ack}

\section{Minimizers of the replica symmetric functional}\label{sec: Minimizers of replica symmetric functional}
This section is devoted to establishing  various properties of the minimizers of the RS functional.
The  following main result of this section is a  description of the minimizers of the RS functional for irreducible positive-definite $\Delta^2.$
\begin{proposition}\label{prop: infimum of RS functional for positive-definite}
    Assume $\Delta^2$ is irreducible and positive-definite.
    The infimum $\inf_{q\in [0,1]^{\mathscr S}}\mathsf{RS}(q)$ is attained by a unique minimizer.
    This minimizer is a solution to the fixed-point equation \eqref{eq: fixed-point equation in RS}.
If $\tau^2=\0$ and $\rho(\Delta^2\Lambda)\le 1/2$, then $\0$ is the  minimizer.
Otherwise, the minimizer is in the interior $(0,1)^{\mathscr S}$.
\end{proposition}

To achieve this, we first show that  any minimizer must be a solution of the fixed-point equation; see Lemma~\ref{lem: infimum of RS}.
In turn,  we can study the number and locations of the  fixed-points as follows.

\begin{proposition}\label{prop: uniqueness of fixed-point}
    For an irreducible $\Delta^2$, the following holds.
    \begin{enumerate}[label=(\roman*)]
        \item If $\tau^2\ne \0$, then there is a unique fixed-point of \eqref{eq: fixed-point equation in RS}, and it is in  $(0,1)^{\mathscr S}$.
        \item Suppose $\tau^2=\0$.
    If $\rho(\Delta^2\Lambda )>1/2$, there are exactly two fixed-points of \eqref{eq: fixed-point equation in RS}; one is $\0$ and the other is in $(0,1)^{\mathscr S}$.
    If $\rho(\Delta^2\Lambda)\le 1/2$, then $\0$ is the only fixed-point.
    
    \end{enumerate}  
    
\end{proposition}
    \begin{remark}\label{rmk: uniqueness literature}
        The uniqueness of the fixed-point in the high temperature  regime $\rho(\Delta^2\Lambda)<1/2$ for positive-definite $\Delta^2$ and nonzero deterministic external field was established in \cite[Theorem 1.4]{DW21} via Banach fixed-point theorem,  and here we strengthen their argument to cover the case $\rho(\Delta^2\Lambda)\le 1/2$, along with a   complementary result for $\rho(\Delta^2\Lambda)>1/2.$
    Earlier work in two-species case \cite{BSS19} employs an elementary argument together with the standard Latala--Guerra lemma \cite[Proposition A.14.1]{Talagrand2013vol1}, which leads to the uniqueness for all nonzero deterministic external field.
    In another direction, \cite[Theorem 3]{ACM21} verifies the uniqueness for a special indefinite, irreducible $\Delta^2$ and any $\tau^2\sgee \0$, which again enables the Latala--Guerra lemma.
    \end{remark}

As there are at most two fixed-points which are comparable by the partial order $\slee$, the notion of a maximal fixed-point is well-defined, and it indeed coincides with  the minimizer  in Proposition~\ref{prop: infimum of RS functional for positive-definite}.
\begin{definition}\label{def: definition of q^* for irreducible}
For  $\tau^2\gee \0$ and an irreducible $\Delta^2$, we denote the maximal fixed-point of \eqref{eq: fixed-point equation in RS} by \begin{equation*}
    q^*=q^*(\Delta^2,\tau^2).
\end{equation*} 
\end{definition}

In Subsection~\ref{sec: further properties of fixed-points}, we collect additional properties of the fixed-point  for later use.
Lemma~\ref{lem: q^* monotone and continuous} establishes  certain continuity and monotonicity of the fixed-point  in the variance profile $(\Delta^2,\tau^2)$, which we later use when approximating the boundary cases in the proof of the main theorem.
Lemma~\ref{lem: monotone region lemma} proves a property about the vector field associated to the fixed-point map, which is crucially used in the proof of Proposition~\ref{prop: min and max of support}.
Finally, in Lemma~\ref{lem: infimum of RS: semidefinite}, we investigate minimizers for positive-semidefinite $\Delta^2$.

\subsection{Proof of Proposition~\ref{prop: uniqueness of fixed-point}}

Our approach to Proposition~\ref{prop: uniqueness of fixed-point} relies on an elementary concavity property  of the fixed-point map.

\begin{lemma}\label{lem: concavity of fixed-point equation}
    The function $x\mapsto \hth^2 (a \sqrt{b+x})$ for $a> 0$ and $b\ge 0$ is strictly concave on the domain $(0,\infty).$
\end{lemma}
\begin{proof}
    It suffices to check the concavity of  $g(x)\colonequals \hth^2 ( \sqrt{x})$.
    Taking derivatives for $x>0$, we have \begin{align*}
        g'(x)&= \hth (\sqrt{x}) \sch^2(\sqrt{x}) \cdot  x^{-1/2},
        \\ 2g''(x) &= \sch^2(\sqrt{x}) (1- 3\hth^2 (\sqrt{x})) \cdot x^{-1} - \hth (\sqrt{x}) \sch^2(\sqrt{x}) \cdot x^{-3/2}
        \\&= x^{-3/2} \sch^2(\sqrt{x}) \bigl( \sqrt{x}- 3\sqrt{x}\hth^2 (\sqrt{x}) -\hth(\sqrt{x}) \bigr).
    \end{align*} 
    In view of the last equation, the proof would be complete if we can show the function \[k(x)\colonequals x-3x\hth^2(x) -\hth(x), \quad x\ge 0,\] is negative for $x>0$.
    Note that $k(0)=0.$
    An elementary calculation shows that \[k'(x)=-2\hth^2(x) -6x\hth(x)\sch^2(x) <0  \quad\text{for} \quad x>0,\] which implies $k(x)<0$ for $x>0$, as desired.
\end{proof}


    

\begin{proof}[\bf Proof of Proposition~\ref{prop: uniqueness of fixed-point}]
    Denote the map on the right-hand side of \eqref{eq: fixed-point equation in RS} by \begin{equation}\label{eq: contraction mapping}
        F(p) \colonequals      \e\hth^2 \bigl( |z|\sqrt{ \tau^2+ 2\Delta^2\Lambda p}\bigr), \quad  p\in [0,1]^{\mathscr S},
    \end{equation} where $|\cdot|$ denotes entrywise modulus.
    Let us denote the $n$-fold composition of $F$ by $F^n$ for any $n\in\mathbb N.$

    Let $q_0\equiv \0$ and $q_1\equiv \1$. Note that $F^n(q_0)$ is monotonically increasing in the sense that $F^n(q_0)\lee  F^{n+1}(q_0).$ To see this, note that obviously $q_0\lee  F(q_0).$ If $F^{n-1}(q_0)\lee  F^n(q_0),$ then $F^n(q_0)\lee  F^{n+1}(q_0)$ since $\tanh^2(x)$ is strictly increasing in $x>0.$ Similarly, we also see that $F^n(q_1)$ is monotonically decreasing. 
    Thus, from Dini's theorem, $F^n(q_0)\to \hat q_0$ and $F^n(q_1)\to \hat q_1$ uniformly over different species for some vectors $\hat q_0$ and $\hat q_1.$ Consequently, $\hat q_0$ and $\hat q_1$ are fixed-points of $F$ and $\hat q_0\lee  \hat q_1.$  

\textit{Claim 1.} If $q_*$ is an arbitrary fixed-point of $F$, then $\hat q_0\lee  q_*\lee  \hat q_1.$ In fact, this holds by noting that $q_0\lee  q_*\lee  q_1$ and again $\tanh^2(x)$ is strictly increasing in $x>0.$ From these, we have $F^n(q_0)\lee  F^n(q_*)=q_*\lee  F^n(q_1)$ and sending $n\to\infty$ validates our claim.

\textit{Claim 2.}  For any $q\in[0,1]^{\mathscr S}$ and $n\ge 0$, it holds that \[\operatorname{supp}(F^{n}(q)) \supseteq \operatorname{supp}( (\Delta^2\Lambda)^{n} q).\]
    This is trivial for $n=0$. 
    Suppose the claim holds for some $n\ge 0$.
    Let $s\in\mathscr S$.
    Suppose $ F^{n+1}(q)_s= \e \hth^2(z\sqrt{(\tau^2)_s+2(\Delta^2 \Lambda F^n(q))_s})=0$.
    Then $\sum_{k\in\mathscr S}(\Delta^2 \Lambda)_{sk} F^n(q)_k=0$ so that $F^n(q)_k=0$ whenever $(\Delta^2\Lambda)_{sk}>0$.
    By the induction hypothesis, this implies $((\Delta^2\Lambda)^nq )_k=0$ whenever $(\Delta^2\Lambda)_{sk}>0$.
    As a result, $((\Delta^2\Lambda)^{n+1}q)_s= \sum_{k\in\mathscr S}(\Delta^2\Lambda)_{sk} ((\Delta^2\Lambda)^n q)_k=0.$
    This completes the induction and the proof of the claim.

    From hereafter, we consider the following two cases separately.
    
    \noindent\textbf{Case 1. $\tau^2\ne \0.$}
    
    \noindent From Claim 1, the uniqueness holds if and only if $\hat q_0 = \hat q_1.$      
    Toward a contradiction, suppose otherwise, i.e.,  $\hat q_0 \ne \hat q_1.$
    Consider the line connecting $\hat q_0$ and $\hat q_1$, \begin{equation}\label{eq: q line}
        q(t)\colonequals  \hat q_0 + t(\hat q_1 -\hat q_0) , \quad t\in\mathbb R,
    \end{equation}
    and define \[ t_*= \sup \{t\le 0: q(t)_s=0 \text{ for some } s\in \mathscr S \}.\] 

    We verify the following properties of $t_*$: 
    \begin{enumerate}[label=(\alph*), itemsep=-0.2em]
        \item  $t_*$ is well-defined;
        \item  $t_*<0$; 
        \item  $q(t_*)\gee \0$;
        \item  $F(q(t))$ is well-defined for $t\ge t_*$.
    \end{enumerate}
    Note $\hat q_0 \ne \hat q_1$ and $\hat q_0\lee  \hat q_1$ imply $(\hat q_0)_s< ( \hat q_1)_s$ for some $s \in \mathscr S$.
    For this species $s$, we have $q(t)_s=0$ for $t=- (\hat q_0)_s/((\hat q_1)_s - (\hat q_0)_s)\le 0$ and the set appearing in the definition of $t_*$ is non-empty.
    This verifies item (a).
    Next, note that  $\tau^2\ne \0$ and, hence, $v\colonequals F(q_0)=\e\hth^2 h\ne \0$.
The irreducibility of $\Delta^2$ ensures that for any $s\in\mathscr S$, there exists some $m_s\in\mathbb N$ such that  $((\Delta^2\Lambda)^{m_s} v)_s>0$.
    Combining this with Claim 2 and the previously established monotonicity, we conclude
    \begin{equation}\label{eq: q_0>0}
        (\hat q_0)_s \ge F^{m_s}(v)_s>0 \quad \text{for all } s\in\mathscr S,
    \end{equation} which readily implies item (b).
     To show item (c), Suppose $q(t_*)_s <0$ for some $s\in \mathscr S$. 
     From \eqref{eq: q line} and \eqref{eq: q_0>0}, we must have $(\hat q_1)_s -(\hat q_0)_s>0$ and, by the intermediate value theorem, there exists $t>t_*$ such that $q(t)_s=0.$ This contradicts the definition of $t_*$ as a supremum and, therefore, $q(t_*)\gee \0.$
     Note that item (c) and $\hat q_0 \lee  \hat q_1$ imply  $\0\lee  q(t_*)\lee  q(t)$ for all $t\ge t_*$.
     Since $\Delta^2 \Lambda$ preserves nonnegativity, $\Delta^2 \Lambda q(t) \gee \0$ and $F(q(t))=\e\hth^2(|z|\sqrt{\tau^2 +2\Delta^2\Lambda q(t)})$ is well-defined for all $t\ge t_*$.
     This checks item (d).
    All the items above are demonstrated.

    By Lemma~\ref{lem: concavity of fixed-point equation},  the map $t\mapsto F(q(t))_s$ is strictly concave  for each $s \in \mathscr S$.
    Here, we used the fact that, since $\Delta^2$ is irreducible and $\hat q_0\sgee \0$ by \eqref{eq: q_0>0}, it must hold $\Delta^2\Lambda \hat q_0 \sgee \0$ and, therefore, for all $s\in\mathscr S$, $t\mapsto F(q(t))_s$ is not identically zero.
    Moreover, $F(q(i))= q(i)$ for $i\in \{0,1\}$.
    These imply the following dichotomy: \begin{align*}
      &F(q(t))\sgee q(t) \quad \text{if} \quad t\in (0,1),
      \\&F(q(t))\slee q(t) \quad \text{if} \quad t\in [t_*,0)\cup (1,\infty).
    \end{align*}
    Note that we had to restrict our domain to $[t_*,\infty)$, in view of property (d) above.
     On the other hand, there exists some $s_*\in \mathscr S$  such that $q(t_*)_{s_*}=0$ and, therefore,  $F(q(t_*))_{s_*}< q(t_*)_{s_*}=0.$ 
    This is a contradiction to $F(q(t_*))\gee \0$, completing the proof of the uniqueness.

    Note that \eqref{eq: q_0>0} and the fact $\hth^2<1$ imply that the unique fixed-point lies inside the interior $(0,1)^{\mathscr S}$.

    \noindent\textbf{Case 2. $\tau^2=\0$.}

    \noindent In this case, $\0$ is a trivial fixed-point, which can be identified with $\hat q_0$.  
    Assume $\rho(\Delta^2\Lambda)\le 1/2$.
    We will show  $\hat q_1 =\0$, in which case $\0$ is the unique fixed-point.
    Toward a contradiction, suppose $\hat q_1 \ne \0$.
    By \eqref{eq: nonzero vector field on the boundary}, $\hat q_1 \in (0,1)^{\mathscr S}$.
    Consider the linear interpolation $q_t\colonequals t\hat q_1$ for $t\in[0,1]$.
    Note that $F(q_t)$ is well-defined since $[0,1]^{\mathscr S}$ is convex.
    In spirit of \cite{DW21}, consider $\Lambda^{1/2} F(q_t).$
    For $0<t<1,$  we can differentiate using Gaussian integration by parts, \begin{equation*}
        \frac{d}{dt}\Lambda^{1/2}F(q_t) = \Lambda^{1/2}D_t\Delta^2\Lambda \hat q_1,
    \end{equation*}  where $D_t \colonequals \text{diag}(\e(\hth^2)''(z \sqrt{2(\Delta^2\Lambda q_t)_s}): s\in\mathscr S)$.
    Here in the calculation of Gaussian IBP, we used the fact that $\Delta^2$ is irreducible and,  for $0<t<1$, $q_t\sgee \0$, hence $\Delta^2\Lambda q_t\sgee \0.$
    By the fundamental theorem of calculus, since $F(q_0)=\0$ and  $|(\hth^2)''(x)|< 2$ Lebesgue a.e., \begin{align*}
        \|\Lambda^{1/2} F(\hat q_1)\|_2\le  \bigl\| \Lambda^{1/2} \int_0^1  |D_t|  dt \, \Delta^2\Lambda \hat q_1 \bigr\|_2< 2 \| \Lambda^{1/2} \Delta^2 \Lambda \hat q_1 \|_2 \le 2 \rho(\Delta^2\Lambda) \|\Lambda^{1/2} \hat q_1\|_2\le \|\Lambda^{1/2} \hat q_1\|_2,
    \end{align*} where  $|\cdot |$ and integration are done entrywise for $D_t.$
    In the second-to-last inequality of the preceding display, we used the fact that $\|\Lambda^{1/2} \Delta^2 \Lambda^{1/2}\|_2 =\rho(\Lambda^{1/2} \Delta^2 \Lambda^{1/2})=\rho(\Delta^2 \Lambda)$.
    Since $F(\hat q_1)=\hat q_1$, the preceding display leads to a contradiction.
    Consequently, we must have $\hat q_1=\0$, as desired.

    Assume the complementary case $\rho(\Delta^2\Lambda)>1/2.$
    Let $\eps>0$ be such that $(1-\eps)^2 2\rho(\Delta^2\Lambda)> 1$.
    We can choose a small $\delta>0$ such that $\hth^2 (x) \ge (1-\eps)x^2$ for all $|x|<\delta$.
    Let $u\sgee \0$ be the Perron vector of the irreducible $\Delta^2\Lambda$ (see \cite[(8.3.9)]{Mey00_book}).
    Let us also choose a small enough $t=t(\eps,\delta,u)>0$  such that \[ \p \bigl(  \sqrt{2\rho(\Delta^2\Lambda) t u} \, |z| \slee \delta \1 \bigr)>1-\eps, \] where $|\cdot|$ is the entrywise modulus of the standard Gaussian vector $z.$
    Then, \begin{align*}
        F(tu)=\e \hth^2(z \sqrt{ 2\rho(\Delta^2\Lambda)t u})\gee (1-\eps)^2 2\rho (\Delta^2\Lambda) t u \sgee tu ,
    \end{align*}
    which implies by induction that $\0\slee F^n(tu)\lee F^{n+1}(tu) \slee \1$ for all $n\in\mathbb N$. 
    It must converge to some fixed-point $\hat u\colonequals \lim_{n\to\infty} F^n(tu)\sgee \0$ by monotonicity.
    By Claim 1, it holds that $\0 \slee \hat u \lee \hat q_1$.
    We have shown that  $\hat q_1$ is non-trivial.
    Now, we will show that $\{\0, \hat q_1\}$ are all the fixed-points.
    Toward a contradiction, suppose there exists another fixed-point $p \notin \{\0,\hat q_1\}$.
    From Claim 1, we must have $\0\lee p\lee \hat q_1$.
    Moreover, by \eqref{eq: nonzero vector field on the boundary},  $p, \hat q_1\in (0,1)^{\mathscr S}$. 
    Then,  the line connecting $p$ and $\hat q_1$ lies strictly inside the interior $(0,1)^{\mathscr S}$, which allows us to carry out the argument leading to a contradiction as in Case 1 with one modification; item (b) simply follows from $p \sgee \0$ in lieu of \eqref{eq: q_0>0}.
    Consequently, $\hat  q_1$ is the only fixed-point other than $\0$, and it is inside the interior $(0,1)^{\mathscr S}$.
\end{proof}

\subsection{Proof of Proposition~\ref{prop: infimum of RS functional for positive-definite}}

We start with ruling out fixed-points on the boundary of the cube $[0,1]^{\mathscr S}$.
\begin{lemma}\label{lem: vector field property on the boundary}
    Assume irreducible $\Delta^2$.
    For $q\ne \0$ on the boundary of $[0,1]^{\mathscr S}$, we have \begin{align}
        \e \hth^2\bigl(z \sqrt{\tau^2+2 \Delta^2 \Lambda q}\bigr) -q &\ne \0. \label{eq: nonzero vector field on the boundary}
        \end{align}
        For any $q\in [0,1]^{\mathscr S}$, 
        \begin{align}
        q+ \eps \bigl(\e \hth^2\bigl(z \sqrt{\tau^2+2 \Delta^2 \Lambda q}\bigr) -q\bigr) &\in [0,1]^{\mathscr S} \text{ for small enough $\eps>0$.} \label{eq: vector field direction okay}
    \end{align} 
\end{lemma}
\begin{remark}
    The proof below also holds for deterministic external fields.
\end{remark}
\begin{proof}
    Let $q \ne \0$ be on the Euclidean boundary of $[0,1]^{\mathscr S}$.
    Then define \[J=\{s\in\mathscr S: q_s =0\},\quad  K= \{s\in\mathscr S: 0<q_s <1\}, \quad \text{and} \quad L= \{s\in\mathscr S: q_s=1\},\] which partition $\mathscr S.$
    Note that $J\ne \mathscr S$ since $q\ne\0$, and $K\ne \mathscr S$ since $q$ is on the boundary.

We first check \eqref{eq: nonzero vector field on the boundary}.
    If $L\ne \emptyset$, then the \eqref{eq: nonzero vector field on the boundary} is trivial from $\hth^2<1.$
    Suppose $L =\emptyset$.
    Then,  since $J\cup K=\mathscr S$ while $J,K\ne\mathscr S$, we have $J,K\ne \emptyset $.
    We claim that there exists $j\in J$ such that $\e \hth^2(z\sqrt{(\tau^2)_j + 2(\Delta^2\Lambda q)_j}) >0.$
    Suppose otherwise.
    We immediately have $\tau^2_j=0= (\Delta^2\Lambda q)_j$ for all $j\in J.$
    Since $q_k >0$ for all $k\in K$ and  \[0=(\Delta^2\Lambda q)_j=\sum_{s\in\mathscr S} (\Delta^2\Lambda)_{js} q_s= \sum_{s\in K} (\Delta^2\Lambda)_{js} q_s, \quad \forall j\in J, \] we have $\Delta^2_{jk}=0$ for all $(j,k)\in J\times K$.
    This contradicts the assumption that $\Delta^2$ is irreducible, and finishes the proof of the claim.
    Now, the claim readily verifies \eqref{eq: nonzero vector field on the boundary}.

    Next, we check \eqref{eq: vector field direction okay}.
    Fix $q\in [0,1]^{\mathscr S}$ and define $J$, $K$, and $L$ as before.
    Since $\mathscr S$ is a finite set, it suffices to show for each $s\in\mathscr S$, if $\eps>0$ is small enough, then \[ 0\le q_s+ \eps \bigl(\e \hth^2\bigl(z \sqrt{(\tau^2)_s+2 (\Delta^2 \Lambda q)_s}\bigr) -q_s\bigr) \le 1.\]
    For  $s \in J\cup K$, this is obvious.
    For $s\in L$, this follows from the inequality $\hth^2<1$.
    This verifies \eqref{eq: vector field direction okay} and completes the proof.
\end{proof}

For irreducible and positive-definite $\Delta^2$, minimizers of the RS functional indeed satisfy the fixed-point equation.
\begin{lemma}\label{lem: infimum of RS}
    Suppose  $\Delta^2$ is irreducible and positive-definite.
    Any $q\in [0,1]^{\mathscr S}$ achieving the infimum of the RS functional satisfies \eqref{eq: fixed-point equation in RS}.

\end{lemma}
\begin{proof}
Let $\eps>0$, $q \in [0,1]^{\mathscr S}$, and $u(q)=\e \hth^2\bigl(z \sqrt{\tau^2+2 \Delta^2 \Lambda q}\bigr) -q$.
By Lemma~\ref{lem: vector field property on the boundary}, $q+\eps u(q) \in [0,1]^{\mathscr S}$ for small enough $\eps>0$, and the right-derivative $\partial_\eps \mathsf{RS} (q+\eps u(q))|_{\eps=0^+}$ is well-defined.
We can then carry out the differentiation using Gaussian IBP to get \begin{equation}\label{eq: right derivative of RS}
    \frac{d}{d\eps} \mathsf{RS} (q+\eps u(q))\Bigr|_{\eps=0^+} =    - u(q)^\intercal \Lambda\Delta^2\Lambda u(q).
\end{equation}
For any $q\ne \0$ on the boundary of $[0,1]^{\mathscr S}$, since $u(q)\ne \0$ by  Lemma~\ref{lem: vector field property on the boundary},  it holds that \eqref{eq: right derivative of RS} is strictly negative.
Moreover, $q+\eps u(q) \in [0,1]^{\mathscr S}$ for small enough $\eps>0$, so such $q$ cannot be a minimizer of $\mathsf{RS}$.
Therefore, a minimizer must be either $\0$ or in the interior $(0,1)^{\mathscr S}$.
In turn, for a minimizer $q\in \{\0\}\cup (0,1)^{\mathscr S}$,  \eqref{eq: right derivative of RS} and the positive-definiteness of $\Delta^2$ readily imply  $u=\0$ as desired.

\end{proof}

The following rules out $\0$ as a minimizer of the RS functional at zero external field and low temperature.
\begin{lemma}\label{lem: infimum of RS when h=0}
   Assume $\tau^2=\0$ and irreducible $\Delta^2$.
     If $\rho(\Delta^2\Lambda )>1/2$, then $\0$ does not attain the infimum of the RS functional.
\end{lemma}
\begin{proof}
Suppose $\rho(\Delta^2\Lambda)>1/2$.
Let $\eps>0$, $q \in [0,1]^{\mathscr S}$.
By the same calculation yielding \eqref{eq: right derivative of RS},  for any $v\gee \0$, the first and second right-derivatives are given by \begin{align*}
    \frac{d}{d\eps }\mathsf{RS}(\eps v) &= -v^\intercal \Lambda \Delta^2\Lambda \e \hth^2 (\sqrt{2\Delta^2\Lambda \eps v}) +\eps v^\intercal \Lambda \Delta^2\Lambda v,
    \\ \frac{d^2}{d\eps^2}\mathsf{RS}(\eps v) &=  -v^\intercal \Lambda\Delta^2\Lambda D \Delta^2\Lambda v +  v^\intercal \Lambda \Delta^2\Lambda v,
\end{align*} where $D\colonequals \text{diag}(\e (\hth^2)''(\sqrt{2(\Delta^2\Lambda \eps v)_s }): s\in\mathscr S)$.
In particular, for any $v\gee \0$, we have \begin{equation}\label{eq: Gateaux derivative at zero}
    \frac{d}{d\eps} \mathsf{RS}(\eps v)\Bigr|_{\eps=0^+}=0,
\end{equation} and since $(\hth^2)''(0)=2$, \begin{equation*}
     \frac{d^2}{d\eps^2} \mathsf{RS} (\eps v)\Bigr|_{\eps=0^+} = -2v^\intercal \Lambda(\Delta^2\Lambda)^2 v +  v^\intercal \Lambda \Delta^2\Lambda v. 
\end{equation*} 
As $\Delta^2$ is irreducible, so is $\Delta^2\Lambda$, and  the Perron vector $u\sgee \0$ of $\Delta^2\Lambda$ exists, that is, $\Delta^2\Lambda u=\rho(\Delta^2\Lambda) u$.  
Plugging this into the preceding display, since $\rho(\Delta^2\Lambda)>1/2$, \[\frac{d^2}{d\eps^2} \mathsf{RS} (\eps u)\Bigr|_{\eps=0^+}= (-2\rho(\Delta^2\Lambda)^2 +\rho(\Delta^2\Lambda) ) u^\intercal \Lambda u<0, \]
which together with \eqref{eq: Gateaux derivative at zero} show that $\0$ cannot attain the infimum of the RS functional. 
    
\end{proof}



\begin{proof}[\bf Proof of Proposition~\ref{prop: infimum of RS functional for positive-definite}]
    Combining Lemma~\ref{lem: infimum of RS}, Proposition \ref{prop: uniqueness of fixed-point}, and Lemma~\ref{lem: infimum of RS when h=0} completes the proof.
\end{proof}

\subsection{Further properties of fixed-points}\label{sec: further properties of fixed-points}
We begin with a simple calculus, whose proof is omitted.
\begin{lemma}\label{lem: tanh^2 gets flatter}
    For any $a\ge b\ge 0$ and $c> 0$, the function $x\mapsto \hth^2( c\sqrt{a+x})- \hth^2 (c\sqrt{b+x}) $ is non-negative and non-increasing on $(0,\infty).$
\end{lemma}

We establish that $q^*$ is  monotone and continuous in the variance profile when $\Lambda$ is fixed.

\begin{lemma}\label{lem: q^* monotone and continuous}
    For $i=1,2$, let $\Delta^2_i$ be irreducible and $q^*_i= q^*_i(\Delta^2_i, \tau^2_i)$  be given by Definition \ref{def: definition of q^* for irreducible}.
    Then, the following holds.
    \begin{enumerate}[label=(\roman*)]
        \item If $\Delta^2_1\gee \Delta^2_2$ and $\tau^2_1 \gee \tau^2_2$, then  $q^*_1\gee q^*_2$.
        \item  If $\Delta^2_1 \gee\Delta^2_2$ and $(\tau^2_1)_s > (\tau^2_2)_s$ for some $s\in\mathscr S$, then $(q^*_1)_s>(q^*_2)_s$.
    \end{enumerate}
    
    Regard $q^*=q^*(\Delta^2,\tau^2)$ as a map on the domain $\{(\Delta^2,\tau^2): \text{irreducible }\Delta^2\}$ with the subspace topology induced by any matrix norm.
    Then, $q^*=q^*(\Delta^2,\tau^2)$ is continuous at any $(\Delta^2,\tau^2)$ such that $\tau^2\ne \0$.
    Moreover, $q^*=q^*(\Delta^2,\tau^2)$ is continuous from above everywhere on the domain: if $\lim_{n\to\infty}(\Delta^2_n,\tau^2_n)= (\Delta^2,\tau^2)$  while $(\Delta^2_n,\tau^2_n)\gee (\Delta^2,\tau^2)$ for all $n\in\mathbb N$, then $\lim_{n\to\infty}q^*(\Delta^2_n,\tau^2_n)= q^*(\Delta^2,\tau^2).$

\end{lemma}
\begin{proof}
    To validate $(i)$, assume $\Delta^2_1\gee \Delta^2_2$ and $\tau_1^2\gee \tau_2^2$, and denote the corresponding $F$ in \eqref{eq: contraction mapping} by $F_1$ and $F_2.$ 
    For $i=1,2$, in view the proof of Proposition \ref{prop: uniqueness of fixed-point}, we may identify $q_i=\lim_{n\to\infty}F^n_i(\1)$.
    For any $p\gee \tilde p\gee \0$, we have $\Delta^2_1 \Lambda p \gee \Delta^2_2 \Lambda \tilde p \gee \0$ and \[F_1(p) = \e\hth^2(|z|(\tau_1^2 +2\Delta^2_1\Lambda p)^{1/2})  \gee \e\hth^2(|z| (\tau_2^2 +2\Delta^2_2\Lambda \tilde p)^{1/2})=F_2(\tilde p).\] 
    The preceding display implies $\1\gee F^n_1(\1) \gee F^n_2(\1)$ for any $n\in\mathbb N$ and, by sending $n\to\infty$, $q^*_1\gee q^*_2.$
    To check $(ii)$, we note that
    \begin{align*}
        F^{n+1}_1(\1) -F^{n+1}_2(\1) &= \e\hth^2(|z|(\tau_1^2 +2\Delta^2_1\Lambda F^n_1(\1))^{1/2}) -\e\hth^2(|z|(\tau_2^2 +2\Delta^2_2\Lambda F^n_2(\1))^{1/2}) 
        \\&\gee \e\hth^2(|z|(\tau_1^2 +2\Delta^2_1 \Lambda F^n_1(\1))^{1/2}) -\e\hth^2(|z|(\tau_2^2 +2\Delta^2_1\Lambda F^n_1(\1))^{1/2}) 
        \\&\gee \e\hth^2(|z|(\tau_1^2 +2\Delta^2_1 \Lambda \1 )^{1/2}) -\e\hth^2(|z|(\tau_2^2 +2\Delta^2_1 \Lambda \1)^{1/2}),
    \end{align*} where the last inequality is due to Lemma~\ref{lem: tanh^2 gets flatter}.
    The last line of the preceding display is independent of $n$, so sending $n\to\infty$ yields $(ii)$.
    
   Next, we check continuity.
 Fix an irreducible $\Delta^2$.
    Regarded as  $|\mathscr S|^2 +|\mathscr S|$-dimensional Euclidean vectors in the usual Euclidean norm, let $\lim_{n\to\infty} (\Delta^2_n, \tau^2_n) = (\Delta^2,\tau^2)$ for some sequence $(\Delta^2_n, \tau^2_n)_{n\ge 1}$, where each $\Delta^2_n $ is irreducible.
    For an arbitrary subsequence $(\Delta^2_{n_m}, \tau^2_{n_m})_{m\ge1}$, there exists a further subsequence $(\Delta^2_{n_{m_k}},\tau^2_{n_{m_k}})_{k\ge1}$ such that the limit $\lim_{k\to\infty}q^*(\Delta^2_{n_{m_k}},\tau^2_{n_{m_k}})$ exists due to the compactness of $[0,1]^{\mathscr S}$.
    Define a jointly continuous vector-valued map \[H(\Delta^2, \tau^2,q)= \e\hth^2(|z| (\tau^2+ 2\Delta^2\Lambda q)^{1/2}) -q  \text{ for }(\Delta^2, \tau^2,q)\in [0,\infty)^{\mathscr S \otimes \mathscr S}\times  [0,\infty)^{\mathscr S}\times  [0,1]^{\mathscr S}.\] 
    From the continuity of $H$, \begin{equation}\label{eq: continuity of H}
    H(\Delta^2, \tau^2,q^*(\Delta^2,\tau^2))=\0=\lim_{k\to\infty}H(\Delta^2_{n_{m_k}},\tau^2_{n_{m_k}},q^*(\Delta^2_{n_{m_k}}\tau^2_{n_{m_k}}))=H(\Delta^2,\tau^2, \lim_{k\to\infty}q^*(\Delta^2_{n_{m_k}},\tau^2_{n_{m_k}})).
    \end{equation} 
    In the case  $\tau^2\ne \0$, by the uniqueness  in Proposition \ref{prop: uniqueness of fixed-point}, the preceding display implies \begin{equation}\label{eq: subsubseq  converge}
        q^*(\Delta^2,\tau^2)=\lim_{k\to\infty}q^*(\Delta^2_{n_{m_k}},\tau^2_{n_{m_k}}).
    \end{equation}
    Therefore, the full sequence converges, that is, $q^*(\Delta^2,\tau^2)=\lim_{n\to\infty} q^*(\Delta^2_n,\tau^2_n)$,  establishing the desired continuity whenever $\tau^2\ne \0$.
Now, consider the case $\tau^2=\0$ and suppose $(\Delta^2_n,\tau^2_n)\gee (\Delta^2,\tau^2)$ for all $n\in\mathbb N$.
    By the previously established monotonicity, we have \begin{equation*}
        \lim_{k\to\infty}q^*(\Delta^2_{n_{m_k}},\tau^2_{n_{m_k}}) \gee q^*(\Delta^2,\tau^2).
    \end{equation*}
    Combining this with \eqref{eq: continuity of H} yields \eqref{eq: subsubseq  converge}, since $q^*(\Delta^2,\tau^2)$ is the maximal fixed-point by definition.
    Therefore, the full sequence converges in this case as well.
    This finishes the proof.
\end{proof}

For each $s\in \mathscr S$, define the function \[G_s(q)=\e \hth^2 \bigl(z \sqrt{\tau^2_s+2(\Delta^2\Lambda q)_s}\bigr)-q_s, \quad q\in [0,1]^{\mathscr S}.\]
For irreducible $\Delta^2$, define the regions \begin{align*}
    \mathsf{R}_{1} &=  \{q\in [0,1]^{\mathscr S}:  \forall s, \,  G_s(q) \ge  0\},
    \\ \mathsf{R}_{2} &= \{q\in [0,1]^{\mathscr S}:  \forall s, \,  G_s(q) \le  0\} \setminus \{\0\}.
\end{align*}
In what follows, we show that the regions $\mathsf{R}_i$ for $i=1,2$ 
are totally ordered with respect to $q^*$, which will play a key role in the proof of Proposition~\ref{prop: min and max of support}.
\begin{lemma}\label{lem: monotone region lemma}
    Assume irreducible $\Delta^2$.
    For $i=1,2$, we have \begin{align}
        \mathsf{R}_i &\subseteq \{q\in [0,1]^{\mathscr S}: (-1)^i q\gee (-1)^iq^*\}. \label{eq: monotone region}
    \end{align}
In fact, we have \begin{equation}\label{eq: strict monotone region}
    \mathsf{R}_i\setminus \{q^*\} \subseteq \{q\in [0,1]^{\mathscr S}: (-1)^i q \sgee  (-1)^iq^*\}.
\end{equation}
\end{lemma}
\begin{proof}
We proceed by induction on the number of species $m\colonequals |\mathscr S|.$

If $m=1$, then the concavity of $G_1$, which follows from Lemma~\ref{lem: concavity of fixed-point equation}, combined with $G_1(0)\ge 0>G_1(1)$ show that $G_1(q)>0$ if $0< q<q^*$ and $G_1(q)<0$ if $q>q^*$, verifying the base case.

Suppose that \eqref{eq: monotone region} and \eqref{eq: strict monotone region} hold for some $m\ge 1.$
We will verify the statement for $m+1$.
Before we begin the induction step, by permuting the species if necessary, we may assume that  any leading principal submatrices are also irreducible.
Denote the projection of any vector $v\in [0,1]^{m+1}$ to the first $m$ coordinates by $v_{\le m}$.
Similarly, denote the leading principal $m\times m$ submatrix of a matrix $A$ by $A_{\le m}$.
We also denote, respectively, $A_{\le m, m+1}$ and $A_{m+1, \le m}$ the projection of the last $(m+1)$'th column and row vectors of $A$ to their first $m$-coordinates.
We first analyze the following vector field, which is one  dimension lower.
 Since  $(\Delta^2\Lambda q)_{\le m} = (\Delta^2\Lambda)_{\le m}q_{\le m} + (\Delta^2\Lambda)_{\le m,m+1}q_{m+1}$ for any $q\in [0,1]^{m+1}$,  we can write  \begin{equation}\label{eq: vector field induction: 1}
     (G_i(q))_{1\le i\le m}= \e \hth^2 \bigl(z ( \tau^2_{\le m}+ 2 (\Delta^2\Lambda)_{\le m, m+1}q_{m+1}+ 2(\Delta^2\Lambda)_{\le m} q_{\le m})^{1/2}\bigr)-q_{\le m}.
 \end{equation}
Define the $m$-dimensional slices $\mathsf S_t = \{q\in [0,1]^{m+1}: q_{m+1}=t\}$ for $t\in [0,1].$
Define the normalization constant $\sum_{1\le i\le m}\lambda_i=c_m \in (0,1)$.
The preceding display shows that, for $t\in [0,1]$, the vector field $\{(G_i(q))_{1\le i\le m}: q \in \mathsf{S}_t\}$ is identical to the one from the variance profile $(c_m\Delta^2_{\le m},\tau^2_{\le m}+2(\Delta^2\Lambda)_{\le m, m+1}t) $ and species ratios $(\lambda_i/c_m)_{1\le i\le m}$.
As $c_m\Delta^2_{\le m}$ is irreducible by construction, we can define a map $t\mapsto \hat p_t \in [0,1]^m$ through the fixed-point equation $(G_i(\hat p_t,t))_{1\le i\le m}=\0$, per  Definition~\ref{def: definition of q^* for irreducible}.
Moreover, observe that $(\Delta^2\Lambda)_{\le m,m+1}=\Delta^2_{\le m,m+1}\Lambda_{\le m} \ne \0$ by the irreducibility of $\Delta^2$, so the effective external field  $\tau^2_{\le m}+2(\Delta^2\Lambda)_{\le m, m+1}t$ on $\mathsf{S}_t$ is increasing in $t\in [0,1]$ and $\ne\0$ for $t\in (0,1].$
Then, by Lemma~\ref{lem: q^* monotone and continuous}, the map $\hat p_t$ is  continuous and non-decreasing in $t\in [0,1]$.
By the induction hypothesis, for $i=1,2$ and $t\in [0,1]$, \begin{equation}\label{eq: vector field induction: 1.5}
\begin{aligned}
    &\{p\in [0,1]^{m}:  \forall k \le m,\, \,  G_k(p,t) \ge  0\}\subseteq \{p\in [0,1]^{m}:  p \lee  \hat p_t\},
    \\&\{p\in [0,1]^{m}:  \forall k \le m,\, \,  G_k(p,t) \le  0\} \setminus \{\0\}\subseteq \{p\in [0,1]^{m}: p \gee   \hat p_t\},
\end{aligned}
\end{equation}
and 
\begin{equation}\label{eq: vector field induction: 1.5.5}
\begin{aligned}
    &\{p\in [0,1]^{m}:  \forall k \le m,\, \,  G_k(p,t) \ge  0\} \setminus \{\hat p_t\}\subseteq \{p\in [0,1]^{m}: p \slee   \hat p_t\},
    \\&\{p\in [0,1]^{m}:  \forall k \le m,\, \,  G_k(p,t) \le  0\} \setminus \{\hat p_t, \0\}\subseteq \{p\in [0,1]^{m}: p \sgee   \hat p_t\}.
\end{aligned}
\end{equation}
On the other hand, the remaining component can be written as follows.
Since $(\Delta^2\Lambda q)_{m+1}=  (\Delta^2\Lambda)_{m+1, \le m} q_{\le m} + (\Delta^2\Lambda)_{m+1,m+1} q_{m+1} $, \begin{equation}\label{eq: vector field induction: 2}
    G_{m+1}(q)=\e \hth^2 \bigl(z ( \tau^2_{m+1}+ 2 (\Delta^2\Lambda)_{m+1, \le m} q_{\le m}+ 2(\Delta^2\Lambda)_{m+1,m+1} q_{m+1})^{1/2}\bigr)-q_{m+1}, 
\end{equation}
which corresponds to the variance profile $((\Delta^2\Lambda)_{m+1,m+1},\, \tau^2_{m+1}+  2(\Delta^2\Lambda)_{m+1, \le m} q_{\le m})$.  
In view of \eqref{eq: vector field induction: 2}, define an auxiliary real-valued function \[\widetilde G_{m+1}(x,y) =\e \hth^2 \bigl(z ( \tau^2_{m+1}+ 2x+ 2(\Delta^2\Lambda)_{m+1,m+1} y)^{1/2}\bigr)-y, \quad (x,y)\in [0,\infty)\times [0,1].\]
From Definition \ref{def: definition of q^* for irreducible}, there exists a non-negative function $[0,(\Delta^2\Lambda)_{m+1, \le m} \1] \ni x\mapsto y(x) $ such that $\widetilde G_{m+1}(x,y(x))=0$, and by Lemma~\ref{lem: q^* monotone and continuous}, it is continuous and increasing.
The domain of this function is non-empty since $(\Delta^2\Lambda)_{m+1,\le m}\1 >0$ due to the irreducibility of $\Delta^2$.
In particular, its image is a non-empty closed interval $[y(0), y( (\Delta^2\Lambda)_{m+1,\le m} \1)] \subseteq [0,1].$
For $t\in [0,1]$, it is not hard to see from the base case $m=1$ that 
\begin{equation} \label{eq: vector field induction: 3}
\begin{aligned}
    &G_{m+1}(p,t)= 0 \quad \text{if}\quad  t = y((\Delta^2\Lambda)_{m+1, \le m} p) \text{ or } \{t=0,\, \tau^2_{m+1}+ (\Delta^2\Lambda)_{m+1, \le m} p=0\}, 
    \\&G_{m+1}(p,t)> 0 \quad \text{if}\quad 0<t< y((\Delta^2\Lambda)_{m+1, \le m} p) \text{ or } \{t=0, \, \tau^2_{m+1}+ (\Delta^2\Lambda)_{m+1, \le m} p >0\},
    \\&G_{m+1}(p,t)<0 \quad \text{if}\quad  t> y((\Delta^2\Lambda)_{m+1, \le m} p).
\end{aligned} 
\end{equation}
Note that  the preceding display exhausts all possibilities.

Having defined the continuous functions $t\mapsto \hat p_t$ and $x\mapsto y(x)$, we define another continuous function \[ f(t)\colonequals y( (\Delta^2\Lambda)_{m+1, \le m} \hat p_t) - t \text{ for } t\in [0,1].\]
Since $\Delta^2$ is irreducible,   $q^*=q^*(\Delta^2,\tau^2)$ and possibly $\0$ are the only elements of $[0,1]^{m+1}$  that  render both \eqref{eq: vector field induction: 1} and \eqref{eq: vector field induction: 2}  simultaneously zero.
The utility of the function $f$ comes from the characterization that for $t\in (0,1]$, we have  $f(t)=0$ if and only if $(G_i(\hat p_t,t))_{1\le i\le m+1}=\0$, in view of \eqref{eq: vector field induction: 3}.
Therefore,  \begin{equation}\label{eq: t>q*m+1}
    f(q^*_{m+1})=0\ne f(t) \quad \text{if}\quad t\in (q^*_{m+1},1].
\end{equation}
Note that $(\Delta^2\Lambda)_{m+1,\le m}\ne \0$ by the irreducibility of $\Delta^2$, which implies $(\Delta^2\Lambda)_{m+1,\le m} (\1 -\hat p_1)>0$.
It follows that  \begin{equation}
    f(1)= y((\Delta^2\Lambda)_{m+1, \le m} \hat p_1)-1 <  y((\Delta^2\Lambda)_{m+1, \le m} \1)-y( (\Delta^2\Lambda)_{m+1,\le m} \1) =0. \label{eq: f(b)<0}
\end{equation}
Combining \eqref{eq: t>q*m+1} and \eqref{eq: f(b)<0}, the intermediate value theorem shows \begin{equation}\label{eq: f(t)<0}
    f(t)<0 \quad \text{for} \quad t\in (q^*_{m+1},1].
\end{equation} 
As for $0\le t\le  q^*_{m+1}$, we will use  concavity to show \begin{equation}\label{eq: f(t) ge 0}
\begin{aligned}
    f(t)&\ge0 \quad \text{for} \quad  t\in [0,q^*_{m+1}],
    \\ f(t)&>0 \quad \text{for} \quad t \in (0, q^*_{m+1}).
\end{aligned}
\end{equation}
To this end, we first note that  \eqref{eq: f(t) ge 0} is trivial whenever $q^*=\0$.
Hence, we may assume $q^*\in (0,1)^{m+1}$ by Definition~\ref{def: definition of q^* for irreducible}.
With obvious notations, since $G_{\le m+1}(\hat p_{q^*_{m+1}},q^*_{m+1})=\0$ and $G_{\le m+1}(\0)\gee \0$, the concavity in Lemma~\ref{lem: concavity of fixed-point equation} implies \begin{equation}\label{eq: showing f(t) ge 0}
    \begin{aligned}
    &G_{\le m+1} (s \hat p_{q^*_{m+1}}, sq^*_{m+1}) \gee \0, \quad \forall s\in[0,1],
    \\&G_{\le m+1} (s \hat p_{q^*_{m+1}}, sq^*_{m+1}) \sgee \0, \quad \forall s\in (0,1).
    \end{aligned}
\end{equation}
In particular, the induction hypothesis \eqref{eq: vector field induction: 1.5} applied to $G_{\le m}$ shows \begin{equation}\label{eq: showing f(t) ge 0: 1}
      \hat p_{sq^*_{m+1}}\gee    s \hat p_{q^*_{m+1}} ,\quad \forall s\in [0,1].
\end{equation}
Moreover, \eqref{eq: vector field induction: 3} and \eqref{eq: showing f(t) ge 0} together imply that \begin{equation}\label{eq: showing f(t) ge 0: 2}
\begin{aligned}
        &y((\Delta^2\Lambda)_{m+1, \le m} s \hat p_{q^*_{m+1}})\ge s q^*_{m+1}, \quad \forall s\in [0,1],
        \\& y((\Delta^2\Lambda)_{m+1, \le m} s \hat p_{q^*_{m+1}})> s q^*_{m+1}, \quad \forall s\in (0,1).
\end{aligned}
\end{equation}
Combining \eqref{eq: showing f(t) ge 0: 1} and \eqref{eq: showing f(t) ge 0: 2}, since $x\mapsto y(x)$ is increasing, we arrive at an equivalent of \eqref{eq: f(t) ge 0},
 \begin{equation*}
    \begin{aligned}
        &f(sq^*_{m+1})=  y((\Delta^2\Lambda)_{m+1, \le m}  \hat p_{sq^*_{m+1}}) -sq^*_{m+1}\ge 0, \quad \forall s\in [0,1],
        \\& f(sq^*_{m+1})=  y((\Delta^2\Lambda)_{m+1, \le m}  \hat p_{sq^*_{m+1}}) -sq^*_{m+1}> 0, \quad \forall s\in (0,1).
    \end{aligned}
\end{equation*}  
We are now in a position to finish the induction with the following four steps.

\textit{Step 1. Verify \eqref{eq: monotone region} for $i=1$.}
Suppose that $q\in \mathsf{R}_1$, i.e., $G_k(q_{\le m}, q_{m+1})\ge 0$ for all $1\le k\le m+1$.
By the induction hypothesis \eqref{eq: vector field induction: 1.5}, we have  $q_{\le m}\lee \hat p_{q_{m+1}}.$
On the other hand,  \eqref{eq: vector field induction: 3} implies $y((\Delta^2\Lambda)_{m+1,\le m} q_{\le m})\ge q_{m+1}.$
The last two inequalities imply $f(q_{m+1})\ge 0$, from which another dichotomy from \eqref{eq: f(t)<0}  and \eqref{eq: f(t) ge 0}   implies $q_{m+1}\le q^*_{m+1}.$
Since $t\mapsto \hat p_t$ is non-decreasing, we further have  $q_{\le m} \lee \hat p_{q_{m+1}}\lee \hat p_{q^*_{m+1}}=q^*_{\le m}$.
The previous inequalities verify $q\lee q^*$, as desired.

\textit{Step 2. Verify \eqref{eq: strict monotone region} for $i=1$.}
If $q^*= \0$, then $\mathsf{R}_1=\{\0\}$ from Step 1, which implies $\mathsf{R}_1\setminus\{q^*\}=\emptyset$.
Hence, we may assume $q^*\sgee \0$ by Proposition~\ref{prop: uniqueness of fixed-point}.
Suppose $q\in \mathsf{R}_1 \setminus \{q^*\}$, that is, $G_s(q)\ge 0$ for all $s\in\mathscr S$ and $q\ne q^*.$
We know $q\lee q^*$ from Step 1.
The goal is to show $q\slee q^*.$
Define $I=\{k\in\mathscr S: q_k= q^*_k\}$ and $J=\{k\in\mathscr S: q_k<q^*_k\}$ so that $I\dot \cup J=\mathscr S$.
Since $q\ne q^*$, we have $J\ne\emptyset.$
Toward a contradiction, assume $I\ne\emptyset$.
By irreducibility, there exist some $i_0\in I$ and $j_0\in J$ such that $\Delta^2_{i_0j_0}>0$.
On the other hand, we can find an increasing, piecewise-linear path $\mathcal{P}$ from $q$ to $q^*$ inside the polytope $\{p\in[0,1]^{\mathscr S} : \forall k\in I, \, p_k = q_k^* \}$, formed by line segments each parallel to some axis in $J$.
Since $G_{i_0}(q^*)=0$, we would obtain a contradiction to the fact $G_{i_0}(q)\ge 0$ if (i)  $\partial_k G_{i_0}\ge 0$ for all $k\in J$, (ii)  $\partial_{j_0} G_{i_0} >0$,   by applying the fundamental theorem of calculus to the function $G_{i_0}$ along the path $\mathcal P$.
To this end,  by Gaussian IBP, for all $l \ne i_0$, we have \begin{equation}\label{eq: partial of G}
    \partial_{l} G_{i_0} (q) = \e (\tanh^2)''(z\sqrt{(\tau^2)_{i_0}+2(\Delta^2 \Lambda q)_{i_0}}) (\Delta^2\Lambda)_{i_0 l}.
\end{equation}
Denote the smooth density  of the random variable  $z\sqrt{(\tau^2)_{i_0}+2(\Delta^2\Lambda q)_{i_0}}$ by $\psi_{i_0,q}$, which is even and decreasing on $[0,\infty)$.
For $l\ne i_0$, if we have degenerate density $\psi_{i_0,q}=\delta_0$, then \eqref{eq: partial of G} is equal to $2(\Delta^2\Lambda)_{i_0l}\ge 0$ because $(\hth^2)''(0)=2.$ 
If $\psi_{i,q}$ is non-degenerate, applying integration by parts to \eqref{eq: partial of G} yields   \begin{equation*}
    \partial_{l} G_{i_0} (q)= (\Delta^2\Lambda)_{i_0l}\int_{-\infty}^\infty (\hth^2)''(x) \psi_{i_0,q}(x)dx = -(\Delta^2 \Lambda)_{i_0l} \int_{-\infty}^\infty(\hth^2)'(x)\psi_{i_0,q}'(x)dx \ge 0, 
\end{equation*} 
where the inequality holds because the integrand is negative on the whole real line.
In both cases, for $i_0\ne l$, \eqref{eq: partial of G} is  always non-negative, and indeed positive whenever $\Delta^2_{i_0 l}>0$.
This implies the desired conditions (i) and (ii), thereby confirming that $I=\emptyset$.

\textit{Step 3. Verify \eqref{eq: monotone region} for $i=2$.}
Suppose $G_k(q_{\le m}, q_{m+1})\le 0$ for all $1\le k\le m+1$ and $q\ne \0$.
Toward a contradiction, assume $q_{\le m}=\0$.
Since $q\ne \0$, we then have $q_{m+1}>0$.
From the irreducibility, there exists $1\le l\le m$ such that $\Delta^2_{l,m+1}>0$.
From \eqref{eq: partial of G} and the argument in Step 2, we have $\partial_{m+1}G_{l}>0$.
Since $G_l(\0)\ge 0$, the fundamental theorem of calculus shows $G_l(q)>0$.
This is a contradiction, so we must have $q_{\le m}\ne \0.$
In turn, by the induction hypothesis \eqref{eq: vector field induction: 1.5}, we have \begin{equation}\label{eq: vector field induction: 4}
    q_{\le m}\gee \hat p_{q_{m+1}}.
\end{equation}
On the other hand, by \eqref{eq: vector field induction: 3}, we have either (i) $q_{m+1}\ge y( (\Delta^2\Lambda)_{m+1,\le m}q_{\le m})$ or (ii) $q_{m+1}=0=\tau^2_{m+1}+(\Delta^2\Lambda)_{m+1,\le m}q_{\le m}$.
The case (ii) is impossible as follows.
By irreducibility, there exists $1\le i_0\le m$ such that $\Delta^2_{m+1,i_0}>0$.
Then, $0=(\Delta^2\Lambda)_{m+1,\le m}q_{\le m}=\sum_{1\le l\le m}\Delta^2_{m+1,l}\lambda_l q_l$ implies  $q_{i_0}=0$.
 As a result,  \eqref{eq: vector field induction: 4} implies $(\hat p_0)_{i_0}=0$ and hence $\hat p_0=\0$ by Proposition~\ref{prop: uniqueness of fixed-point}.
Thanks to $q_{\le m}\ne \0$,  another induction hypothesis \eqref{eq: vector field induction: 1.5.5} shows $q_{\le m}\sgee \0$.
However,  this is a contradiction to $q_{i_0}=0$, thereby ruling out the latter case (ii).
In the former case (i), together with \eqref{eq: vector field induction: 4}, it holds that $f(q_{m+1})\le 0$.
Consequently, \eqref{eq: f(t)<0} and \eqref{eq: f(t) ge 0} show either (a) $0=q_{m+1}=f(0)=y( (\Delta^2\Lambda)_{m+1,\le m}\hat p_0)$ or (b) $q_{m+1}\ge q^*_{m+1}$.
To rule out the subcase (a), notice that $0=y( (\Delta^2\Lambda)_{m+1,\le m}\hat p_0)$ already implies,  by Proposition~\ref{prop: uniqueness of fixed-point}, \begin{equation}
    0=\tau^2_{m+1}.\label{eq: vector field induction: 5}
\end{equation}
In particular, we have $0\ge G_{m+1}(q_{\le m},0)=\e \hth^2 (z (2(\Delta^2\Lambda)_{m+1,\le m}q_{\le m})^{1/2})\ge 0$, which shows \begin{equation}\label{eq: vector field induction: 6}
    0=(\Delta^2\Lambda)_{m+1,\le m}q_{\le m}.
\end{equation}
From $0=q_{m+1}$, \eqref{eq: vector field induction: 5}, and \eqref{eq: vector field induction: 6}, one sees that subcase (a) reduces to the previous case (ii) which was shown to be impossible. 
Finally, we are left with  subcase (b) and, by \eqref{eq: vector field induction: 4}, we have $q_{\le m} \gee \hat p_{q_{m+1}}\gee \hat p_{q^*_{m+1}}=q^*_{\le m}$, finishing the proof of Step 3.

\textit{Step 4. Verify \eqref{eq: strict monotone region} for $i=2.$}
Suppose $G_k(q_{\le m}, q_{m+1})\le 0$ for all $1\le k\le m+1$ and $q\notin \{\0,q^*\}$.
By Step 3, we have $q\gee q^*.$ 
Now, we can repeat the argument in Step 2 to derive a contradiction whenever $\{k\in\mathscr S: q_k = q^*_k\}\ne \emptyset.$



\end{proof}

For later use, we establish a variant of Proposition \ref{prop: infimum of RS functional for positive-definite} under the weaker assumption of positive-semidefinite and irreducible $\Delta^2$.
    Under this weaker assumption, we  lose uniqueness of minimizers; for any minimizer $q$ and $v\in \ker (\Delta^2\Lambda)$, we have $\mathsf{RS}(q+v)=\mathsf{RS}(q)$.
    However, the following lemma guarantees that there are  minimizers that still satisfy the fixed-point equation \eqref{eq: fixed-point equation in RS}.

\begin{lemma}\label{lem: infimum of RS: semidefinite}
    Assume $\Delta^2$ is irreducible and positive-semidefinite.
    Then, $q^*(\Delta^2,\tau^2)$ is a minimizer of the RS functional and a fixed-point of \eqref{eq: fixed-point equation in RS}.


\end{lemma}

\begin{proof}
     
     
     \noindent For $n\in\mathbb N$, consider   \begin{equation}\label{eq: approximating Delta^2}
        \Delta^2_n\colonequals \Delta^2+n^{-1} I \quad \text{and} \quad  \tau^2_n\colonequals \tau^2+n^{-1}\1.
    \end{equation}
    For each $n\in\mathbb N$,  $\Delta^2_n$ is positive-definite and  irreducible, and $\tau^2_n\sgee \0.$
    Hence, for each $n\in\mathbb N$, Proposition \ref{prop: infimum of RS functional for positive-definite} guarantees a unique minimizer $q^*_n\colonequals q^*(\Delta^2_n,\tau^2_n)$ of the RS functional $\mathsf{RS}_n$ under the variance profile $(\Delta^2_n,\tau^2_n)$.
    In other words, $\mathsf{RS}_n(q^*_n)=\min_{q\in [0,1]^{\mathscr S}} \mathsf{RS}_n(q).$
    
    On the one hand, by the Lipschitz continuity of the Parisi functional in Lemma~\ref{lem: Lipschitz of Parisi functional in model parameters}  (or an elementary calculus), it is not hard to see that $\mathsf{RS}_n\to \mathsf{RS}$ uniformly on $[0,1]^{\mathscr S}$  as $n\to\infty$.
    On the other hand, by continuity from above in Lemma~\ref{lem: q^* monotone and continuous}, we have $\lim_{n\to\infty}q^*_n= q^*$, where $q^*$ is the maximal fixed-point under the variance profile $(\Delta^2,\tau^2)$ according to Definition \ref{def: definition of q^* for irreducible}. 
    Combining these, $q^*$ must be a minimizer; indeed, we have \begin{equation}\label{eq: extending to non-negative entries: 0}
        \mathsf{RS}(q^*)=\lim_{n\to\infty}\mathsf{RS}_n(q^*_n)=\lim_{n\to\infty} \inf_{q\in [0,1]^{\mathscr S}}\mathsf{RS}_n(q)=\inf_{q\in [0,1]^{\mathscr S}}\mathsf{RS}(q).
    \end{equation}

    

\end{proof}


\section{Support of Parisi measures}\label{sec: Support of Parisi measures}
This section is dedicated to establishing an important property of a Parisi measure, which was the key ingredient in Step 4 of the proof sketch.
Denote the space of probability measures on $[0,1]^{\mathscr S}$ with supports totally ordered with respect to the partial order $\lee$ by $\mathcal M^\uparrow$.
Let  $\mathcal M^\uparrow_d\subseteq \mathcal M^\uparrow$ be the discrete probability measures with totally ordered supports.
As outlined in the introduction, the Parisi functional can be continuously extended from $\mathcal M^\uparrow_d$ to  $\mathcal M^\uparrow$, and the Parisi measure can formally be defined as a minimizer of the Parisi functional in $\mathcal M^\uparrow$; see Proposition~\ref{prop: Lipschitz continuity of Parisi functional} and Lemma~\ref{lem: closure of discrete with totally ordered support}.
Indeed,  $\mathcal M^\uparrow$ is compact as a weak closure of $\mathcal M^\uparrow_d$.

For any $\mu\in\mathcal M^\uparrow$, we define $q_{\text{min}}^\mu$ and $q_{\text{max}}^\mu$ to be the unique minimizer and maximizer of the $L^1$ norm on $\operatorname{supp}(\mu)$, respectively.
These  are well-defined since the support of a measure is totally ordered and compact in $[0,1]^{\mathscr S}$.
\begin{proposition}\label{prop: min and max of support}
 Assume irreducible and positive-definite $\Delta^2$. 
 For any Parisi measure $\mu\in\mathcal M^\uparrow$, we have either $q^{\mu}_{\mathrm{min}} =q^*$ or $q^\mu_{\mathrm{min}} \slee q^*$.  
We also have either $q^{\mu}_{\mathrm{max}} =q^*$ or $q^\mu_{\mathrm{max}} \sgee q^*$.  
If we further assume $\tau^2\sgee \0$, then $q^\mu_{\min}\sgee\0$.
\end{proposition}
The  idea to achieve this is to show that for any Parisi measure $\mu$, the end points  $q^\mu_{\mathrm{min}}$ and $q^\mu_{\mathrm{max}}$ are in the regions $\mathsf R_1$ and $\mathsf R_2$, respectively, in accordance with Lemma~\ref{lem: monotone region lemma}. 
The  assertion  $q^\mu_{\min}\sgee \0$ and $q^\mu_{\max}\slee \1$ will be treated separately.
Toward this, we analyze a certain critical point equation for Parisi measures,  which is stated in Proposition~\ref{prop: support of Parisi measure}, by discretizing a Parisi measure as in \eqref{eq: Psi for discrete measure}.
Hence,  most of the effort in this section is to examine this discretization, which is the content of the next subsection.

\subsection{Preparatory lemmas }
Define the random variables  $Y^s_0= h_s$ for $s\in\mathscr S$ and
\begin{align*}
    Y^s_l &= h_s+\sum_{1\le l'\le l}z_{l'} (Q_{l'}^s -Q_{l'-1}^s)^{1/2} \quad \text{for} \quad  1\le l\le r, \, \, s \in \mathscr S,
    \\W^s_l &= \exp (\zeta_{l-1} (X_{l}^s -X_{l-1}^s)) \quad  \text{for} \quad  1\le l\le r.
\end{align*} 
Note that $W^s_l$ is $(h_s,(z_{k})_{1\le k\le l})$-measurable and $\e _l W^s_l=1$ for all $l$ by \eqref{eq: inductive definition}.
Define the operators \[T_{k}(Z)=(\e_{k+1}Z)^{\zeta_{k-1}/\zeta_k}, \quad 0\le k\le r,\]  for any  $((h_s)_{s\in\mathscr S},(z_{l})_{1\le l\le k+1})$-measurable random variable $Z\ge 0$.
Here, $\e_{r+1}$ is understood to be the identity operator.
We define its composition \[\widetilde T_k = T_k\circ \dots \circ T_r, \quad 0\le k\le r.\]
These operators are related to the weights in the discretization in \eqref{eq: Psi for discrete measure}, namely \[\e_h\e_{1,\dots, l} W^s_1\dots  W^s_l (\e_{l+1,\dots,r} W^s_{l+1}\dots W^s_{r} \hth  Y^s_r )^2,\] in the following way.
\begin{lemma}\label{lem: weight W}
    For any $1\le l\le r$ and $s \in \mathscr S,$  \[ W^s_l= \frac{\widetilde T_l \ch Y^s_r}{\e_l \widetilde T_l \ch Y^s_r}.  \] 
\end{lemma}
\begin{proof}[Proof sketch]
    We compute the last three terms and leave the formal inductive proof to the reader.
    From \eqref{eq: inductive definition} and \eqref{eq: terminal X}, we have $\exp (\zeta_{r-1}X_{r-1}^s) = \e_r \ch^{\zeta_{r-1}}(Y^s_r)$ and \[W^s_r = \frac{\ch^{\zeta_{r-1}}(Y^s_r)}{\e_r \ch^{\zeta_{r-1}}(Y^s_r)}.\] 
    One more step yields $\exp (\zeta_{r-2}X_{r-2}^s) = \e_{r-1}  (\e_r \ch^{\zeta_{r-1}}(Y^s_r))^{\zeta_{r-2}/\zeta_{r-1}}$ and \[W^s_{r-1} = \frac{(\e_r \ch^{\zeta_{r-1}}(Y^s_r))^{\zeta_{r-2}/\zeta_{r-1}}}{\e_{r-1}  (\e_r \ch^{\zeta_{r-1}}(Y^s_r))^{\zeta_{r-2}/\zeta_{r-1}}}.\] 
    Yet another step yields $\exp (\zeta_{r-3}X_{r-3}^s) = \e_{r-2} (\e_{r-1}  (\e_r \ch^{\zeta_{r-1}}(Y^s_r))^{\zeta_{r-2}/\zeta_{r-1}})^{\zeta_{r-3}/\zeta_{r-2}} $ and \[W^s_{r-2} = \frac{(\e_{r-1}  (\e_r \ch^{\zeta_{r-1}}(Y^s_r))^{\zeta_{r-2}/\zeta_{r-1}})^{\zeta_{r-3}/\zeta_{r-2}}}{\e_{r-2} (\e_{r-1}  (\e_r \ch^{\zeta_{r-1}}(Y^s_r))^{\zeta_{r-2}/\zeta_{r-1}})^{\zeta_{r-3}/\zeta_{r-2}}}.\] 
        
\end{proof}

    



We will revisit the operator $\widetilde T_l$ after we establish an auxiliary lemma about symmetric random variables.

\begin{lemma}\label{lem: auxiliary lemma: 3}
    Let  $g:\mathbb R\to\mathbb R$ be a measurable function satisfying $g(y)\le g( -y)$ for all $y\ge0.$
    Let $X$ be a random variable with an even density function that is  non-increasing on $[0,\infty)$.
    Assume a mild integrability condition $\sup_{y\in\mathbb R}\e |g(y+X)|<\infty.$
    For any $c \ge 0$, we have \[\e g( cy+X)\le \e g( -cy+X), \quad \forall \, y\ge 0.\]
    In particular, for a measurable function $h:\mathbb R\to\mathbb R$ which is even, non-increasing on $[0,\infty)$, and satisfies $\sup_{t\ge 0}\e |h(t+X)|<\infty$,  the function $t\mapsto \e h(t+X)$ is non-increasing on $[0,\infty)$.
    Similarly, if such $h$ is now non-decreasing on $[0,\infty)$, then  $t\mapsto \e h(t+X)$ is non-decreasing on $[0,\infty)$.
\end{lemma}
\begin{proof}
    Denote the even density of $X$ by $f:\mathbb R\to\mathbb [0,\infty)$.
    For $c,y\ge 0$, by change of variables, \begin{align*}
        \e g( cy+X)- \e g( -cy+X)& = \int_{-\infty}^\infty g(cy+w)f(w) dw -   \int_{-\infty}^\infty g(-cy+w)f(w) dw
        \\& = \int_{-\infty}^\infty g(z) \bigl( f(z- cy) -f(z+cy)\bigr) dz.
    \end{align*}
    Splitting  the integral and changing the variables again, the last integral is equal to \begin{align*}
    &\int_{-\infty}^0 g(z) (f(z- cy) -f(z+cy)) dz  +\int_{0}^\infty g(z) (f(z- cy) -f(z+cy)) dz
        \\&  =\int_{0}^\infty g(-z) (f(-z-cy)- f(-z+cy)) dz  +\int_{0}^\infty g(z) (f(z- cy) -f(z+cy)) dz
        \\&  =\int_{0}^\infty (g(z) -g(-z)) (f(z- cy) -f(z+cy)) dz ,
    \end{align*} where the last equality is due to the evenness of $f$.
    For any $z,y\ge 0$,  the triangle inequality shows  $|z-cy|\le z+cy$ and, therefore, $f(z- cy) \ge f(z+cy)$. 
    This, together with the assumption on $g$, implies that the preceding display is non-positive, which is the first assertion of the lemma.

    As a special case, consider a measurable function $h:\mathbb R\to\mathbb R$ which is  even and non-increasing on $[0,\infty)$.
    Let $0\le a<b$.
    Denote the center by $c= (a+b)/2>0$ and the radius by $r= (b-a)/2>0.$
    Define $g(y)=h(c+y)$ for  $y\in\mathbb R$.
    One can easily check that $g(y)=h(c+y)\le h(c-y)=g(-y)$ for any $y\ge 0$ due to the assumptions on $h.$
    From the first assertion above, \[\e h(b+X) =\e h(c+r+X) =\e g(r+X) \le \e g(-r+X)= \e h(c-r+X)= \e h(a+X),\] which shows that $t\mapsto \e h(t+X)$ is non-increasing on $[0,\infty)$.
    For $h$ non-decreasing on $[0,\infty)$, we can consider $-h$ to conclude that $t\mapsto \e h(t+X)$ is non-decreasing on $[0,\infty)$, which completes the proof.
\end{proof}

With this auxiliary lemma, we  establish some useful properties of the operator $\widetilde T_l$.
\begin{lemma}\label{lem: properties of g_l}
    For each $1\le l\le r$, define a measurable function $g_l$ through the relation $g_l(Y^s_l)= \widetilde T_l(\ch Y^s_r).$
    Then, for all $1\le l\le r$, the functions $g_l$ are positive, smooth, even, and non-decreasing on $[0,\infty)$.
\end{lemma}
\begin{proof}
Since $\widetilde T_l (\ch Y^s_r)$ is $Y^s_l$-measurable, there exists some some measurable function $g_l$ such that $g_l(Y^s_l)= \widetilde T_l(\ch Y^s_r)$.
We check the asserted properties of $g_l$ by backward induction on $l.$
The base case $l=r$ clearly holds, since $g_r= \ch^{\zeta_{r-1}}$.
Suppose the claim is true for some $1\le l\le r.$
Since $\widetilde T_{l-1}= T_{l-1}\circ \widetilde T_l$, we have \[ g_{l-1}(Y^s_{l-1}) = T_{l-1} ( \widetilde T_l (\ch Y^s_r))= T_{l-1} (g_l(Y^s_l) ) = (\e_l g_l(Y^s_l))^{\zeta_{l-2}/\zeta_{l-1}}.\] 
The preceding display and $Y^s_l= Y^s_{l-1}+ z_l (Q^s_l-Q^s_{l-1})^{1/2}$ show
 \begin{equation}\label{eq: recursive g}
     g_{l-1}(x) = (\e_l  g_l(x+z_l (Q^s_l-Q^s_{l-1})^{1/2}))^{\zeta_{l-2}/\zeta_{l-1}}, \quad x\in\mathbb R.
 \end{equation} 
Then, the induction hypothesis and Lemma~\ref{lem: auxiliary lemma: 3} readily show that $g_{l-1}$ satisfies the required properties, which completes the induction.

\end{proof}

We prove another auxiliary lemma on symmetric random variables, which serves as the base step of the induction in the sequel.
\begin{lemma}\label{lem: auxiliary lemma: 1}
    Let $f:\mathbb R\to\mathbb R$ be an odd measurable function such that $f\ge 0$ on $[0,\infty)$.
    Let $g:\mathbb R\to\mathbb R $ be a  bounded measurable function satisfying $g(y)\le g( -y)$ for any $y\ge0.$
    Let $X$ be a  symmetric random variable.
    Suppose $ \e f(|X|)<\infty.$
    Then, we have  \[\e f(X) g( X) \le 0.\] 
\end{lemma}
\begin{proof}
    Since $X\stackrel{d}{=} S|X|$ for some independent symmetric Rademacher random variable $S$, we can write \begin{align*}
        \e f(X) g(X) &=2^{-1}\e \bigl( f(|X|) g( |X|) +f(-|X|)g( -|X|)\bigr)
        \\&=2^{-1}\e  f(|X|)  \bigl( g(|X|) - g(-|X|)\bigr)
        \le 0,
    \end{align*} 
    where we used the fact that $f$ is odd in the second equality.
    Note that the last inequality is valid from the assumptions on $g$ and the fact that $f\ge 0$ on $[0,\infty)$.

\end{proof}

We record a generalization of Lemma~\ref{lem: auxiliary lemma: 1} for centered Gaussian random variables.
\begin{lemma}\label{lem: generalization}
    Let $n\in\mathbb N$ and $(z_i)_{1\le i\le n}$ be independent centered Gaussian random variables.
    Denote the partial sums by $S_i= \sum_{1\le k\le i}z_i$ for $1\le i\le n.$
    Let $(g_{i})_{1\le i\le n}$ be a family of non-negative, bounded, and   measurable functions such that $g_{i}(y) \le g_{i}(-y)$ for all $1\le i\le n$ and $y\ge 0$.
    Let $f$ be an odd measurable function such that $f\ge 0$ on $[0,\infty)$.
    Suppose $\e f(|S_n|)<\infty$.  
    Then, we have \[ \e f(S_n) \prod_{1\le i\le n}  g_{i}( S_i) \le 0.\] 
    In particular, if $(h_i)_{1\le i\le n}$ is a familty of non-negative and even functions that are non-increasing on $[0,\infty)$, we have \[\e f(S_n) \prod_{1\le i\le n}  h_i(x+S_i) \le 0, \quad \forall x\ge 0.\]
\end{lemma}
\begin{proof}
    We use induction on $n$.
    The base case $n=1$ holds by Lemma~\ref{lem: auxiliary lemma: 1}.
    Suppose the statement is true for some $n\ge 1$, and consider the case for $n+1$.
    By conditioning, we have \begin{equation}\label{eq: conditioning: 1}
            \e f(S_{n+1}) \prod_{1\le i\le n+1}g_{i}(S_i) = \e \bigl(  \e [ g_{1}(S_1)|S_2,\dots, S_{n+1} ]  \cdot f(S_{n+1}) \prod_{2\le i\le n+1} g_{i}(S_i) \bigr).
    \end{equation}
    The conditional expectation in \eqref{eq: conditioning: 1} can be simplified by the tower property: for any bounded measurable function $h$, \begin{align}
        \e \bigl[ h(S_1) |S_2,\dots, S_{n+1} \bigr]&= \e \bigl[ \e [  h(S_1) |S_2,z_3,\dots,z_{n+1}]\, \bigr| S_2,\dots, S_{n+1} \bigr] \nonumber
        \\&= \e \bigl[ \e [ h(S_1) |S_2]\, \bigr| S_2,\dots, S_{n+1} \bigr]
        = \e\bigl[ h(S_1) |S_2\bigr], \label{eq: conditioning: 2}
    \end{align}
    where the second equality holds since $(S_1,S_2)$ is independent with $(z_3,\dots,z_{n+1})$.
    We can write $S_1=\rho S_2+ R$, where $\rho \colonequals \e z_1^2 / (\e z_1^2 +\e z_2^2)$ and  $R\sim N(0, \e z_1^2 \e z_2^2 /(\e z_1^2 +\e z_2^2) )$ is a normal random variable independent with  $S_2$. 
    Using \eqref{eq: conditioning: 2}, the conditional expectation appearing in \eqref{eq: conditioning: 1} is  \begin{equation}\label{eq: conditioning: 3}
        \e \bigl[g_{1}(S_1)|S_2 \bigr]= \e \bigl[g_{1}( \rho S_2 + R)|S_2\bigr] = \hat g_{1}( S_2),
    \end{equation} where the function $\hat g_{1}:\mathbb R^2 \to [0,\infty)$ is given by $\hat g_{1}(y)= \e g_{1}(\rho y+R)$ for  $y\in\mathbb R$.
     Combining \eqref{eq: conditioning: 1}, \eqref{eq: conditioning: 2}, and \eqref{eq: conditioning: 3}, if we define $\tilde g_{2}:\mathbb R\to [0,\infty)$ by the pointwise product $\tilde g_{2} =\hat g_{1}\cdot g_{2} $, we have \begin{equation*}
        \e f(S_{n+1}) \prod_{1\le i\le n+1} g_{i}(S_i)= \e f(S_{n+1}) \tilde  g_{2}(S_2) \prod_{3\le i\le n+1} g_{i}(S_i).
    \end{equation*}
    Notice that $\hat g_{1}(  y)\le \hat g_{1}( -y)$ for all $y\ge 0$ by Lemma~\ref{lem: auxiliary lemma: 3}, whereas $g_{2}$ satisfies the same inequality by assumption.
    This, together with $\hat g_{1}, g_{2} \ge 0$, implies $\tilde g_{2}(y)\le \tilde g_{2}(-y)$ for all $y\ge 0$.
    Now, the induction hypothesis can be applied to  the  Gaussian random variables $(S_2, S_3,\dots, S_{n+1})$ and the functions $(\tilde g_{2},g_{3},\dots,g_{n+1})$, finishing the induction.

    For  even and non-negative functions $(h_i)_{1\le i\le n}$  which are non-increasing on $[0,\infty)$, we first fix $x\ge 0$ and define $g_{i}(y)=h_i(x+y)$ for any $y\in\mathbb R$ and $1\le i\le n$.
    Since $g_{i}(y)=h_i(x+y)\le h_i(x-y)=g_{i}(-y)$ for all $y\ge 0$ and $1\le i\le n$, the previous result can be applied  in this special case, yielding the desired result.
\end{proof}


We end this subsection with two parallel lemmas that are critically used in the proof of Proposition~\ref{prop: min and max of support} in the next subsection.
\begin{lemma} \label{lem: weight reduction lemma}
For all $1\le l\le r$ and $s \in \mathscr S$, we have \[ |\e_{l,\dots, r} W^s_l\dots W^s_r \hth Y^s_r| \le |\hth 
Y^s_{l-1}|.\] 
\end{lemma} 
\begin{proof}
For brevity, let us drop the species notation $s$ and write $Y_l= h +\sum_{1\le i\le l} \eta_{i}$ for $0\le l\le r$, where  $\eta_{i}\colonequals z_{i} (Q^s_{i} -Q^s_{i-1})^{1/2}$.
Fix $1\le l\le r.$
From Lemma~\ref{lem: weight W} and  Lemma~\ref{lem: properties of g_l}, the target inequality  can be written as  \begin{equation*}
    \Bigl| \e_{l,\dots,r}\prod_{k:l\le k\le r}\frac{g_k(Y_k)}{\e_{k}
    g_k(Y_k)}  \cdot \hth Y_r \Bigr| \le |\hth  Y_{l-1}|,
\end{equation*}  where the functions $(g_k)_{l\le k\le r}$ enjoy the properties of the latter lemma.
In particular, the functions $(g_k)_{l\le k\le r}$ are even.
Combining this with the fact that $(-\eta_i)_{l\le i\le r}$ is equal in distribution to $(\eta_i)_{l\le i\le r}$, it is not hard to see that both sides of the preceding display without the absolute values are odd measurable functions in $Y_{l-1}$.
Therefore,  by substituting $x=Y_{l-1}$, it suffices to show that  \begin{equation*}
    \e_{l,\dots,r}\prod_{k:l\le k\le r}\frac{g_k(x+S_{lk})}{\e_{k}
    g_k(x+S_{lk})}  \cdot \hth (x+S_{lr}) \le \hth  (x), \quad \forall \, x\ge 0,
\end{equation*} where $S_{lk}\colonequals \sum_{l\le i\le k} \eta_i$ for each $k\ge l$.
Since  $\e_{l,\dots,r}\prod_{k:l\le k\le r}\frac{g_k(x+S_{lk})}{\e_{k}
    g_k(x+S_{lk})} =1$, the preceding inequality is equivalent to \begin{equation*}
        \e_{l,\dots,r}\prod_{k:l\le k\le r}\frac{g_k(x+S_{lk})}{\e_{k}
    g_k(x+S_{lk})}  \cdot (\hth (x+S_{lr})-\hth(x)) \le 0, \quad \forall \, x\ge 0.
    \end{equation*}
By \eqref{eq: recursive g}, we have $\e_{k+1}g_{k+1}(x+S_{l,k+1}) = (g_k(x+S_{lk}))^{\zeta_k/\zeta_{k-1}}$ for all $k$ and $x\in \mathbb R$.
Telescoping the preceding display and multiplying out the deterministic quantity $\e_l g_l(x+S_{ll})>0$ yields another equivalent inequality \begin{equation*}
     \e_{l,\dots,r}\prod_{k: l\le k\le r-1} (g_k(x+S_{lk}))^{-m_k}  \cdot  g_r(x+S_{lr}) \cdot (\hth (x+S_{lr}) -\hth(x))  \le  0, \quad \forall \, x\ge 0,
\end{equation*} where $m_k\colonequals   \zeta_k \zeta_{k-1}^{-1}  -1 \ge 0$.
Further multiplying $\ch(x)>0$ to both sides and using the formula $\sh(a)\ch(b)-\ch(a)\sh(b)=\sh(a-b)$, our target inequality becomes \begin{equation}\label{eq: target inequality}
    \e_{l,\dots,r}\prod_{k: l\le k\le r-1} (g_k(x+S_{lk}))^{-m_k}  \cdot  g_r(x+S_{lr}) \cdot (\ch(x+S_{lr}))^{-1} \cdot \sh (S_{lr})   \le  0, \quad \forall \, x\ge 0.
\end{equation}
Notice that the functions  $(g_k^{-m_k})_{l\le k\le r-1}$ and $g_r\cdot  \ch^{-1}= \ch^{-1+\zeta_{r-1}}$ are positive, even, and non-increasing on $[0,\infty)$ by Lemma~\ref{lem: properties of g_l}.
Consequently,  Lemma~\ref{lem: generalization} can be applied to validate \eqref{eq: target inequality}, which completes the proof.
\end{proof}

We will apply the following lemma to the function $\kappa(x)=\hth^2(x)$.
\begin{lemma}\label{lem: weight reduction lemma: 2}
    Let $1\le \ell\le r-1$ and $s\in \mathscr S$.
     For any measurable function $\kappa:\mathbb R\to [0,\infty)$ that is bounded, even, and non-decreasing on $[0,\infty)$, we have \[\e_h\e_{1,\dots, \ell} W^s_1\dots W^s_\ell \kappa(Y^s_\ell) \ge \e_h\e_1 \kappa( Y^s_1).\]
\end{lemma}
\begin{proof}
    Throughout the proof, we adopt the same notational convention from the proof of Lemma~\ref{lem: weight reduction lemma}, in particular dropping the notation for species $s$.
    We first claim that, for any function $\kappa$ satisfying the assumption of the lemma, conditional on the external field $h$, we have \[ \e_{1,\dots, \ell} W_1\dots W_\ell \kappa( Y_\ell) \ge \e_{1,\dots, \ell} \kappa(Y_\ell).\]
    We use induction on $\ell$.
    From Lemma~\ref{lem: weight W} and Lemma~\ref{lem: properties of g_l}, the base case $\ell=1$ follows from 
    \[ \e_1 g_1(Y_1) \kappa(Y_1) = \e_1 g_1(|Y_1|) \kappa(|Y_1|) \ge \e_1 g_1(|Y_1|) \cdot \e_1 \kappa(|Y_1|)= \e_1 g_1(Y_1)\cdot  \e_1 \kappa(Y_1), \] where the inequality is due to the FKG inequality for the functions $g_1$ and $\kappa$ that are non-decreasing on $[0,\infty)$, thanks to Lemma~\ref{lem: properties of g_l}.
    Suppose the claim holds for some $\ell\ge 1$ and consider the case for $\ell+1$.
    Again by the FKG inequality as in the base case $\ell=1$, we have $\e_{\ell+1}W_{\ell+1}\kappa(Y_{\ell+1})\ge \e_{\ell+1}\kappa(Y_{\ell+1})$ and hence \begin{equation}\label{eq: claim: weight reduction}
        \e_{1,\dots, \ell+1} W_1\dots W_{\ell+1} \kappa( Y_{\ell+1})  \ge \e_{1,\dots, \ell} W_1\dots W_{\ell} \e_{\ell+1}\kappa( Y_{\ell+1}).
    \end{equation}
    Notice that $\e_{\ell+1}\kappa(Y_{\ell+1})$ is $Y_\ell$-measurable and can be written as $\tilde \kappa(Y_\ell)$ for some measurable function $\tilde k$, that is, $\tilde \kappa(x) =\e_{\ell+1}\kappa(x+\eta_{\ell+1})$ for $x\in\mathbb R.$
    By Lemma~\ref{lem: auxiliary lemma: 3}, $\tilde \kappa$ is non-decreasing on $[0,\infty)$, and it is clearly bounded, non-negative, and even by the evenness of $\kappa$ and the symmetry of the random variable $\eta_{\ell+1}$.
    Consequently, we can apply the induction hypothesis to the function $\tilde \kappa$ to obtain \[ \e_{1,\dots, \ell} W_1\dots W_{\ell} \tilde \kappa(Y_\ell) \ge \e_{1,\dots,\ell}\tilde \kappa(Y_{\ell})=\e_{1,\dots,\ell+1} \kappa(Y_{\ell+1}), \] which together with \eqref{eq: claim: weight reduction} finishes the induction and the proof of the claim.

    Our next claim is that, for any function $k$ satisfying the assumption of the lemma, we have \[\e_h \e_{1,\dots \ell}\kappa(Y_\ell)\ge \e_h\e_1 \kappa(Y_1).\] 
    Indeed, since the external field $h$ is centered Gaussian, we have that $Y_\ell$ and $Y_1$ are both centered Gaussians unconditionally, where $\e Y_\ell^2 \ge \e Y_1^2$.
    Therefore, we can write $Y_\ell\stackrel{d}{=}\alpha Y_1$ for some constant $\alpha\ge 1$, and we have $\e \kappa(Y_\ell)= \e \kappa(\alpha Y_1)=\e \kappa(\alpha |Y_1|) \ge \e \kappa(|Y_1|) =\e \kappa(Y_1) $, where the inequality holds because $k$ is non-decreasing on $[0,\infty)$.
    This proves the claim.\footnote{More generally, this claim holds for an external field with a symmetric non-increasing density.}
    Combining the previous two claims completes the proof.
\end{proof}


\subsection{Proof of Proposition~\ref{prop: min and max of support}}

We begin by stating an extension of a standard result regarding the support of a Parisi measure, whose proof will be kept in  Appendix~\ref{sec: Extension of Parisi functional}.
\begin{proposition}[Extension of \cite{AC15_PTRF}, Theorem 5 or  \cite{Tal06}, Proposition 3.2]\label{prop: support of Parisi measure}
  If $\Delta^2$ is positive-definite and $\mu\in \mathcal M^\uparrow$ is a Parisi measure, then \[q=\e_h\Psi_\mu(h,  q) \quad \text{for all} \quad q\in\operatorname{supp}(\mu). \] 
\end{proposition}

    Here, the vector-valued map $\Psi_{\mu}$ is defined as follows.
Consider a modification of  \cite[Equation (18)]{AC15_PTRF} by changing $M(\cdot)$ therein to $x+M(\cdot)$  to accommodate external field $x$. Denote this modification by $\Gamma_{\mu,\xi}(x,u)$, where $\mu$ is a probability measure supported  on $[0,1]$ and $\xi:\mathbb R \to\mathbb R$ is an analytic function defining the model therein, which we can choose for our purpose. 
For each $s\in\mathscr S$, set $\xi_s(x)=2^{-1}(2\Delta^2\Lambda \1)_s x^2$ and define a map $S_s:\mathbb R\to \mathbb R$   by $S_s(x)=(2\Delta^2\Lambda \1)_s x$. 
We also denote the projection to the $s$'th coordinate by $\pi_s:\mathbb R^{\mathscr S}\to \mathbb R$.
   For any $\mu\in\mathcal M^\uparrow$,  we can now define the vector-valued map  \begin{equation}\label{eq: definition of vector valued map Psi} 
        \Psi_\mu(x,p) = \Bigl(\Gamma_{\mu\circ   (2\Delta^2\Lambda)^{-1} \circ \pi_s^{-1}\circ S_s, \; \xi_s}\Bigl(x_s, \frac{(2\Delta^2\Lambda p)_s}{(2\Delta^2\Lambda \1)_s}\Bigr)\Bigr)_{s \in \mathscr S}  \quad\text{for}\quad (x,p)\in   \mathbb R^{\mathscr S}\times[0, 1]^{\mathscr S},
    \end{equation}
    where for any measure $\nu$ and a suitable map $T$, $\nu\circ T^{-1}$ denotes the push-forward and $\nu\circ T$ denotes the pull-back of $\nu.$

Some remarks on \eqref{eq: definition of vector valued map Psi} are in order.
It is well-defined since $\Delta^2 \Lambda \1 \sgee \0$ whenever $\Delta^2$ is positive-definite, and can be viewed as a multi-dimensional generalization of \cite[Equation (3.28)]{Tal06}.
In fact, letting $\Psi_{\nu,V}$  be \cite[Equation (3.28)]{Tal06} for a probability measure $\nu$ supported on $[0,V]$ for some $V>0$, we have $\Psi_{\mu\circ (\xi')^{-1}, \xi'(1)} (x,\xi'(u))= \Gamma_{\mu,\xi}(x,u)$ for any $(x,u)\in \mathbb R\times [0,1]$, any probability measure $\mu$ supported on $[0,1]$, and any $\xi.$
Lastly, for a discrete measure $\mu \in \mathcal M^\uparrow_d$ corresponding to \eqref{eq: zeta sequence} and \eqref{eq: q sequence}, i.e., $\mu(q_l)= \zeta_l-\zeta_{l-1}$ for $0\le l\le r$, we can explicitly write \eqref{eq: definition of vector valued map Psi}.
Indeed, by comparing \eqref{eq: derivative of X^s_0 in Q^s}  with  \cite[Equation (3.27)]{Tal06}, for any $s \in \mathscr S$ and $1\le l\le r-1$, \begin{equation}\label{eq: Psi for discrete measure}
    \e_h\Psi_\mu (h, q_l )_s= \e_h\e_{1,\dots, l} W^s_1\dots  W^s_l \bigl(\e_{l+1,\dots,r} W^s_{l+1}\dots W^s_{r} \hth  Y^s_r \bigr)^2  .
\end{equation} 

The proof of Proposition~\ref{prop: min and max of support} is as follows.
    We can find a sequence $(\mu_n)_{n\ge 1}$ in $\mathcal M_d^\uparrow$ such that $\mu_n \stackrel{w}{\to}\mu$ as $n\to\infty$ while keeping $q^{\mu_n}_{\text{min}} =q^\mu_{\text{min}}$ and $q^{\mu_n}_{\text{max}} =q^\mu_{\text{max}}$ for all $n\ge 1.$    
    
    Let us begin by establishing the first two assertions of the proposition.
    Fix $n\ge 1$ and consider the sequences \eqref{eq: zeta sequence} and \eqref{eq: q sequence} associated to $\mu_n.$
    Fix $n\ge 1$ and consider the sequences \eqref{eq: zeta sequence} and \eqref{eq: q sequence} associated to $\mu_n.$
    Without loss of generality, we may assume  $\zeta_0=0$ and $\zeta_{r-1}=1$.
  Then,  \[q^\mu_{\text{min}}=q^{\mu_n}_{\text{min}}=q_1 \quad\text{and}\quad q^\mu_{\text{max}}=q^{\mu_n}_{\text{max}}=q_{r-1}.\]
    
We first argue that  \begin{equation}\label{eq: claim 1}
    \e\Psi_{\mu_n}(h,  q_{r-1})\gee \e \hth^2(h+z\sqrt{ 2 \Delta^2\Lambda q_{r-1}}).
\end{equation}
To this end, observe that by \eqref{eq: Psi for discrete measure}, the preceding display is equivalent to \begin{equation}\label{eq: surrogate claim 1}
    \e_h\e_{1,\dots, l} W^s_1\dots  W^s_{r-1} (\e_{r} W^s_{r} \hth  Y^s_r )^2  \ge  \e \hth^2(h_s+z\sqrt{ 2 (\Delta^2\Lambda q_{r-1})_s}),\quad  \forall\,  s\in\mathscr S.
\end{equation}
Since $\zeta_{r-1}=1$,  we have $W^s_r= \ch Y^s_r / \e_r \ch Y^s_r$ and \[ \e_r W^s_r \hth(Y^s_r) = \frac{\e_{r}\sh Y^s_{r}}{\e_{r}\ch Y^s_{r}}  = \hth (Y^s_{r-1}),\] where the last equality follows from the elementary calculations $\e \sh (a+bz) =\sh(a)\exp(b^2/2) $ and $\e \sh (a+bz) =\sh(a)\exp(b^2/2) $ for any $a,b\in\mathbb R$ and standard normal $z$.
Then, \eqref{eq: surrogate claim 1} readily follows from Lemma~\ref{lem: weight reduction lemma: 2} for  $\kappa(x)=\hth^2(x)$ and $l=r-1$.
 On the other hand,   by \cite[Propsotision 1-(ii)]{AC15_PTRF}, $\lim_{n\to\infty }\Psi_{\mu_n} =\Psi_\mu$ uniformly.
  Then, Proposition~\ref{prop: support of Parisi measure} and \eqref{eq: claim 1} imply \begin{equation*}
      q^\mu_{\text{max}} = \e \Psi_\mu (h,  q^\mu_{\text{max}}) = \lim_{n\to\infty} \e \Psi_{\mu_n} (h,  q^\mu_{\text{max}}) \gee  \e \hth^2(h+z\sqrt{2\Delta^2 \Lambda q^\mu_{\text{max}}}).
  \end{equation*}
From this and Lemma~\ref{lem: monotone region lemma}, we deduce either $q^*= q^\mu_{\mathrm{max}}$ or $q^*\slee q^\mu_{\mathrm{max}}$.
    
    In a similar manner, we argue that \begin{equation}\label{eq: claim 2}
        \e\Psi_{\mu_n}(h,  q_1)\lee \e \hth^2(h+z\sqrt{2\Delta^2\Lambda q_1}).
    \end{equation}
    To this end, note that, for any $s\in\mathscr S$, we have  $W^s_1=1 $ since $\zeta_0=0$. 
    Combining this  with Lemma~\ref{lem: weight reduction lemma} for $l=2$, we obtain  \[ \e \Psi_{\mu_n}(h,q_1)_s = \e_h \e_1 (\e_{2,\dots, r} W^s_2\dots W^s_r \hth(Y^s_r))^2 \le \e_h\e_1 \hth^2(Y^s_1),\] which immediately validates \eqref{eq: claim 2}.
 Sending $n\to\infty$ in \eqref{eq: claim 2}, we have \[q^\mu_{\text{min}}=\e \Psi_{\mu}(h,q^\mu_{\text{min}}) =\lim_{n\to\infty} \e\Psi_{\mu_n}(h,q^\mu_{\text{min}}) \lee \e \hth^2(h+z (2\Delta^2\Lambda q^\mu_{\text{min}})^{1/2}).\] Lemma~\ref{lem: monotone region lemma} implies either $q^*=q^\mu_{\text{min}}$ or $q^*\sgee q^\mu_{\text{min}}$.
 The first two assertions have been verified.

 Finally, let us prove the last assertion of the proposition.
 Suppose $\tau^2\sgee \0.$
By Jensen's inequality, \[\e \Psi_{\mu_n}(h,q_1)_s = \e_h \e_1 (\e_{2,\dots, r} W^s_2\dots W^s_r \hth(Y^s_r))^2 \ge \e_h  (\e_1\e_{2,\dots, r} W^s_2\dots W^s_r \hth(Y^s_r))^2.\]
From the proof of \cite[Lemma 14.7.15]{Talagrand2013vol2}, there exists a constant $c=c(\Lambda,\Delta^2,\tau^2)>0$ depending only on the model parameters $(\Lambda,\Delta^2,\tau^2)$ such that (recall $W^s_1=1$) \[|\e_1\e_{2,\dots, r} W^s_2\dots W^s_r \hth(Y^s_r)| \ge c |\hth(h_s)|.\] 
Combining the last two displays, we arrive at, for any $n\ge 1$ and $s\in\mathscr S$, \[\e \Psi_{\mu_n}(h,q_1)_s \ge c \, \e_h \hth^2(h_s)>0, \] where the last inequality is due to $\tau^2_s>0.$
Invoking the uniform convergence $\Psi_{\mu_n} \to \Psi_\mu$ and Propposition~\ref{prop: support of Parisi measure}, we can send $n\to\infty$ in the last display and, since $c \, \e_h \hth^2(h_s)$ does not depend on $n$, we obtain the desired result $q^\mu_{\min}\sgee \0$.
The proof is now complete.



\section{Interpolation and the Aizenman--Sims--Starr scheme}\label{sec: interpolation and  Aizenman-Sims-Starr cavity computation}
In this section, we carry out the rest of Step 4 in the proof sketch, assuming that the free energy cost \eqref{eq: free energy cost: overview} stated in Step 3 is valid.
For  positive-semidefinite $\Delta^2$, $\tau^2\gee \0$, and $t\in [0,1]$, consider the interpolating system with variance profile $(t\Delta^2, \tau^2+ (1-t)2\Delta^2\Lambda q^*(\Delta^2,\tau^2))$, which corresponds to the Hamiltonian \begin{equation} \label{eq: interpolating Hamiltonian}
       H_{N,t}(\sigma)=   \frac{ \sqrt{t}}{\sqrt{N}}\sum_{1\le i,j\le N} J_{ij}  \sigma_i\sigma_j  
            + \sqrt{1-t}  \sum_{s\in\mathscr S}(2(\Delta^2\Lambda q^*(\Delta^2,\tau^2))_s)^{1/2} \sum_{i\in I_s}z_i\sigma_i
            + \sum_{1\le i\le N}h_{i} \sigma_i
\end{equation} and the minimum of the RS functional, which we later show in Lemma~\ref{lem: stability of q^*} to be 
              \begin{equation} \label{eq: interpolated approximating RS functional}
                  \mathsf{RS}^* (t) \colonequals \log 2 +  \sum_{s \in \mathscr S} \lambda_{s}\e\log \ch\bigl(z \sqrt{\tau^2_s+ 2(\Delta^2 \Lambda q^*(\Delta^2,\tau^2))_s } \bigr)+\frac{t}{2} B(\1 -q^*(\Delta^2,\tau^2)),
\end{equation}   where $B(x)\colonequals B(x,x)$ is the associated quadratic form of the bilinear form   \[B(x,y)\colonequals x^{\intercal} \Lambda\Delta^2\Lambda y \quad \text{for} \quad x,y\in [-1,1]^{\mathscr S}. \] 
Denote the region that is coordinatewise ordered relative  to $q^*(\Delta^2,\tau^2)$ by
\begin{equation}
    \mathsf T = \{u\in[0,\infty)^{\mathscr S}:   q^*(\Delta^2,\tau^2) \lee  u \text{ or }  q^*(\Delta^2,\tau^2) \gee  u\}.
     \label{eq: definition of T}
\end{equation}
The following proposition  supplies a  sufficient condition to ensure replica symmetric solutions.
\begin{proposition}\label{prop: sufficient condition for RS}
     Let  $\Delta^2$ be irreducible and positive-definite. 
     Let $\tau^2\sgee \0.$
        Suppose there exists some constant $C=C(\Lambda, \Delta^2,\tau^2)>0$ such that  for any $u\in \mathsf T $, $t\in [0,1]$, \[ \frac{1}{N} \e \log \sum_{R_{12}=u} \exp ( H_{N,t}(\sigma^1) +H_{N,t}(\sigma^2) ) \le 2\mathsf{RS}^*(t) - C B(u-q^*) +o(1),\] where $o(1)\to 0$ as $N\to\infty$,  uniformly in $u$.
        Then, the model exhibits replica symmetric solution.
\end{proposition}
From the given free energy cost, under the interpolated Gibbs measure, if $R_{12}$ is in $\mathsf{T}$, then  it must  be close to $q^*.$
Hence, the main task of the proof of this proposition is to remove the constraint $ \mathsf{T}$, for which we invoke ASS scheme as in the following subsection.


\subsection{The ASS scheme along the interpolation}
We first record an important feature of the Gaussian interpolation: the quantity $q^*(t\Delta^2, \tau^2+(1-t)2\Delta^2\Lambda q^*(\Delta^2,\tau^2))$ is constant for all $t\in[0,1]$.
Later in the proof of the main theorem, this will allow us to transfer the GT 1RSB bound established in Section~\ref{sec: Two-dimensional Guerra--Talagrand bound} for the variance profile $(\Delta^2,\tau^2)$ to the new profile $(t\Delta^2, \tau^2+(1-t)2\Delta^2\Lambda q^*(\Delta^2,\tau^2))$.
\begin{lemma}\label{lem: stability of q^*}
    For any irreducible $\Delta^2$ and  $t\in (0,1]$, we have  \[q^*(t\Delta^2, \tau^2+(1-t)2\Delta^2\Lambda q^*(\Delta^2,\tau^2))= q^*(\Delta^2,\tau^2).\]
    In particular, if we further assume that $\Delta^2$ is positive-semidefinite, then \eqref{eq: interpolated approximating RS functional} is precisely the minimum of the RS functional at any $t\in [0,1]$ along the interpolation.
\end{lemma}
\begin{proof}
    Recall Definition \ref{def: definition of q^* for irreducible}.
    Denote the interpolated variance profile at $t\in (0,1]$ by \[(\tilde \Delta^2, \tilde \tau^2)  =  (t\Delta^2,\tau^2+(1-t)2\Delta^2 \Lambda q^*(\Delta^2,\tau^2)).\]
    We claim that $q^*(\tilde \Delta^2,\tilde \tau^2)=q^*(\Delta^2,\tau^2)$.
    Suppose $q^*(\Delta^2,\tau^2)=\0$. 
    By Proposition \ref{prop: uniqueness of fixed-point}, $\tau^2=\0$ and $\rho(\Delta^2\Lambda)\le 1/2$.
    In this case, $(\tilde \Delta^2, \tilde \tau^2)= (t\Delta^2, \0)$ and $\rho(\tilde \Delta^2 \Lambda) = t \rho(\Delta^2\Lambda)\le 1/2$, so by Proposition \ref{prop: uniqueness of fixed-point}, $q^* (\tilde \Delta^2,\tilde \tau^2)=\0$.
    Suppose $q^*(\Delta^2,\tau^2) \ne \0.$
    Then, the fixed-point equation for $(\tilde\Delta^2, \tilde\tau^2)$ reads
    \[ q= \e \hth^2\bigl( z \bigl( \tau^2+ (1-t)2\Delta^2\Lambda q^*(\Delta^2,\tau^2) +2  t\Delta^2\Lambda q\bigr)^{1/2}\bigr),\]
    in which  $q=q^*(\Delta^2,\tau^2)$ is an obvious nonzero solution.
    Then, Proposition \ref{prop: uniqueness of fixed-point} implies the desired claim.
    Finally,  for $t\in (0,1]$, the claim implies $ \tilde \tau^2 +2\tilde \Delta^2 \Lambda q^*(\tilde \Delta^2, \tilde \tau^2)  = \tau^2 +2 \Delta^2 \Lambda q^*(\Delta^2,  \tau^2) $, from which \eqref{eq: interpolated approximating RS functional} is easily seen to be the minimum by Lemma~\ref{lem: infimum of RS: semidefinite}.
    At $t=0,$ the RS functional \eqref{eq: RS functional} does not depend on $q$, so \eqref{eq: interpolated approximating RS functional} is  trivially the minimum.

\end{proof}

We now explain the ASS scheme.
First of all, since $q^*(t\Delta^2,\tau^2+(1-t)2\Delta^2\Lambda q^*(\Delta^2,\Lambda))$ is constant along the interpolation by Lemma~\ref{lem: stability of q^*}, we can simply denote it by $q^*$.
Since the Parisi formula depends continuously on the model parameters (see Appendix~\ref{sec: Proof of continuity of Parisi formula in model parameters}),  we may assume that  $\lambda$'s are rational: for each $s\in\mathscr S$, there exists $k_s \in \mathbb N$ such that $\lambda_s= k_s/k$ where \[k\colonequals \sum_{s\in\mathscr S}k_s.\]
For $N\in\mathbb N$, let $(\tilde h_{N,x}(\sigma))_{\sigma\in\{\pm 1\}^N}$ be the perturbing  Gaussian process  in \cite[Equation (27)]{Pan15}, but let us assume it involves a deterministic infinite array $x\in [1,2]^{\mathbb N\otimes \mathbb N}$ rather than $x$ having i.i.d. $\mathrm{Uniform}[1,2]$ entries.
Whenever the array $x$ consists of i.i.d. $\mathrm{Uniform}[1,2]$  random variables, we will denote it by $U$ instead of $x.$
Let $s_N=N^{\gamma}$ for some fixed $4^{-1} < \gamma <2^{-1}$.
For $(\Delta^2,\tau^2)$, $t\in[0,1]$, and deterministic $x\in [1,2]^{\mathbb N\otimes\mathbb N}$, define the modified Hamiltonian\footnote{The ``hat'' stands for the modification of the normalization by $k$ and ``pert'' stands for the perturbation $\tilde h_{N,x}$.}
\begin{align}
\hat H_{N,t,x}^{\mathrm{pert}} (\sigma)&= \frac{\sqrt{t}}{\sqrt{N+k}}\sum_{1\le i,j\le N}J_{ij} \sigma_i\sigma_j+\sqrt{1-t}  \sum_{s\in\mathscr S}(2(\Delta^2\Lambda q^*)_s)^{1/2} \sum_{i\in I_s}z_i\sigma_i+ \sum_{1\le i\le N}h_{i} \sigma_i  \nonumber
           \\&\quad +s_N \tilde  h_{N,x} (\sigma) , \label{eq: hat pert interpolating Hamiltonian}
    \end{align}    where $(z_i)_{1\le i\le N}$ are i.i.d. standard Gaussian, the variance profiles of $J$ and $h$ are given by \eqref{eq: variance of random variables}, and all the random variables are independent.
    We will keep the dependence on the array $x$ explicit in the subscript.
    Define the Gibbs measure associated to the modified Hamiltonian  by \[\hat G_{N,t,x}^{\mathrm{pert}}(\sigma) = \frac{\exp \hat H_{N,t,x}^{\mathrm{pert}}(\sigma)}{\sum_{\sigma\in\{\pm 1\}^N} \exp( \hat H_{N,t,x}^{\mathrm{pert}}(\sigma))}, \quad \sigma\in \{\pm 1\}^N,\]
and we denote the integration with respect to the measure $(\hat G_{N,t,x}^{\mathrm{pert}})^{\otimes \infty}$ on $(\{\pm 1\}^N)^{\otimes \infty}$ by $\lla \cdot \rra_{N,t,x}^{\mathrm{hat,pert}}$.
Let $D_{N,t,x}(f,n,w,p)$ denote \cite[Equation (31)]{Pan15} under the Gibbs measure $\hat G_{N,t,x}^{\mathrm{pert}}$.
Precisely, it is given by, for  any Borel measurable function $|f|\le 1$, $n\ge 2$, $w\in\mathscr W$, and $p\ge 1$   for some countable dense subset $\mathscr W\subseteq [0,1]^{\mathscr S}$,\begin{align*}
    D_{N,t,x}(f,n,w,p)=  \e \lla f(R^n) R_w (\sigma^1,\sigma^{n+1})^p\rra_{N,t,x}^{\mathrm{hat,pert}} &-\frac{1}{n}\e \lla f(R^n)\rra_{N,t,x}^{\mathrm{hat,pert}} \lla  R_w (\sigma^1,\sigma^{n+1})^p\rra_{N,t,x}^{\mathrm{hat,pert}}  
    \\&-\frac{1}{n}\sum_{\ell =2}^n \e \lla f(R^n) R_w (\sigma^1,\sigma^{\ell})^p\rra_{N,t,x}^{\mathrm{hat,pert}}  ,
\end{align*} 
where $R^n\colonequals (R(\sigma^\ell,\sigma^{\ell'})_s)_{s\in\mathscr S, 1\le \ell,\ell'\le n}$ and $R_w(\sigma^1,\sigma^2)  \colonequals \sum_{s\in\mathscr S}\lambda_s w_s R(\sigma^1,\sigma^2)_s$.
For each $n\ge 1,$ define $\mathcal P_n$ to be the countable set of all multivariate  monic polynomials on $[-1,1]^{|\mathscr S|\times n^2}$.
As in \cite[Lemma~3.3]{Pan13_SK_book}, for sufficiently large $N\ge 1$, there exist Borel sets $A_{N,t}\subseteq [1,2]^{\mathbb N \otimes \mathbb N}$ such that for any   $n \ge 2$, $f\in\mathcal P_n$,  $w \in \mathscr W$, and $p \ge 1$,  we have $\lim_{N\to\infty }\p_{U}(A_{N,t}) =1$ and
\begin{align}\label{eq: choosing nonrandom x}
    \lim_{N\to\infty}  D_{N,t,x_{N,t}}(f,n,w,p)&=0 \text{ along any sequence $\{x_{N,t}\}_{N\geq 1}$ with $x_{N,t}\in A_{N,t}$ for all $N\geq 1.$}  
\end{align} 
Consider any non-random sequence $(x_{N,t})_{N\ge 1}$ in \eqref{eq: choosing nonrandom x}.
By \cite[Theorem 3]{Pan15}, any subsequential limiting array  of the overlap array $(R(\sigma^\ell,\sigma^{\ell'})_s)_{s\in\mathscr S, \ell,\ell' \ge 1}$ sampled from $\e (\hat G_{N,t,x_{N,t}}^{\mathrm{pert}})^{\otimes \infty}$ satisfies the Ghirlanda--Guerra identities (GGI).\footnote{
We note here that the following technical approximation argument is used after taking the limit in \eqref{eq: choosing nonrandom x}; for any complete Borel probability measure on a compact Hausdorff metrizable space,  any bounded Borel measurable function is an almost sure limit of continuous functions (cf. \cite[Theorem 2.18, Corollary of Theorem 2.24]{Rud_RCA}), and any continuous function on $[-1,1]^d$ for any $d\ge 1$ is uniformly approximated by polynomials.}
Moreover, the ASS scheme in \cite[Section 5]{Pan15} implies \begin{equation}\label{eq: consequence of ASS scheme}
    \lim_{N\to\infty}  \e F_{N}(t) =\mathscr P_t(\mu) = \inf_{\nu\in \mathcal M^\uparrow} \mathscr P_t(\nu), \quad \forall \mu \in \mathcal N_t, 
\end{equation} where  $\mathscr P_t$ is the Parisi functional with the model parameters $(\Lambda,t\Delta^2, \tau^2+(1-t)2\Delta^2 \Lambda q^*(\Delta^2,\tau^2))$ and \[\mathcal N_t\colonequals \{\text{subsequential weak limits of the sequence } (\e( \hat G_{Nk,t,x_{N,t}}^{\mathrm{pert}})^{\otimes 2}(R_{12 }\in \cdot ))_{N\ge 1}\}.\]
Indeed, $\mathscr P_t(\mu)$ is well-defined for $\mu\in \mathcal N_t$ as follows.
\begin{lemma}\label{lem: Talagrand positivity}
For any $0\le t\le 1$, we have  $\mathcal N_t\subseteq \mathcal M^\uparrow.$    
\end{lemma}
\begin{proof}
    Any  subsequential weak limit of the sequence of array $(\e (\hat G_{Nk,t,x_{N,t}}^{\mathrm{pert}})^{\otimes \infty}(R(\sigma^\ell,\sigma^{\ell'})\in \cdot ))_{ \ell,\ell'\ge 1}$ satisfies the GGI by \eqref{eq: choosing nonrandom x}.
      By the GGI, Talagrand's positivity principle holds exactly in the limit and, therefore, all members of $\mathcal N_t$ is supported on $[0,1]^{\mathscr S}$.
    Moreover, by \cite[Theorem 4]{Pan15}, for any $\mu\in \cup_{0\le t\le 1}\mathcal N_t$, there exist Lipschitz and non-decreasing functions $L_s^\mu: [0,1]\to [0,1]$ for each species $s \in \mathscr S$  such that   \[\operatorname{supp}(\mu) \subseteq \{ (L_s^\mu(a))_{s \in \mathscr S}:0\le a\le 1\}.\]
    The right-hand side of the preceding display has a totally ordered trajectory.
    Therefore, $\operatorname{supp}(\mu)$ is totally ordered and the proof is complete.
\end{proof}

\subsection{Proof of Proposition~\ref{prop: sufficient condition for RS}}

We start by showing that the modification of the Hamiltonian cannot be detected in the limiting free energies.
For  $N\in\mathbb N$, $t\in[0,1]$, $x\in[1,2]^{\mathbb N\otimes\mathbb N}$, and $u\in [-1,1]^{\mathscr S}$, denote the free energies corresponding to the Hamiltonians \eqref{eq: interpolating Hamiltonian} and \eqref{eq: hat pert interpolating Hamiltonian},  respectively,  by $\e F_N(t)$ and $\e \hat F_{N,x}^{\mathrm{pert}}(t)$, 
and also the constrained free energies by
    \begin{align}
  \psi_{N }(t,u)&= \frac{1}{N} \e \log \sum_{R_{12 }= u} \exp(H_{N,t}(\sigma^1)+H_{N,t}(\sigma^2)), \label{eq: definition of constrained free energy}
  \\ \hat \psi_{N,x}^{\mathrm{pert}} (t,u)&= \frac{1}{N}  \e \log \sum_{R_{12 }= u} \exp(\hat H_{N,t,x}^{\mathrm{pert}}(\sigma^1)+\hat H_{N,t,x}^{\mathrm{pert}}(\sigma^2)). \nonumber
    \end{align}            

\begin{lemma}\label{lem: hat in free energies}
We have
 \[\limsup_{N\to\infty}\sup_{0\le t\le 1}\sup_{x\in [1,2]^{\mathbb N\otimes\mathbb N}} |  \e \hat F^{\mathrm{pert}}_{N,x}(t)- \e F_N(t)| = 0\]
 and \[\limsup_{N\to\infty}\sup_{u\in [-1,1]^{\mathscr S}} \sup_{0\le t \le 1} \sup_{x\in [1,2]^{\mathbb N\otimes\mathbb N}} |\hat \psi_{N,x}^{\mathrm{pert}}(t,u)- \psi_{N}(t,u)|=0.\]
\end{lemma}
\begin{proof}
Define an intermediate Hamiltonian $ H_{N,t,x}^{\mathrm{pert}}(\sigma)=    H_{N,t} (\sigma) + s_N \tilde h_{N,x} (\sigma)$ and the corresponding free energies
\begin{align*}
        \phi_{N,x}^{\mathrm{pert}}(t) &= \frac{1}{N} \e \log \sum_{\sigma\in \{\pm 1\}^N} \exp(H_{N,t}^{\mathrm{pert}}(\sigma)), 
 \\\psi_{N,x }^{\mathrm{pert}} (t,u)&= \frac{1}{N}\e \log \sum_{R_{12 }= u} \exp(H_{N,t,x}^{\mathrm{pert}}(\sigma^1)+H_{N,t,x}^{\mathrm{pert}}(\sigma^2)).
\end{align*}
By \cite[Inequality (29)]{Pan15},  Jensen's inequality and $s_N=o(N^{1/2})$ imply \begin{align}
    \limsup_{N\to\infty}\sup_{0\le t \le 1} \sup_{x\in [1,2]^{\mathbb N\otimes\mathbb N}} |\phi_{N,x}^{\mathrm{pert}}(t) -\e F_{N}(t)|&=0, \label{eq: limiting perturbed free energy}
    \\\limsup_{N\to\infty}\sup_{0\le t \le 1}\sup_{u\in [-1,1]^{\mathscr S}}\sup_{x\in [1,2]^{\mathbb N\otimes\mathbb N}}|\psi_{N,x}^{\mathrm{pert}}(t,u) -\psi_{N}(t,u)| &= 0. \label{eq: limiting perturbed constrained free energy}
\end{align}
Fix $t\in [0,1]$. 
Define a mean-zero Gaussian process \[y_{N,t}(\sigma)= \frac{\sqrt{kt}}{\sqrt{N(N+k)}}\sum_{1\le i,j\le N}J'_{ij}\sigma_i\sigma_j, \quad \sigma\in\{\pm 1\}^N,\] where $J'$ is a copy of $J$, independent of everything else.
Then, as processes on $\{\pm 1\}^N$, we can write $(H_{N,t,x}^{\mathrm{pert}}(\sigma))_{\sigma \in \{\pm 1\}^N } \stackrel{d}{=} (\hat H_{N,t,x}^{\mathrm{pert}}(\sigma) +y_{N,t}(\sigma))_{\sigma \in \{\pm 1\}^N}$; see also \cite[Equation (60)]{Pan15}.
Note that $y_{N,t}$ is of a smaller order   \[\sup_{N,t,\sigma}\e y_{N,t}(\sigma)^2 <\infty.\] 
On the other hand, by Jensen's inequality, \[
    \e \log \sum_{\sigma} \exp H_{N,t,x}^{\mathrm{pert}}(\sigma) 
    \le \e \log \sum_{\sigma} \exp \hat H_{N,t,x}^{\mathrm{pert}}(\sigma) 
    \le  \e \log \sum_{\sigma} \exp H_{N,t,x}^{\mathrm{pert}}(\sigma) +   \sup_{N,t,\sigma}\e y_{N,t}(\sigma)^2.\]
The last two displays together imply \[\limsup_{N\to\infty} \sup_{0\le t\le 1}\sup_{x\in [1,2]^{\mathbb N\otimes\mathbb N}} | \e \hat F^{\mathrm{pert}}_{N,x}(t) - \phi_{N,x}^{\mathrm{pert}}(t)| =0.\]
   Combining this with \eqref{eq: limiting perturbed free energy}, we arrive at the first assertion of the lemma.
   
   The second assertion of the lemma can be verified similarly, using the same process $y_{N,t}$ and \eqref{eq: limiting perturbed constrained free energy}.
 
\end{proof}

Define $B_N$ and $\tilde B_N$ to be versions of $B$ in which  $\Lambda\Delta^2 \Lambda$ is replaced with $\Lambda_N \Delta^2 \Lambda_N$ and $\Lambda_N\Delta^2 \Lambda$, respectively.
The following uniform convergences are immediate from their definitions:
\begin{equation}\label{eq: B_N converge uniformly}
   \begin{aligned}
   &\lim_{N\to\infty} \sup_{x,y\in[-1,1]^{\mathscr S}}|B_N(x,y)-B(x,y)|\le \lim_{N\to\infty}\sum_{s,t\in\mathscr S}\Delta^2_{st}|\lambda_{s,N}\lambda_{t,N}-\lambda_s\lambda_t| = 0,   
   \\& \lim_{N\to\infty} \sup_{x,y\in[-1,1]^{\mathscr S}}|\tilde B_N(x,y)-B(x,y)|\le \lim_{N\to\infty}\sum_{s,t\in\mathscr S}\Delta^2_{st}|\lambda_{s,N}-\lambda_s|\lambda_t = 0.
   \end{aligned}
\end{equation}
Throughout the rest of the article, we use \( o(1) \) to denote a quantity that vanishes as \( N \to \infty \), but which may depend on other parameters.

We carry out two straightforward calculations about the interpolated free energies.
\begin{lemma}\label{lem: RS equals perturbed free energy when t=0}
    For any $x \in [1,2]^{\mathbb N\otimes \mathbb N}$, we have  \[\e \hat F_{N,x}^{\mathrm{pert}}(0)= \mathsf{RS}^*(0)  +o(1),\] where $o(1)$ is uniform in $x$.
    In particular, the same holds if we replace $x$ by uniform $U$.
\end{lemma}
\begin{proof}
     In the absence of the perturbing Hamiltonian $\tilde h_{N,U}$, we can explicitly calculate \begin{align*}
        \e F_{N}(0)&= \frac{1}{N} \e \log \sum_{\sigma} \exp\Bigl(\sum_{s\in\mathscr S}((2\Delta^2\Lambda q^*)_s)^{1/2} \sum_{i\in I_s}z_i\sigma_i
     +\sum_{1\le i\le N}h_{i} \sigma_i \Bigr)
     \\&= \log 2 + \frac{1}{N} \e \log \frac{1}{2^N}\sum_{\sigma} \prod_{1\le i\le N} \exp \Bigl(  \sigma_{i} \bigl(z_{i}   ((2\Delta^2\Lambda q^*)_{s(i)})^{1/2}
     +h_{i} \bigr) \Bigr)
     \\&= \log 2 + \frac{1}{N}\sum_{1\le i\le N} \e \log  \ch  \Bigl(  z_i  ((2\Delta^2\Lambda q^*)_{s(i)})^{1/2}
     +h_{i}  \Bigr)
     \\&= \log 2 + \sum_{s \in \mathscr S}\lambda_{s,N}  \e \log  \ch  \bigl(  z    ((2\Delta^2\Lambda q^*)_{s} +\tau^2_s)^{1/2}
      \bigr)
     \\&=   \mathsf{RS}^*(0) +o(1).
    \end{align*} 
    Combining the preceding display with Lemma~\ref{lem: hat in free energies}, we obtain \[\limsup_{N\to\infty}\sup_{x\in [1,2]^{\mathbb N\otimes\mathbb N}}|\mathsf{RS}^*(0)-\e \hat F_{N,x}^{\mathrm{pert}}(0)|=0.\]
\end{proof}

\begin{lemma}    \label{lem: derivative of interpolating free energy: hat version}
For $0< t< 1$ and $x\in [1,2]^{\mathbb N\otimes \mathbb N}$, \[2\frac{d}{dt} \e \hat F_{N,x}^{\mathrm{pert}}(t) = B(\1-q^*) -\e\biglla B(R_{12}-q^*)\bigrra_{N,t,x}^{\mathrm{hat,pert}}+o(1). \] Here, the vanishing quantity $o(1)$ is uniform in $t$ and $x.$
In particular, the same holds if we replace $x$ by uniform $U$.
\end{lemma}
\begin{proof}
 A direct differentiation yields \begin{align*}
        2N \frac{d}{dt}\e \hat F_{N,x}^{\mathrm{pert}} (t)&= \e \Biglla \frac{1}{\sqrt{(N+k) t}} \sum_{1\le i,j\le N} J_{ij}  \sigma_i\sigma_j \Bigrra_{N,t,x} ^{\mathrm{hat,pert}}
         \\&\quad - \e \Biglla   \frac{1}{\sqrt{1-t}}\sum_{1\le i\le N} z_{i} (2(\Delta^2\Lambda q^*)_{s(i)})^{1/2}\sigma_i \Bigrra_{N,t,x} ^{\mathrm{hat,pert}}.
    \end{align*}
    By Gaussian IBP, 
    \begin{align*}
    2 N \frac{d}{dt}\e \hat F_{N,x}^{\mathrm{pert}}(t) &=   \frac{1}{N+k}\e \Biglla\sum_{1\le i,j\le N}\bigl(1- \sigma^1_i\sigma^1_j \sigma^2_i\sigma^2_j\bigr) \Delta^2_{s(i)s(j)}\Bigrra_{N,t,x}^{\mathrm{hat,pert}}
    \\&\quad- 2\e \Biglla \sum_{1\le i\le N} (1- \sigma^1_i \sigma^2_i) (\Delta^2 \Lambda q^*)_{s(i)} \Bigrra_{N,t,x}^{\mathrm{hat,pert}}
    \\&=\frac{N^2}{N+k}\e \Biglla \sum_{s,s' \in \mathscr S} \lambda_{s,N} \Delta^2_{ss'} \lambda_{s',N} \bigl(1-(R_{12})_s (R_{12})_{s'}\bigr) \Bigrra_{N,t,x}^{\mathrm{hat,pert}}
    \\&\quad- 2N\e \Biglla \sum_{s \in \mathscr S} \lambda_{s,N} (\Delta^2\Lambda q^*)_{s} (1-(R_{12})_s)   \Bigrra_{N,t,x}^{\mathrm{hat,pert}}
\end{align*}
Rearranging, \eqref{eq: B_N converge uniformly} implies \begin{align*}
    2\frac{d}{dt} \e \hat F_{N,x}^{\mathrm{pert}}(t)  &=\e\biglla  B(\1) - B(R_{12}) -2 B(\1, q^*) +2 B(R_{12}, q^*)\bigrra_{N,t,x}^{\mathrm{hat,pert}}+o(1)
    \\&= \e\biglla  B(\1)-2 B(\1, q^*) +B(q^*) - B(R_{12})  +2 B(R_{12}, q^*) -B(q^*)\bigrra_{N,t,x}^{\mathrm{hat,pert}}+o(1)
    \\&= \e\biglla  B(\1-q^*) - B(R_{12}-q^*)\bigrra_{N,t,x}^{\mathrm{hat,pert}}+o(1)
    \\&= B(\1-q^*) -\e\biglla B(R_{12}-q^*)\bigrra_{N,t,x}^{\mathrm{hat,pert}} +o(1).
\end{align*}

\end{proof}

The following is a simple consequence of the concentration of Gaussian measure.
\begin{lemma}[Adaptation of Proposition 13.4.3, \cite{talagrand2004}]\label{lem: cost of free energy implies unlikely}
    Let $t\in[0,1]$, $u\in[-1,1]^{\mathscr S}$, $x\in [1,2]^{\mathbb N\otimes \mathbb N}$, and $N\in\mathbb N$.
    If $\hat \psi_{N,x}^{\mathrm{pert}}(t,u)\le 2\e \hat F_{N,x}^{\mathrm{pert}}(t) -\eps$ for some $\eps>0$, then \[\e (\hat G_{N,t,x}^{\mathrm{pert}})^{\otimes 2} (R_{12}=u) \le Ce^{-CN}\] for some constant $C=C(\eps,\Lambda, \Delta^2,\tau^2)>0$ independent of $(t,x,u)$ in particular.
\end{lemma}
\begin{proof}
    The proof is identical to the proof of \cite[Proposition 13.4.3]{talagrand2004}.
    The only nontriviality is  to check that  $C$ is independent of $x$ and $t$.
    To this end, from Gaussian concentration,  it suffices to show that  \[\limsup_{N\to\infty}\frac{1}{N} \sup_{t,x,\sigma} \e \hat H^{\mathrm{pert}}_{N,t,x}(\sigma)^2 <\infty.\]
    Since $\sup_{x\in [1,2]^{\mathbb N\otimes \mathbb N}}\e \tilde h_{N,x}(\sigma)^2\le 4$ and $s_N=o(N^{1/2})$, we have    
    \begin{align*}
    N^{-1}\e \hat H^{\mathrm{pert}}_{N,t,x}(\sigma)^2 &\le t  \sum_{s,s'\in\mathscr S } \lambda_{s,N}\Delta^2_{s,s'}\lambda_{s',N} +(1-t)  \sum_{s\in\mathscr S}(2 \Delta^2\Lambda q^*)_s\lambda_{s,N} + \sum_{s\in\mathscr S}\tau^2_s\lambda_{s,N}+4N^{-1}s_N^2
    \\&\le c+o(1)
\end{align*} 
for some $c=c(\Lambda, \Delta^2,\tau^2)>0$ and a vanishing quantity $o(1)$, both independent of $t, x,\sigma$.
This finishes the proof.
\end{proof}

For $r\ge 0$, define the axis-aligned bounding box of  any bounded set $S\subseteq \mathbb R^{\mathscr S}$ by \[\operatorname{Box}S= \bigl\{v\in \mathbb R^{\mathscr S}: \forall i,\; \inf_{x \in S}x_i\le  v_i\le \sup_{x\in S}x_i\bigr\}.\]
For a positive-definite $\Delta^2$, we define the ratio \begin{equation}\label{eq: geometric condition for Delta^2}
    R_{\Delta^2,\Lambda} = \inf\bigl\{ R\ge 1: \operatorname{Box} \{v\in \mathbb R^{\mathscr S}: B(v)\le 1 \} \subseteq \{v\in \mathbb R^{\mathscr S}: B(v)\le R\}\bigr\}. 
\end{equation}
In words, $R_{\Delta^2,\Lambda}$ is the scaling factor between two concentric ellipsoids (derived from the quadratic form associated to $\Delta^2$) such that the axis-aligned bounding box of the smaller ellipsoid just fits inside the larger one.  
We require $\Delta^2$ to be positive-definite to ensure that the ellipsoid is non-degenerate and hence $R_{\Delta^2,\Lambda}<\infty.$


We are now in a position to prove the main goal of this section.
\begin{proof}[\bf Proof of Proposition~\ref{prop: sufficient condition for RS}]
We divide the proof into two steps.

\vspace{0.5em}
\noindent {\bf Step 1.}
Our first step is to show that there exist constants $K=K( \Lambda,\Delta^2,\tau^2)>0$ such that for all $t\in [0,1]$, \begin{equation}\label{eq: step 1}
    \limsup_{N\to\infty}\e \lla B(R_{12}-q^*)\rra_{Nk,t,U}^{\mathrm{hat,pert}} \le K(\mathsf{RS}^*(t) - \lim_{N\to\infty}\e \hat  F_{Nk,U}^{\mathrm{pert}}(t) ).
\end{equation}
To this end, fix $\eps>0$.
By Lemma~\ref{lem: hat in free energies}, there exists an absolute constant $N_0>0$ such that for all $N\ge N_0$, 
\[\sup_{t\in[0,1]}\sup_{x\in [1,2]^{\mathbb N\otimes \mathbb N}}| \e \hat F_{N,U}^{\mathrm{pert}}(t) -\e \hat F_{N,x}^{\mathrm{pert}}(t)|\le \eps.\]
    If $B(u-q^*) \ge 2C^{-1}(\mathsf{RS}^*(t) - \e \hat  F_{N,U}^{\mathrm{pert}}(t) +2 \eps)$ for some $u\in \mathsf T$, $t\in[0,1]$, and $N\ge N_0$, then the assumption of the lemma, the preceding display, and Lemma~\ref{lem: hat in free energies} imply that, for some $o(1)$ quantity  uniform in $u$ and $x$ but may depend on $(t,\Lambda,\Delta^2,\tau^2)$, \[\hat \psi_{N,x}^{\mathrm{pert}}(t,u) \le 2 \e \hat  F_{N,x}^{\mathrm{pert}}(t)-2\eps +o(1), \quad \forall x\in [1,2]^{\mathbb N\otimes \mathbb N}. \] 
    Lemma~\ref{lem: cost of free energy implies unlikely} then yields, for all large $N\ge  N_1(t,\eps, \Lambda,\Delta^2,\tau^2) \in\mathbb N$, we have     \begin{equation}\label{eq: decay when a cost}
        \sup_{x\in [1,2]^{\mathbb N\otimes \mathbb N}}\e (\hat G_{N,t,x}^{\mathrm{pert}})^{\otimes 2} (R_{12}=u) \le C'e^{-C'N}
    \end{equation}
    for some constant $C'=C'(\eps, \Lambda, \Delta^2,\tau^2)>0$.
    Now, recall  the definition \eqref{eq: geometric condition for Delta^2} and define \begin{align*}
          \mathsf{E}_1&=\{u\in [-1,1]^{\mathscr S}: B(u-q^*)\ge 2C^{-1}(\mathsf{RS}^*(t) - 
    \lim_{N\to\infty}\e \hat  F_{N,U}^{\mathrm{pert}}(t) +3\eps)\},
    \\\mathsf{E}_2&=\{u\in [-1,1]^{\mathscr S}: B(u-q^*)\ge 2R_{\Delta^2,\Lambda}C^{-1}(\mathsf{RS}^*(t) - 
    \lim_{N\to\infty}\e \hat  F_{N,U}^{\mathrm{pert}}(t) + 4\eps)\}.    
      \end{align*}
    Note that $\mathsf{E}_1\supset \mathsf{E}_2$ and they depend on $t,\eps.$
    There exists $N_2=N_2(t,\eps, \Lambda,\Delta^2,\tau^2)\in\mathbb N$ such that For all $N\ge N_2$, we can estimate \begin{align}
    \sup_{x\in [1,2]^{\mathbb N\otimes \mathbb N}}\e (\hat G_{N,t,x}^{\mathrm{pert}})^{\otimes 2} ( R_{12}\in \mathsf{E}_1 \cap \mathsf T)  &\le \sum_{u\in \mathsf{E}_1 \cap \mathsf T} \sup_{x\in [1,2]^{\mathbb N\otimes \mathbb N}}\e (\hat G_{N,t,x}^{\mathrm{pert}})^{\otimes 2} (R_{12}=u)\nonumber
    \\&\le \sum_{u}\1_{ \{u \in \prod_{s\in\mathscr S}\{ k/ |I_s|: k\in\mathbb Z, \,  |k|\le  |I_s| \} \}} C'e^{-C'N}  \nonumber
    \\&\le (2N+1)^{|\mathscr S|} C'e^{-C'N}, \label{eq: far away but controllable}
    \end{align} where the second inequality is due to \eqref{eq: decay when a cost}.

    Next, we claim that  \begin{equation}\label{eq: far away and non-monotone}
        \forall t\in [0,1],\quad \limsup_{N\to\infty}  \e (\hat G_{Nk,t,U}^{\mathrm{pert}})^{\otimes 2} (R_{12}\in \mathsf{E}_2) =0.
    \end{equation}
Toward a contradiction, fix $t\in [0,1]$ and suppose there exists $\eta>0$ such that \[\limsup_{N\to\infty}  \e  (\hat G_{Nk,t,U}^{\mathrm{pert}})^{\otimes 2} (R_{12}\in \mathsf{E}_2 ) \ge \eta.\] 
By considering a subsequence if necessary, we may assume the limit superior is achieved as a limit.
Recall the event $A_{N,t}$ in \eqref{eq: choosing nonrandom x}.
Since $\lim_{N\to\infty}\p_U (A_{N,t})=1$, we have for sufficiently large $N$, \[ \e_U [ \1_{ \{U\in A_{N,t}\}} \e (\hat G^{\mathrm{pert}}_{Nk,t,U})^{\otimes2}(R_{12 }\in \mathsf{E}_2)]\ge \eta/2\] so that  there exists $x_{N,t}\in A_{N,t}$ for large enough $N$ such that \begin{equation}\label{eq: far away contradiction}
    \e (\hat G^{\mathrm{pert}}_{Nk,t,x_{N,t}})^{\otimes2}(R_{12 }\in \mathsf{E}_2)\ge \eta/2.
\end{equation}
On the other hand, by the compactness of the space of probability measures on $[-1,1]^{\mathscr S}$, there exists a subsequential weak  limit $\nu$ of the sequence of probability measures $(\e (\hat G^{\mathrm{pert}}_{Nk,t,x_{N,t}})^{\otimes2}(R_{12 }\in \cdot ))_{N\ge 1}$.
(In fact, by Lemma~\ref{lem: Talagrand positivity}, it holds that $\operatorname{supp}(\nu)\subseteq [0,1]^{\mathscr S}$.)
Denote the interior of a set $A\subseteq \mathbb R^{\mathscr S}$ by $\operatorname{Int}(A)$.
By  the Portmanteau theorem, we have \[\nu ( \operatorname{Int}(\mathsf{E}_1 \cap \mathsf{T}))\le \limsup_{N\to\infty} \e (\hat G^{\mathrm{pert}}_{Nk,t,x_{N,t}})^{\otimes2}(R_{12 }\in \operatorname{Int}(\mathsf{E}_1 \cap \mathsf{T}))=0,  \] where the equality is due to \eqref{eq: far away but controllable}.
By \eqref{eq: consequence of ASS scheme}, $\nu$ is a Parisi measure, so  Proposition \ref{prop: min and max of support} implies \[\{q^\nu_{\text{min}},q^\nu_{\text{max}}\}\subseteq  \{q^*\} \cup \operatorname{Int} (\mathsf T).\]
The last two displays show, using $\operatorname{Int}(A\cap B)=\operatorname{Int}(A)\cap \operatorname{Int}(B)$, \[\{q^\nu_{\text{min}},q^\nu_{\text{max}}\} \subseteq \bigl( \{q^*\} \cup \operatorname{Int} (\mathsf T)\bigr)\setminus \operatorname{Int}(\mathsf{E}_1 \cap \mathsf{T})= \{q^*\}\cup\bigl( \operatorname{Int}(\mathsf T)\setminus \operatorname{Int}(\mathsf{E}_1) \bigr). \]
In turn, this implies $\operatorname{supp}(\nu) \cap \mathsf{E}_2 =\emptyset$ by \eqref{eq: geometric condition for Delta^2}. Equivalently, we have \[\nu(\mathsf{E}_2)=0.\]
From \eqref{eq: far away contradiction}, since $\mathsf{E}_2$ is a closed set, the Portmanteau theorem implies \[\nu(\mathsf{E}_2)\ge \liminf_{N\to\infty}\e (\hat G^{\mathrm{pert}}_{Nk,t,x_{N,t}})^{\otimes2}(R_{12 }\in \mathsf{E}_2)\ge \eta/2.\]
The last two displays contradict each other, thereby proving the claim.

    From a trivial bound $|B(R_{12}-q^*)|\le 4 \max_{i,j} (\Lambda\Delta^2 \Lambda)_{ij} \equalscolon M <\infty$, we can bound
    \begin{equation} \label{eq: upper bound for expected quadratic form}
    \begin{aligned}
        \e \lla B(R_{12}-q^*)\rra_{Nk,t,U}^{\mathrm{hat,pert}} 
        \le \e \lla B(R_{12}-q^*)\1_{\{R_{12}\in \mathsf{E}_2^c\}}\rra_{Nk,t,U}^{\mathrm{hat,pert}} +M\e (\hat G_{Nk,t,U}^{\mathrm{pert}})^{\otimes 2} (R_{12}\in \mathsf{E}_2).
    \end{aligned}
    \end{equation}
    The first term on the right-hand side of \eqref{eq: upper bound for expected quadratic form} is bounded by $2R_{\Delta^2,\Lambda}C^{-1}(\mathsf{RS}^*(t) -  \lim_{N\to\infty} \e \hat  F_{Nk,U}^{\mathrm{pert}}(t) \\+ 4\eps)  $.
    By \eqref{eq: far away and non-monotone}, the second term  vanishes as $N\to\infty$.
    As a result, for all $0\le t\le 1$, \begin{equation*}
       \limsup_{N\to\infty} \e \lla B(R_{12}-q^*)\rra_{Nk,t,U}^{\mathrm{hat,pert}}\le 2R_{\Delta^2,\Lambda}C^{-1}(\mathsf{RS}^*(t)-\lim_{N\to\infty} \e \hat  F_{Nk,U}^{\mathrm{pert}}(t)+4\eps).
    \end{equation*} 
The proof of \eqref{eq: step 1} is complete by sending $\eps \to 0$.
    

\vspace{0.5em}
\noindent {\bf Step 2.} 
The next step is to deduce replica symmetry from Step 1.
Fix $0\le t\le 1.$
    From Lemma~\ref{lem: derivative of interpolating free energy: hat version} and Equation \eqref{eq: interpolated approximating RS functional}, 
    \begin{equation}\label{eq: derivative: Gronwall}
        (\mathsf{RS}^*(t)-\e \hat  F_{N,U}^{\mathrm{pert}}(t))'= \frac{1}{2}\e \lla B(R_{12}-q^*)\rra_{N,t,U}^{\mathrm{hat,pert}} +o(1) .
    \end{equation}
    For $0\le t\le 1,$ denote $g_N(t)=\mathsf{RS}^*(t)-\e \hat  F_{N,U}^{\mathrm{pert}}(t)$ and $g(t)=\lim_{N\to\infty}g_N(t)$ which exists by Lemma~\ref{lem: hat in free energies}.
    Combining  \eqref{eq: derivative: Gronwall} and \eqref{eq: step 1} yields \begin{equation}\label{eq: limiting derivative: Gronwall}
        \limsup_{N\to\infty}g'_{Nk}(t)\le Kg(t).
    \end{equation}
    Observe that $\sup_{N\ge 1} \sup_{0\le t\le 1}|g'_N(t)|<\infty$
     by \eqref{eq: derivative: Gronwall}
     and, therefore, $g$ is continuous as a uniform limit of continuous functions by invoking the Arzel\`a--Ascoli's theorem.
     Note that $ g(0)=0\le g$ by Lemma~\ref{lem: RS equals perturbed free energy when t=0} and   \eqref{eq: derivative: Gronwall}.
     Sending $N\to\infty$ in $g_{Nk}(t)-g_{Nk}(0)=\int_0^t g'_{Nk}(x) dx$, we have \begin{align*}
         0\le g(t)\le \limsup_{N\to\infty}\int_0^t g'_{Nk}(x) dx \le \int_0^t  \limsup_{N\to\infty} g'_{Nk}(x) dx\le  K\int_0^t g(x)dx .
     \end{align*} where, respectively, the third and last inequalities are due to the reverse Fatou's inequality (for bounded functions on a compact domain $[0,t]$) and \eqref{eq: limiting derivative: Gronwall}.
    Applying the integral form of Gr\"onwall's inequality in the preceding display, we obtain $g(1)=0.$
    The proof is complete by noticing that $\lim_{N\to\infty} \e \hat  F_{N,U}^{\mathrm{pert}}(1) =\lim_{N\to\infty}\e F_{N}(1)$ by Lemma~\ref{lem: hat in free energies}.
\end{proof}

\section{Establishing a free energy cost for the coupled free energy}\label{sec: Two-dimensional Guerra--Talagrand bound}

We implement Step 3 of the proof sketch in this section and establish the following  version of \eqref{eq: free energy cost: overview}.
For $q^*=q^*(\Delta^2,\tau^2)$, define \begin{equation*}
    \widetilde{\mathsf T}  = \{u\in[0,\infty)^{\mathscr S}:  \Delta^2\Lambda (u-q^*)\gee \0 \text{ or }  \Delta^2\Lambda (u- q^*)\lee \0  \}. 
\end{equation*}
Recall the notations in \eqref{eq: interpolated approximating RS functional} and \eqref{eq: definition of constrained free energy}.
\begin{lemma}\label{lem: GT bound}
Let $\Delta^2$ be irreducible.
Suppose $\rho(\Gamma\Delta^2\Lambda)<1/2$.
There exists a constant $C=C(\Lambda,\Delta^2, \tau^2)>0$ such that for any $u\in \widetilde{\mathsf T}$ and sufficiently large $N\in \mathbb N$,
\[ \psi_N(1,u)\le  2\mathsf{RS}^*(1)- \frac{C}{\rho(\Delta^2\Lambda)} B(u-q^*) +o(1),\] where $o(1)$ is uniform in $u$.
Moreover, the constant $C$ is bounded away from zero under the change of the parameters \[\inf_{t\in[0,1]}C(\Lambda, t\Delta^2, \tau^2+(1-t)2\Delta^2 \Lambda q^*(\Delta^2,\tau^2))>0.\]
    
\end{lemma}
It is immediate that $\mathsf T \subseteq \widetilde{\mathsf T}$, where $\mathsf T$  was defined in \eqref{eq: definition of T}.
We will apply the bound only on $\mathsf T$; nevertheless, we provide a proof in full generality on $\widetilde{\mathsf T}$ utilizing the following GT 1-RSB bound, whose proof is omitted as it is identical to the $1$-dimensional case in \cite[Theorem 13.5.1]{Talagrand2013vol2}. 
\begin{lemma}[Guerra--Talagrand 1-RSB bound]\label{lem: constrained 1RSB bound}
Consider two vectors $q_1,q_2\in [0,\infty)^{\mathscr S}$ such that $\Delta^2 \Lambda q_1 \lee \Delta^2 \Lambda  q_2.$
Define the Gaussian vectors \[  Y_{s, j}= z_{s,j} (2(\Delta^2\Lambda q_{1})_s)^{1/2} +z'_{s,j} (2(\Delta^2 \Lambda (q_{2}-q_{1}))_s)^{1/2} +h_s, \quad  s \in \mathscr S, \, j\in\{1,2\},\]
where $h_s$ are centered Gaussian random variables with variances $\e h_s^2=\tau^2_s$, and  $z_{s,j},z_{s,j}'$ are the standard Gaussian random variables independent among different species and such that $c_s=\e z_{s,1} z_{s,2}$, $c_s'=\e z_{s,1}' z_{s,2}'$, and the sequences $(z_{s,j})_{s \in \mathscr S, 1\le j\le 2}$, $(z_{s,j}')_{s \in \mathscr S, 1\le j\le 2}$, $(h_s)_{s\in\mathscr S}$ are independent.
Let $u\in [-1,1]^{\mathscr S}$ be given by \[u_s=c_s(q_{1})_s + c_s'(q_{2}-q_{1})_s, \quad s \in \mathscr S.\]
For any $0\le m\le 1$ and any $b\in\mathbb R^{\mathscr S}$, we have
    \begin{align*}
        \psi_{N} (1,u)&\le 2\log 2 +  B(\1-q_{2}) -m 
        \bigl( B(q_{2}) -B(q_{1})+ B(u) -B(c\circ q_{1})\bigr)
        \\& - \sum_{s \in \mathscr S}\lambda_{s,N} b_s u_s 
        \\&+\frac{1}{m}\sum_{\substack{s \in \mathscr S}} \lambda_{s,N} \e\log \e'(\ch Y_{s,1}\ch Y_{s,2} \ch b_s +\sh Y_{s,1} \sh Y_{s,2} \sh b_s)^m
        \\&+o(1),
    \end{align*} where $o(1)$ is uniform in the parameters $m,b,c,c',q_1,q_2,\tau^2$ and in particular $u.$
\end{lemma}

Denote the upper bound in the previous lemma by \[U_N=U_N(c,c';q_1,q_2;m, b).\]
We explore various properties of this upper bound, which is done in the next subsection.

\subsection{Controlling the 1-RSB upper bound}
In the following two lemmas, we choose $m=1/2$.
\begin{lemma}\label{lem: m=1/2 less than 2RS}
    Choose the parameters $ b=\0$, $m=1/2$, $c, c'\in [-1,1]^{\mathscr S}$, $q_1=q^*$, $q_2$ satisfying $\Delta^2\Lambda(q_2-q^*)\gee \0$, and $u=c\circ q^* +c'\circ (q_2-q^*)$.
    Suppose either $q_2=q^*$ or $c=c'=\pm \1$.
    Then, it holds that \begin{equation*}
        U_N(c,c';q^*,q_2;1/2, \0) \le 2\mathsf{RS}^*(1)+o(1),
    \end{equation*} where $o(1)$ is uniform in $c,c',q_2$, in particular $u.$

\end{lemma}
\begin{proof}
    By directly plugging in the parameters, \begin{align}
        U_N(c,c';q^*,q_2;1/2, 0)&= 2\log 2 +B(\1-q_2) -  2^{-1}(B(q_2)-B(q^*) -B(c\circ q^*)+B(u)) \nonumber\\&\quad +2\sum_{s}\lambda_{s,N} \e\log \e' \sqrt{\ch Y_{s,1}\ch Y_{s,2}} +o(1). \label{eq: m=1/2 upper bound: 0}
    \end{align}
    By the Cauchy--Schwarz inequality, for any $s \in \mathscr S$, \begin{align*}
        \e' \sqrt{\ch Y_{s,1} \ch Y_{s,2}}&\le \sqrt{\e' \ch Y_{s,1}\e' \ch Y_{s,2}}
        \\&= \exp( (\Delta^2 \Lambda (q_2-q^*))_s) \cdot \prod_{j=1,2} \sqrt{\ch (h_s +z_{s,j} (2 (\Delta^2\Lambda q^*)_s)^{1/2})}. \nonumber
    \end{align*} 
    Combining the previous display with \eqref{eq: m=1/2 upper bound: 0} and using that $z_{s,1}\stackrel{d}{=}z_{s,2}$, the desired result would follow from \eqref{eq: B_N converge uniformly} and $\lim_{N\to\infty}\lambda_{s,N} = \lambda_s$, provided that \begin{equation}\label{eq: m=1/2 upper bound: 1}
        2B(\1-q_2) - (B(q_2)-B(q^*)-B(c\circ q^*)+B(u))+ 4B(\1, q_2-q^*)\le 2B(\1-q^*).
    \end{equation}
    Expanding the quadratic forms with the identity $B(a+b)=B(a)+2B(a,b)+B(b)$ for $a,b\in [-1,1]^{\mathscr S}$,  \eqref{eq: m=1/2 upper bound: 1} is equivalent to \begin{equation*}
        B(q_2)+B(c\circ q^*)\le B(u)+B(q^*).
    \end{equation*} 
    It is easy to check that either of the assumptions of the Lemma~imply the preceding display.

\end{proof}


\begin{lemma}\label{lem: Hess in b when m=1/2}
    For any $u\in [0,\infty)^{\mathscr S}$ such that  $\Delta^2\Lambda (u-q^*)\gee \0$ and $ b\in\mathbb R^{\mathscr S}$, \[0\le \partial_{b_l}\partial_{b_k} U_N (\1,\1; q^*,u; 1/2, b) \le \lambda_{k,N} \1_{k=l}. \]
\end{lemma}
\begin{proof}
Since $c=c'=\1,$ we may write $Y_k\colonequals Y_{k,1}=Y_{k,2}$ for any $k\in\mathscr S.$
Denote  \begin{align*}
    N_k=N_k( b_k)&=\ch^2 Y_{k} \sh b_k +\sh^2 Y_{k} \ch  b_k,
    \\D_k=D_k( b_k)&=\ch^2 Y_{k}\ch b_k +\sh^2 Y_{k}  \sh  b_k.
\end{align*}
Note $D_k>0$ since $\ch>|\sh|.$
They satisfy the relations $N'=D$ and $D'=N$.
A direct differentiation yields \begin{equation*}
    \partial_{ b_k} U_N (\1,\1; q^*, u; 1/2, b) = \lambda_{k,N}  \e \frac{\e' N_k D_k^{-1/2}}{\e' D_k^{1/2}} -\lambda_{k,N} u_k,
\end{equation*} which shows
$\partial_{ b_l}\partial_{ b_k} U_N (\1,\1; q^*, u; 1/2, b)=0$ whenever $k\ne l$.
A further differentiation and some algebra yields \begin{equation}\label{eq: Hess in lambda when m=1/2}
    {\lambda_{k,N}}^{-1}\partial_{ b_k}\partial_{ b_k} U (\1,\1; q^*,u; 1/2, b)= 1- \frac{1}{2}\e \Bigl( \frac{\e' N_k^2 D_k^{-2} D_k^{1/2} }{\e' D_k^{1/2}} +  \Bigl(\frac{\e' N_k D_k^{-1} D_{k}^{1/2}}{\e' D_k^{1/2}}\Bigr)^2\Bigr) .
\end{equation}
By an elementary formula for hyperbolic tangents, we can write \[\frac{N_k}{D_k}= \frac{\hth  b_k + \hth^2 Y_{k}}{1+\hth^2 Y_{k}\hth  b_k}=\hth( b_k+\text{arctanh}(\hth^2 Y_{k})).\]
Note that $D_k= \ch^2 (Y_k) \ch ( b_k) (1+\hth^2 (Y_k) \hth( b_k))>0$ and, hence, the factors $D_k^{1/2}$ in \eqref{eq: Hess in lambda when m=1/2} can be regarded as a change of densities. 
Now, the preceding display shows $|N_k/D_k|\le 1$, from which it is easy to see that $0\le \partial_{ b_k}\partial_{ b_k} U (\1,\1; q^*,u; 1/2, b)|\le \lambda_{k,N}.$
This completes the proof.
\end{proof}

Parallel results for $m=0$ is provided in the next two lemmas.

\begin{lemma}\label{lem: m=0 2RS}
For $m=0$, $b=\0$, and $q_2=q^*$, \[U_N(c,c'; q_1, q^*;0,\0)= 2\mathsf{RS}^*(1)+o(1),\] where $o(1)$ is uniform in $c,c',q_1$, in particular $u.$
\end{lemma} 
\begin{proof}
Note that $Y_{k,1}\stackrel{d}{=}Y_{k,2}\stackrel{d}{=} h_k+z(2 (\Delta^2\Lambda q^*)_k)^{1/2}$ for a standard Gaussian $z$.
Then, a straightforward calculation shows
\begin{align*}
    U_N(c,c'; q_1, q^*;0,0)&= 2\log 2+ B(\1-q^*) +\sum_{k \in \mathscr S}\lambda_{k,N}\e\log \ch Y_{k,1}\ch Y_{k,2}+o(1) 
    \\&=2\mathsf{RS}^*(1)+o(1).
\end{align*} 
\end{proof}

\begin{lemma}\label{lem: Hess in b: m=0}
    For   any $u\in [0,1]^{\mathscr S}$ such that  $\Delta^2\Lambda (u-q^*)\lee \0$  and $ b\in\mathbb R^{\mathscr S}$, \[0\le \partial_{ b_l}\partial_{ b_k} U_N (\1,0; u,q^*; 0, b)\le \lambda_{k,N} \1_{k=l}. \]
\end{lemma}
\begin{proof}
A direct differentiation yields \begin{equation*}
    \partial_{ b_k} U (\1,0; u,q^*; 0, b) =  \lambda_{k,N} \e \frac{\ch Y_{k,1}\ch Y_{k,2} \sh b_k +\sh Y_{k,1} \sh Y_{k,2} \ch  b_k}{\ch Y_{k,1}\ch Y_{k,2} \ch b_k +\sh Y_{k,1} \sh Y_{k,2} \sh  b_k} -\lambda_{k,N} u_k,
\end{equation*} which shows $\partial_{ b_l}\partial_{ b_k} U (\1,0; u,q^*; 0, b)=0$ whenever $k\ne l.$

Denote the numerator and the denominator of the previous display by \begin{align*}
    N=N( b_k)&=\ch Y_{k,1}\ch Y_{k,2} \sh b_k +\sh Y_{k,1} \sh Y_{k,2} \ch  b_k,
    \\D=D( b_k)&=\ch Y_{k,1}\ch Y_{k,2} \ch b_k +\sh Y_{k,1} \sh Y_{k,2} \sh  b_k.
\end{align*}
They satisfy the relations $N'=D$ and $D'=N$, which show \[\lambda_{k,N}^{-1}\partial_{ b_k}\partial_{ b_k} U (\1,0; u,q^*; 0, b)= \e \frac{N'D-ND'}{D^2}=\e\frac{D^2-N^2}{D^2}=1-\e \frac{N^2}{D^2}.\]
By an elementary formula for hyperbolic tangents, we can write \[\frac{N}{D}= \frac{\hth  b_k + \hth Y_{k,1} \hth Y_{k,2}}{1+\hth Y_{k,1} \hth Y_{k,2}\hth  b_k}=\hth( b_k+\text{arctanh}(\hth Y_{k,1} \hth Y_{k,2})),\] which finishes the proof. 
\end{proof}


\subsection{Proof of Lemma~\ref{lem: GT bound}}

We begin with an elementary lemma from analysis which we omit the proof.
\begin{lemma}\label{lem: inf of jointly continuous is continuous}
    Let $A$ and $B$ be metric spaces.
    Assume  $B$ is compact.
    Consider a jointly continuous function $g: A\times B \to \mathbb R$.
    Then, the function $f:A\to\mathbb R$ given by $f(a)=\inf_{b\in B}g(a,b)$ is continuous.
\end{lemma}

We first prove one half of the lemma.

\begin{lemma}\label{lem: GT bound: u>q^*}
Lemma~\ref{lem: GT bound} holds   all $u\gee \0$ such that $\Delta^2\Lambda (u-q^*)\gee \0$.
\end{lemma}

\begin{proof}
If $u\in [0,\infty)^{\mathscr S}\setminus [0,1]^{\mathscr S}$, then it is clear form the definition in  \eqref{eq: definition of constrained free energy}  that $\psi_{N}(1,u)=-\infty$, which trivially verifies the lemma.
We may therefore assume $\0\lee u\lee \1.$

Suppose $c=c'=\1$, which implies $q_2=u$ and $ Y_1=Y_2= (h_s+z_{s} (2(\Delta^2\Lambda q^*)_s)^{1/2}+z_{s}' (2(\Delta^2\Lambda (u-q^*))_s)^{1/2})_{s \in \mathscr S}\equalscolon Y$. 
We compute the gradient \[W_N(u)\colonequals \nabla_{ b} U_N(\1,\1; q^*,u; 1/2, b)|_{ b =0} = -\Lambda_N u +\Lambda_N \e \frac{\e' \hth^2 Y \ch Y}{\e' \ch Y}, \] where the operations on the right-hand side are understood entrywise.
We will first establish an integral representation of  $W_N(u)$.
To do that, let $\Gamma(u)$ be a diagonal matrix with its $ll$'th  entry \begin{equation}\label{eq: diagonal matrix at u}
     \e \frac{ \e' \sch^4 Y_l \ch Y_l}{\e'\ch Y_l}. 
\end{equation}
For $0\le t\le 1$, consider the interpolation $u_t\colonequals tu +(1-t) q^*$ and define a diagonal matrix \[\hat \Gamma(u)= \int_0^1\Gamma(u_t) dt,\] where the integration is done entrywise.

We claim that \begin{equation}\label{eq: integral representation of W: m=1/2}
    W_N(u) =  \bigl(2 \hat \Gamma(u)  \;  \Lambda_N \Delta^2 \Lambda- \Lambda_N\bigr)(u-q^*).
\end{equation}
We first assume that $\Delta^2\Lambda (u-q^*)\sgee \0$. 
Since \[\partial_{u_k} Y_l =   (\Delta^2\Lambda)_{lk} \frac{z_{l}'}{\sqrt{2 (\Delta^2\Lambda (u-q^*))_l}},\] we have $\e Y_{l} \partial_{u_k} Y_l= (\Delta^2\Lambda)_{lk}$ for any $k,l\in \mathscr S$.
By Gaussian IBP,    \begin{align*}
    \partial_{u_k}W_N(u)_l&= - \lambda_{l,N}\1_{k=l}+ \lambda_{l,N} \e \bigl( \frac{\e' \partial_{u_k} Y_l   (\hth^2 \ch)'(Y_l) }{\e' \ch Y_l}- \frac{\e' \hth^2(Y_l)\ch (Y_l)}{(\e' \ch Y_l)^2}\e' \partial_{u_k}Y_l \sh(Y_l) \bigr) \nonumber
    \\&=-\lambda_{l,N}\1_{k=l}+ \lambda_{l,N} (\Delta^2\Lambda )_{lk} \e \Bigl( \frac{\e'   (\hth^2 \ch)''(Y_l) }{\e' \ch Y_l}- \frac{\e' \hth^2(Y_l)\ch (Y_l)}{\e' \ch Y_l}\Bigr) \nonumber
    \\&=-\lambda_{l,N}\1_{k=l}+  (\Lambda_N\Delta^2\Lambda )_{lk} 2\Gamma(u)_l. 
\end{align*}
Now, \eqref{eq: integral representation of W: m=1/2} follows by applying the fundamental theorem of calculus to each function $t\mapsto W_N(u_t)_l$ for $l \in \mathscr S$, invoking that $W_N(u_0)=W_N(q^*)=\0$.
The general case $\Delta^2\Lambda (u-q^*)\gee \0$ can be dealt with approximation as follows. 
Let  $\0\lee u\lee \1$  be such that $\Delta^2\Lambda(u-q^*)\gee \0$, and for all $t\in[0,1]$, define $w_t= tu +(1-t)\1$.
By the irreducibility of $\Delta^2$, we have $\Delta^2\Lambda (w_t-q^*)= t\Delta^2\Lambda (u-q^*)+ (1-t)\Delta^2\Lambda \1 \sgee \0$ for all $t\in [0,1)$, so \eqref{eq: integral representation of W: m=1/2} holds for $w_t$ for all $t\in [0,1)$.
Since  $\lim_{t\to 1}w_t =u$, the continuity of both sides of \eqref{eq: integral representation of W: m=1/2} yields the desired result at $u$.
This finishes the proof of the above claim.

By the FKG inequality, $\e' \sch^4 |Y_l| \ch |Y_l|\le \e' \sch^4 |Y_l| \e' \ch |Y_l| $ and, hence, \eqref{eq: diagonal matrix at u} is bounded by $ \e \sch^4 Y_l$.
Moreover, since $h_l$ has a decreasing and symmetric density, the assumption $ \Delta^2\Lambda (u-q^*)\gee \0$ implies \begin{equation*}
    \e \sch^4Y_l =\e\sch^4 (h_l+ z(2 (\Delta^2\Lambda u)_l)^{1/2}) \le \e\sch^4 (h_l+ z(2 (\Delta^2\Lambda q^*)_l)^{1/2}) 
\end{equation*} and hence \begin{equation}\label{eq: consequence of symmetric decreasing density}
    \0 \slee \Gamma(u)\lee \Gamma (q^*)=\Gamma,
\end{equation} where $\Gamma$ was defined in \eqref{eq: diagonal operator}.
Note that \eqref{eq: consequence of symmetric decreasing density} immediately implies \begin{equation}\label{eq: entrywise monotonicity of Gamma}
    \0\slee \hat \Gamma (u) \lee \Gamma .
\end{equation}
Below, we establish two claims about these matrices.

\textit{Claim 1:} $\rho( \hat \Gamma (u) \Lambda_N^{1/2} \Delta^2 \Lambda \Lambda_N^{-1/2})= \rho(\hat \Gamma (u) \Delta^2 \Lambda)$.

\noindent\textit{Proof of Claim 1.}
Diagonal matrices commute, so $\rho( \hat \Gamma (u) \Lambda_N^{1/2} \Delta^2\Lambda \Lambda_N^{-1/2} )= \rho(  \Lambda_N^{1/2} \hat \Gamma (u) \Delta^2\Lambda \Lambda_N^{-1/2})$.
Moreover, similar matrices have the same eigenvalues.
The proof is complete by combining these facts. $\qed$

\textit{Claim 2:} $\rho(\hat \Gamma (u) \Delta^2 \Lambda)\le \rho( \Gamma  \Delta^2 \Lambda)$.

\noindent\textit{Proof of Claim 2.} From \eqref{eq: entrywise monotonicity of Gamma}, we deduce that, for any $n\in\mathbb N$, $0\le  (\hat \Gamma (u) \Delta^2 \Lambda)^n \le (\Gamma \Delta^2 \Lambda)^n$ entrywise and, therefore, $\|(\hat \Gamma (u) \Delta^2 \Lambda)^n\| \le \|(\Gamma \Delta^2 \Lambda)^n\|$ where $\|\cdot\|$ is the Frobenius norm. 
Invoking  Gelfand's formula \cite[Theorem 10.13]{Rud91}, the proof is finished by taking the $n$'th root and sending $n\to\infty$. $\qed$

By Taylor's theorem, Lemma~\ref{lem: m=1/2 less than 2RS} and Lemma~\ref{lem: Hess in b when m=1/2} imply that for some $o(1)$ uniform in $u$, \begin{equation*}
     U_N(\1, \1;q^*,u; 1/2, b) \le 2\mathsf{RS}^*(1) +  W_N(u)^\intercal b +  b^\intercal \Lambda_N b +o(1)
\end{equation*}
Plugging in $ b=-2^{-1}\Lambda_N^{-1}W_N(u)$, we have \begin{equation}\label{eq: Taylor: inf_lambda: m=1/2}
    \psi_N(1,u)\le  \inf_ b U_N(\1, \1;q^*, u; 1/2, b)\le 2\mathsf{RS}^*(1) - 4^{-1} W_N(u)^\intercal \Lambda_N^{-1} W_N(u) +o(1).
\end{equation} 
Using \eqref{eq: integral representation of W: m=1/2}, we can calculate the second term on the right-hand side of the preceding display \begin{align}
    W_N(u)^\intercal \Lambda_N^{-1} W_N(u) &= (u-q^*)^\intercal (2 \Lambda \Delta^2\Lambda_N \hat \Gamma(u) -\Lambda_N)\Lambda_N^{-1}(2 \hat \Gamma(u) \Lambda_N \Delta^2\Lambda  -\Lambda_N)(u-q^*)\nonumber
    \\& \hspace{-1em}= (u-q^*)^\intercal \Lambda_N^{1/2} (I-2 \Lambda_N^{-1/2}\Lambda \Delta^2\Lambda_N^{1/2} \hat \Gamma(u)) (I-2 \hat \Gamma(u) \Lambda_N^{1/2} \Delta^2\Lambda \Lambda_N^{-1/2})\Lambda_N^{1/2}(u-q^*).\label{eq: complicated quadratic form}
\end{align}
Denote the smallest singular value of a  matrix $A$ by $s_{\text{min}}(A).$
For convenience, let us denote \[M_N(u)=2 \hat \Gamma(u) \Lambda_N^{1/2} \Delta^2\Lambda \Lambda_N^{-1/2}.\]
Claim 1, Claim 2, and the assumption $\rho(\Gamma \Delta^2\Lambda)<1/2$ imply that  $I-M_N(u)$ is invertible for each $N\in\mathbb N$ and, hence, $s_{\text{min}}(I-M_N(u))>0$ for any $N\in\mathbb N.$
Moreover, with the operator norm for matrices, $u\mapsto M_N(u)$ is obviously continuous  and $s_{\text{min}}$ is continuous by Weyl's inequality.
As \begin{equation}\label{eq: definition of T: half of it}
    T\colonequals \{u\in[0,1]^{\mathscr S}: \Delta^2\Lambda(u-q^*)\gee \0\}
\end{equation} is compact, these imply that, for any $N$, \begin{equation}\label{eq: definition of constant C_N}
    C_N\colonequals \inf_{u \in T}\bigl(s_{\text{min}}(I-M_N(u))\bigr)^2>0.
\end{equation}
Now,  \eqref{eq: complicated quadratic form} can be bounded below by \begin{equation}\label{eq: lower bound for quadratic form}
    C_N\|\Lambda_N^{1/2}(u-q^*)\|_2^2.
\end{equation}
Combining \eqref{eq: lower bound for quadratic form} with \eqref{eq: Taylor: inf_lambda: m=1/2}, we have for some $o(1)$ uniform in $u$,   \begin{equation*}
    \psi_N(1,u)\le \inf_ b U_N(\1, \1;q^*, u; 1/2, b)\le 2\mathsf{RS}^*(1) - 4^{-1} C_N\|\Lambda_N^{1/2}(u-q^*)\|_2^2 +o(1).
\end{equation*} 
On the other hand, since $\Delta^2$ is symmetric, $B_N(v)= v^\intercal \Lambda_N\Delta^2\Lambda_N v \le \rho(\Lambda_N^{1/2}\Delta^2\Lambda_N^{1/2})\|\Lambda_N^{1/2}v\|_2^2$ for $v=u-q^*.$
Together with the last display, we have for large enough $N$, \begin{equation*}
    \psi_N(1,u)\le  \inf_ b U_N(\1, \1;q^*, u; 1/2, b) \le 2\mathsf{RS}^*(1)- \frac{C_N}{4 \rho(\Delta^2\Lambda_N)} B_N(u-q^*) +o(1),
\end{equation*} again using the fact that $\rho(\Lambda_N^{1/2}\Delta^2\Lambda_N^{1/2})=\rho(\Lambda_N^{1/2}\Delta^2\Lambda_N\Lambda_N^{-1/2})=\rho(\Delta^2\Lambda_N).$
Now, in view of \eqref{eq: B_N converge uniformly}, the proof of the first part of the lemma will be complete once we verify that $C_N$ converges as $N\to\infty$.
(It is already  clear that the denominator converges, i.e.,  $\lim_{N\to\infty}\rho(\Delta^2 \Lambda_N)=\rho(\Delta^2 \Lambda)$ by the continuity of spectral radius.)
To this end, as a consequence of min-max formula for singular values,
we obtain a variational formula for \eqref{eq: definition of constant C_N} \begin{equation}\label{eq: variational formula of smallest singular value}
   C_N=\inf_{u\in T} \inf_{\|x\|_2=1} \|(I-M_N(u))x\|_2^2,
\end{equation} from which it suffices to establish the uniform convergence of $\|(I-M_N(u))x\|_2$ in $u$ and $x$ as $N\to\infty.$
With the obvious definition of $M(u)=\lim_{N\to\infty} M_N(u) = 2\hat \Gamma(u) \Lambda^{1/2}\Delta^2  \Lambda^{1/2}$, we estimate using the triangle inequality \begin{align*}
    \bigl|\|(I-M_N(u))x\|_2-\|(I-M(u))x\|_2
    \bigr| &\le \| (M_N(u)-M(u))x\|_2\le \|M_N(u)-M(u)\|_2 \|x\|_2.
\end{align*}
By the submultiplicative property of the matrix norm $\|\cdot \|_2$, \begin{align*}
    2^{-1}\|M_N(u)-M(u)\|_2&= \|\hat \Gamma(u) \Lambda_N^{1/2}\Delta^2 \Lambda\Lambda_N^{-1/2} - \hat \Gamma(u) \Lambda^{1/2}\Delta^2 \Lambda^{1/2}\|_2
    \\&\le \|\hat \Gamma(u)\|_2 \cdot \| \Lambda_N^{1/2}\Delta^2 \Lambda^{1/2}\Lambda_N^{-1/2}-\Lambda^{1/2}\Delta^2\|_2 \cdot \|\Lambda^{1/2}\|_2.
\end{align*}
Since  $\|\hat \Gamma(u)\|_2\le \max_{s\in\mathscr S}\Gamma_{ss}$ by \eqref{eq: entrywise monotonicity of Gamma},  we obtain the desired uniform convergence from the last two displays.

Finally, it remains to establish the asserted stability under the change of parameters.
From  \eqref{eq: variational formula of smallest singular value} and the previous uniform convergence of $C_N$, we can write the quantity $\lim_{N\to\infty}C_N$ as  \[C=C(\Lambda, \Delta^2, \tau^2 )= \inf_{u\in T} \inf_{\|x\|_2=1} \Bigl\|\Bigl(I-2\int_0^1 \Gamma( wu+(1-w)q^*) dw \;\Lambda^{1/2}\Delta^2\Lambda^{1/2} \Bigr)x\Bigr\|_2^2,\] where we recall that (cf. \eqref{eq: diagonal matrix at u}) $\Gamma(u)$ is a diagonal matrix with entries \[\Gamma(u)_{ss}= \e \frac{\e' \sch^3 Y_s }{\e' \ch Y_s} \quad \text{for} \quad Y_s=z_{s} (\tau^2_s+ 2(\Delta^2\Lambda q^*)_s)^{1/2}+z_{s}' (2(\Delta^2\Lambda (u-q^*))_s)^{1/2} \text{ and } s\in\mathscr S.\] 
Let $t\in[0,1]$ and consider the change of variance profile $\Delta^2 \to t\Delta^2$ and $\tau^2 \to \tau^2 +(1-t)2\Delta^2\Lambda q^*$.
By Lemma~\ref{lem: stability of q^*}, $q^*$ is invariant under this change of variance profile.
In contrast,  the random vector $Y$ in the preceding display becomes $Y(t,u)\colonequals (z_{s} (\tau^2_s+  2(\Delta^2\Lambda q^*)_s)^{1/2}+z_{s}' \sqrt{t}(2(\Delta^2\Lambda (u-q^*))_s)^{1/2})_{s\in\mathscr S}$, and we can
define a function $D: [0,1]\times \{u\in [0,1]^{\mathscr S}: \Delta^2\Lambda(u-q^*)\gee \0\} \times \{x \in\mathbb R ^{\mathscr S}:\|x\|_2=1\}\to \mathbb R$ given by  \[D(t,u,x)=\Bigl\|\Bigl(I-2 \text{diag}\Bigl( \int_0^1\e \frac{\e' \sch^3 Y(t,wu+(1-w)q^*)_s }{\e' \ch Y(t,wu+(1-w)q^*)_s} \; dw:s\in\mathscr S \Bigr)  \Lambda^{1/2}\Delta^2\Lambda^{1/2} \Bigr)x\Bigr\|_2^2 \] so that \[C(t)\colonequals C(\Lambda,t\Delta^2,\tau^2+(1-t)2\Delta^2\Lambda q^*(\Delta^2,\tau^2))= \inf_{u\in T}\inf_{\|x\|_2=1} D(t,u,x).\]
In view of Lemma~\ref{lem: inf of jointly continuous is continuous} and the compactness of the set  $T \times \{x \in\mathbb R ^{\mathscr S}:\|x\|_2=1\}$, the function $C(t)$ is continuous whenever $D$ is jointly continuous.
Consequently, in order to verify the desired result \[\inf_{t\in[0,1]}C(t)>0,\] it suffices to check that $C(t)>0$ for each $t\in[0,1]$ and that $D$ is jointly continuous.
Indeed, we can easily repeat the argument from \eqref{eq: consequence of symmetric decreasing density} to \eqref{eq: definition of constant C_N}, and conclude $C(t)>0$ for each $t\in[0,1].$ 
Moreover, it is straightforward to see that $D$ is jointly continuous, as all the operations involved in the definition of $D$ respect continuity.
This concludes the proof.

\end{proof}

The following is the other half of the lemma.

\begin{lemma}\label{lem: GT bound: u<q^*}
Lemma~\ref{lem: GT bound} holds for all $u\gee\0$ such that  $\Delta^2\Lambda(u-q^*)\lee \0.$
    
\end{lemma}
\begin{proof}    
As explained in the beginning of the proof of Lemma~\ref{lem: GT bound: u>q^*}, we may assume $\0\lee u\lee \1.$

Consider the 1RSB upper bound $U_N=U_N(c,c';q_1,q^*;m, b)$.
Set the parameters $c=\1$, $c'=0$ so that $u=q_1$ and $z_{k,1}=z_{k,2}\equalscolon z_k.$
Consider the gradient evaluated at $ b=0$ \begin{equation*}
    W_N(u)\colonequals \nabla_ b  U (\1,0; u,q^*; 0, b)|_{ b=0}= -\Lambda_{N} u+ \Lambda_{N} \e \hth Y_{1} \hth Y_{2},
\end{equation*} where, on the right-hand side, the operations are understood entrywise for the correlated random vectors $Y_1$ and $Y_2$ taking values in $\mathbb R^{\mathscr S}$, whose entries are $(Y_j)_l = z_l (2 (\Delta^2\Lambda u)_l)^{1/2} + z_{l,j}' (2 (\Delta^2\Lambda (q^*-u))_l)^{1/2}+h_l$ for $j=1,2$ and $l \in \mathscr S$.

The rest of the proof is identical to the proof of Lemma~\ref{lem: GT bound: u>q^*}; we only sketch the necessary modifications.
To obtain an analogue of \eqref{eq: integral representation of W: m=1/2}, we first assume that  $\0\slee \Delta^2\Lambda u \slee \Delta^2\Lambda q^*.$
Since \[\partial_{u_k} (Y_j)_l =   (\Delta^2\Lambda)_{lk} \Bigl(\frac{z_l }{\sqrt{2 (\Delta^2\Lambda u)_l}}-\frac{z_{l,j}'}{\sqrt{2 (\Delta^2\Lambda (q^*-u))_l}}\Bigr),\] we have $\e (Y_j)_{l} \partial_{u_k} (Y_j)_l=0$ and $\e (Y_j)_{l} \partial_{u_k} (Y_{3-j})_l=   (\Delta^2\Lambda)_{lk}$ for any $j=1,2$ and $k,l\in \mathscr S$.
By Gaussian IBP,    \begin{align*}
    \partial_{u_k}(W_N(u))_l&= -\lambda_{l,N}\1_{k=l}+ \lambda_{l,N}\sum_{j=1,2} \e \partial_{u_k} (Y_j)_l  \sch^2 (Y_j)_l  \hth (Y_{3-j})_l  \nonumber
    \\&=-\lambda_{l,N}\1_{k=l}+ \lambda_{l,N}(\Delta^2\Lambda)_{lk} \sum_{j=1,2}  \e   \sch^2 (Y_j)_l  \sch^2 (Y_{3-j})_l \nonumber
    \\&=-\lambda_{l,N}\1_{k=l}+  2 (\Lambda_N\Delta^2\Lambda)_{lk} \e   \sch^2 (Y_1)_l  \sch^2 (Y_{2})_l. 
\end{align*}
If we define \[\Gamma(u)= \text{diag}( \e \sch^2(Y_1)_l \sch^2(Y_2)_l: l \in \mathscr S),\]
we obtain an analogue of \eqref{eq: integral representation of W: m=1/2} under the current assumption $\0\slee \Delta^2\Lambda u \slee \Delta^2\Lambda q^*.$
The general case $\0\lee \Delta^2\Lambda u \lee \Delta^2\Lambda q^*$ can be approximated as follows.
If $q^*=\0$, then $\Delta^2 \Lambda u=\0$ and $B(u-q^*)=0$, so that the lemma is trivially true from $\psi_N(1,u)\le  2\e F_N(1)\le 2\mathsf{RS}^*(1)+o(1)$.
The remaining case is  $q^*\in (0,1)^{\mathscr S}$, in view of Proposition~\ref{prop: uniqueness of fixed-point}.
We  choose $\0 \slee v\slee q^*$ so that $\0 \slee \Delta^2 \Lambda v\slee \Delta^2\Lambda q^*$ by the irreducibility of $\Delta^2$, and define $w_t =tu +(1-t)v$ for $t\in [0,1]$.
One can easily check  $\Delta^2\Lambda w_t \sgee \0 $ and $\Delta^2\Lambda(q^*-w_t)\sgee \0$ for all $t\in[0,1)$.
Now, we approximate $u$ by $w_t$ as $t\to 1$, and use the continuity of both sides of \eqref{eq: integral representation of W: m=1/2} to validate it in this general case.

For any $l\in\mathscr S$, we have $(Y_1)_l\stackrel{d}{=}(Y_2)_l \stackrel{d}{=}2 (\Delta^2\Lambda q^*)_l z +h_l$ for some standard normal random variable $z$, so by the Cauchy--Schwarz inequality, \begin{equation}\label{eq: control of Gamma for u<q^*}
    0\slee  \Gamma(u)\lee  \Gamma(q^*)=\Gamma,
\end{equation}where $\Gamma$ was defined in \eqref{eq: diagonal operator}. 
To finish the proof, we use, respectively, Lemma~\ref{lem: m=0 2RS},  Lemma~\ref{lem: Hess in b: m=0}, and \eqref{eq: control of Gamma for u<q^*} in place of Lemma~\ref{lem: m=1/2 less than 2RS}, Lemma~\ref{lem: Hess in b when m=1/2}, and  \eqref{eq: consequence of symmetric decreasing density}.
Of course, we need to  reverse the inequality in the definition \eqref{eq: definition of T: half of it}.

\begin{proof}[\bf Proof of Lemma~\ref{lem: GT bound}]
      The proof is obtained by combining Lemmas~\ref{lem: GT bound: u>q^*} and \ref{lem: GT bound: u<q^*}.
\end{proof}

\end{proof}

\section{Proof of Theorem \ref{thm: main theorem}}\label{sec: Proof of main theorem}
We start with the proof of the presence of RSB solution above the AT surface.
Recall the following.
\begin{theorem}[Theorem 1.15, \cite{DW21}]\label{thm: RSB}
Assume $\Delta^2$ is positive-definite and $\tau^2\sgee \0.$
If $\rho(\Gamma \Delta^2\Lambda)  >1/2$, then \[\lim_{N\to\infty }\e F_N < \inf_{q\in [0,1]^{\mathscr S}}\mathsf{RS}(q).\]    
\end{theorem}
We can extend this to any  irreducible positive-semidefinite $\Delta^2$ and $\tau^2\gee \0$.

\begin{lemma}\label{lem: RSB refined}
   Theorem \ref{thm: RSB} holds under the weaker assumption of irreducible positive-semidefinite $\Delta^2$ and $\tau^2\gee \0$.
\end{lemma}
\begin{proof}

    Since $\Delta^2$ is irreducible, the matrix $(\Lambda \Gamma)^{1/2}\Delta^2(\Lambda \Gamma)^{1/2}\gee \0$ is also irreducible, and we can find the Perron eigenvector with all entries positive; see \cite[Equation (8.3.8)]{Mey00_book}.
    Thanks to Lemma~\ref{lem: infimum of RS: semidefinite}, we can finish the proof by exactly repeating the arguments in the proof of  \cite[Theorem 1.15]{DW21}.
\end{proof}

We are now in a position to prove the main theorem.

\begin{proof}[\bf{Proof of Theorem \ref{thm: main theorem}}]

By Lemma~\ref{lem: RSB refined}, it remains to consider the case   $\rho(\Gamma\Delta^2\Lambda)\le 1/2$, which we divide into three steps.

\vspace{0.5em}
\noindent \textbf{Step 1. $\rho(\Gamma \Delta^2 \Lambda)<1/2$, irreducible, positive-definite $\Delta^2$, and $\tau^2\sgee\0$.}
 To establish the theorem in this case, it suffices to validate the assumption of  Proposition~\ref{prop: sufficient condition for RS}, which we recall for convenience: for any $0\le t\le 1$, there exist a constant \(C = C(\Lambda,\Delta^2,\tau^2) > 0\) and an \(o(1)\) quantity which is uniform in \(u \in \mathsf{T}\) (but may depend on $t$) such that \begin{equation}\label{eq: desired result}
   \psi_N(t,u)\le 2\mathsf{RS}^*(t) -C B(u-q^*(\Delta^2,\tau^2))+o(1) \quad \text{for all } u\in\mathsf T, \, \, 0\le t\le 1.
\end{equation}
To this end, fix $t\in [0,1]$ and consider a change of variance profile $\Delta^2 \rightsquigarrow t\Delta^2$ and $\tau^2 \rightsquigarrow \tau^2+ (1-t) 2\Delta^2 \Lambda q^*(\Delta^2,\tau^2)$.
By Lemma~\ref{lem: stability of q^*}, we can unambiguously write $q^*$, and from this, one can check that $\Gamma$ is invariant under this change of variance profile.
On the contrary, it is easy to check the following changes \begin{equation*}
    \psi_N(1,u) \rightsquigarrow \psi_N(t,u), \quad 2\mathsf{RS}^*(1)\rightsquigarrow2\mathsf{RS}^*(t), \quad B(u-q^*)\rightsquigarrow tB(u-q^*).
\end{equation*} 
In particular, $\rho(\Gamma \Delta^2\Lambda) \rightsquigarrow t \rho(\Gamma \Delta^2\Lambda)<1/2$ whereas the ratio $B(u-q^*)/ \rho(\Delta^2\Lambda)$ does not change.
Consequently, under the new variance profile,  Lemma~\ref{lem: GT bound} implies  \eqref{eq: desired result} as desired.
Here, we emphasize  the constant $C$ remains bounded away from zero uniformly in $t\in[0,1]$---this is precisely guaranteed by Lemma~\ref{lem: GT bound}.

\vspace{0.5em}
\noindent\textbf{Step 2. $\rho(\Gamma \Delta^2 \Lambda)<1/2$, irreducible, positive-semidefinite $\Delta^2$.}
Consider $(\Delta^2_n,\tau^2_n)_{n\ge 1}$ in \eqref{eq: approximating Delta^2} and denote the diagonal matrix with entries \eqref{eq: diagonal operator} under the variance profile $(\Delta^2_n,\tau^2_n)$ by $\Gamma_n$.
By the continuity of the spectral radius, we have $\lim_{n\to\infty}\rho(\Gamma_n \Delta^2_n\Lambda) =\rho(\Gamma \Delta^2\Lambda)<1/2.$
Moreover, $\Delta^2_n$ is irreducible and positive-definite, and $\tau^2_n\sgee \0$, so we can apply Step 1 to the variance profile $(\Delta^2_n,\tau^2_n)$ and conclude that the model is replica symmetric for sufficiently large $n$.
Now, we can combine Lemma~\ref{lem: free energy Lipschitz} and the convergence in \eqref{eq: extending to non-negative entries: 0} to obtain the desired result by sending $n\to\infty.$

\vspace{0.5em}
\noindent\textbf{Step 3. $\rho(\Gamma \Delta^2 \Lambda)=1/2$, irreducible, positive-semidefinite $\Delta^2$.}
In this case, we fix $\Delta^2$ and approximate $\tau^2$ by $\tau^2_n\colonequals \tau^2 +n^{-1}\1$.
Denote the diagonal matrix with entries \eqref{eq: diagonal operator} under the variance profile $(\Delta^2,\tau^2_n)$ by $\Gamma_n$.
Item (ii) in Lemma~\ref{lem: q^* monotone and continuous} ensures $\Gamma_n\slee \Gamma$.

We claim that $\rho(\Gamma_n\Delta^2\Lambda)<\rho(\Gamma\Delta^2\Lambda)=1/2$ for all $n\in\mathbb N.$
To this end, fix $n\in\mathbb N$.
Irreducibility of $\Delta^2$ implies $\Gamma_n\Delta^2\Lambda$ is also irreducible, so it has a Perron vector $x_n\sgee \0$.
Since $\Gamma_n\slee \Gamma$, it holds that $ \rho( \Gamma_n\Delta^2\Lambda)x_n=\Gamma_n\Delta^2\Lambda x_n \slee \Gamma\Delta^2\Lambda x_n $ and, therefore, \[  \rho( \Gamma_n\Delta^2\Lambda)< \min_{s\in\mathscr S} \frac{(\Gamma\Delta^2\Lambda x_n)_s}{(x_n)_s} \le \rho(\Gamma\Delta^2\Lambda),\] where the last inequality is due to the Collatz--Wielandt formula (see, e.g., \cite[(8.3.3)]{Mey00_book}) applied to the irreducible matrix $\Gamma \Delta^2 \Lambda \gee \0$.
This validates the claim.

Thanks to this claim, we can apply Step 2 so that the model is replica symmetric at $(\Delta^2,\tau^2_n)$.
By the continuity in Lemma~\ref{lem: q^* monotone and continuous}, $q^*_n=q^*_n(\Delta^2,\tau^2_n)$ converges to $q^*=q^*(\Delta^2,\tau^2)$ as $n\to\infty$ and we obtain \eqref{eq: extending to non-negative entries: 0} by a trivial modification of the proof of Lemma~\ref{lem: infimum of RS: semidefinite}.
On the other hand, the free energy also converges as $n\to\infty$ by  Lemma~\ref{lem: free energy Lipschitz}, which finishes the proof.


\end{proof}

\appendix


\section{Extension of  the Parisi functional}\label{sec: Extension of Parisi functional}
In this appendix, we extend the Parisi functional to a larger domain that contains continuous measures as well, and validate Proposition~\ref{prop: support of Parisi measure}.
The results  are presented for centered Gaussian external field, but it should be noted that they all hold for deterministic external fields with necessary modifications.
The proofs of this appendix are straightforward extensions of those of the single-species SK model.

The first result of this appendix is the extension of the Parisi functional to arbitrary measures on $[0,1]^{\mathscr S}$ with totally ordered supports.
Before stating the result, let us first provide relevant definitions.
Denote the $L^p$ norm on $\mathbb R^{\mathscr S}$ by $\|\cdot \|_p$ for $1\le p\le \infty$. 
For any probability measures $\mu$ and $\nu$ on $[0,1]^{\mathscr S}$, define the $1$-Wasserstein distance by $W_1(\mu,\nu)= \inf_{X\sim \mu,\; Y\sim \nu}\e \|X-Y\|_1$, which metrizes the weak topology.
Recall $\mathcal M^{\uparrow}$ and $\mathcal M^\uparrow_d$ defined in Section~\ref{sec: Support of Parisi measures}.

The following confirms the Lipshitz continuity of the Parisi functional on $\mathcal M^\uparrow_d.$
\begin{proposition}\label{prop: Lipschitz continuity of Parisi functional}
For any $\mu_1,\mu_2\in\mathcal M^\uparrow_d$, \[|\mathscr P(\mu_1) -\mathscr P(\mu_2)|\le \|\Lambda \Delta^2 \Lambda\|_\infty W_1(\mu_1,\mu_2).    \]
\end{proposition}

To extend the Parisi functional to all of $\mathcal M^\uparrow$, it remains to observe the following.
\begin{lemma} \label{lem: closure of discrete with totally ordered support}
    Under the weak topology, the closure of $\mathcal M^\uparrow_d$ is equal to $\mathcal M^\uparrow.$
    In particular, $\mathcal M^\uparrow$ is compact.
\end{lemma}
\begin{proof}
    For any $\mu\in\mathcal M^\uparrow$, we can find a sequence of discrete probability measures $(\mu_n)_{n\ge 1}$ that weakly converges to $\mu$, and such that $\cup_{n\ge 1}\operatorname{supp}(\mu_n)\subseteq \operatorname{supp}(\mu).$
    In particular, $\operatorname{supp}(\mu_n)$ is totally ordered for any $n\ge 1.$
    This proves half of the lemma.
    
    Conversely, suppose there exists a sequence of discrete measures $(\mu_n)_{n\ge 1}$ with totally ordered support that converges weakly to some $\mu.$
    We need to show that $\operatorname{supp}(\mu)$ is totally ordered.
    Suppose otherwise. 
    We can find two points $x,y\in \operatorname{supp}(\mu)\cap [0,1]^{\mathscr S}$ such that, without loss of generality, $x_1> y_1$ and $x_2< y_2.$
    There exist open sets $B_x \ni x$ and $B_y\ni y$  such that $B_x\cap B_y =\emptyset$, and $z_1>w_1$ and $z_2<w_2$ whenever $z\in B_x$ and $w\in B_y.$
    By the Portmanteau theorem, since $\mu_n\stackrel{w}{\to} \mu$ as $n\to\infty$, \begin{equation*}
        \liminf_{n\to\infty} \mu_{n}(B_x) \ge  \mu(B_x)>0\quad   \text{and} \quad \liminf_{n\to\infty} \mu_{n}(B_y) \ge \mu (B_y)>0,
    \end{equation*} so we can find a large $m$ such that $\mu_{m}(B_x)\wedge \mu_{m}(B_y)>0$.
    This is a contradiction to $\operatorname{supp}(\mu_{m})$ being totally ordered.
\end{proof}

    This extension was initially carried out by Guerra \cite{Gue03} in the setting of the SK model, and  generalizations in various models have been established, e.g., vector spin models \cite[Lemma~2.10]{Che23},  multi-species mixed $p$-spin spherical models \cite[Theorem 1.5]{BS22}, and non-convex multi-species models \cite[Lemma~4.10]{Che24}.
As a consequence of the extension, minimizers are guaranteed to exists; we say  $\mu\in\mathcal M^\uparrow$ is a Parisi measure if \[\mathscr P(\mu)=\inf_{\nu\in\mathcal M^\uparrow}\mathscr P(\nu).\]

Having extended the Parisi functional, we can now state the second result of this appendix on the  directional derivative of the Parisi functional, which serves as a tool for the proof of Proposition~\ref{prop: support of Parisi measure}.
    \begin{proposition}\label{prop: derivative of push forward of Parisi functional}
       Suppose $\Delta^2\1 \sgee \0$.  For any $\mu\in\mathcal M^\uparrow$ and any measurable function $a: \operatorname{supp}(\mu) \to \mathbb R^{\mathscr S}$ such that $\0\lee p_1+ a(p_1)\lee p_2+ a(p_2)\lee \1$ whenever $p_1\lee p_2$ are in $ \operatorname{supp}(\mu)$,  \begin{equation*}
        \frac{d}{dt}\mathscr P(\mu_t)\Bigr|_{t=0^+}= \int_{\operatorname{supp}(\mu)}   a(q)^\intercal \Lambda \Delta^2 \Lambda (q-\e_h\Psi_\mu(h,  q))  \,d\mu(q),
    \end{equation*} where the left-hand side is the right derivative at $t=0$, and $\mu_t$ for $0\le t\le 1$ is the push-forward measure of $\mu$ under the map  $q\mapsto q+ ta(q)$.
    \end{proposition}

The rest of this section is devoted to establishing the above propositions and Proposition~\ref{prop: support of Parisi measure}.
Given $r\ge 1$, let $(v_\alpha)_{\alpha\in\mathbb N^r}$ be the Ruelle probability cascade (RPC) corresponding to the sequence in \eqref{eq: zeta sequence}.
Let $(z_{\beta})_{\beta\in \mathbb N^1 \cup \dots \cup \mathbb N^r}$ be i.i.d. standard Gaussian random variables, independent of RPC, and for each $s\in\mathscr S$, define a Gaussian process \begin{equation}\label{eq: Gaussian process C}
    C^s(\alpha) =\sum_{\beta\in p(\alpha)}z_{\beta} (Q_{|\beta|}^s-Q_{|\beta|-1}^s)^{1/2}, 
\end{equation} where $p(\alpha)\colonequals \{\alpha_1, (\alpha_1,\alpha_2), \dots, (\alpha_1,\dots,\alpha_r)\}$ is the path from the root to $\alpha\in\mathbb N^r$ and $|\beta|$ is such that $\beta\in \mathbb N^{|\beta|}$.
For $\alpha^1,\alpha^2 \in \mathbb N^r$, define $\alpha^1\wedge\alpha^2 = |p(\alpha^1)\cap p(\alpha^2)|.$
The covariance structure of \eqref{eq: Gaussian process C} is easily read off from \eqref{eq: definition of Q}, \begin{equation}\label{eq: covariance of Gaussian process C}
    \e C^s(\alpha^1)C^s(\alpha^2) =  Q_{\alpha^1\wedge \alpha^2}^s=2 \sum_{s' \in \mathscr S} \Delta^2_{ss'}\lambda_{s'} q_{\alpha^1 \wedge\alpha^2 }^{s'}, \quad s \in \mathscr S. 
\end{equation}
It is a standard property of RPC that \begin{equation}\label{eq: standard property of RPC}
    \e_h X_{0}^{s}= \e \log\sum_{\alpha\in\mathbb N^{r}} v_\alpha \ch(h_s+ C^s(\alpha)), \quad s \in \mathscr S.
\end{equation} 

The following differentiation is well-known (see, e.g., \cite[Lemma~3.6]{Tal06} or \cite[Section 14.7]{Talagrand2013vol2}), but we include its proof for completeness. 
\begin{lemma}\label{eq: classical calculation for derivative of X^s_0}
For $s,s'\in\mathscr S$ and $1\le l\le r-1$, 
    \begin{equation}\label{eq: derivative of X^s_0 in Q^s}
        \partial_{Q^s_l} \e_hX^{s'}_0= - \1_{\{s=s'\}} 2^{-1} (\zeta_l-\zeta_{l-1})\e_h\e_{1,\dots, l} W^{s'}_1\dots  W^{s'}_l \bigl(\e_{l+1,\dots,r} W^{s'}_{l+1}\dots W^{s'}_{r} \hth  Y^{s'}_r \bigr)^2,  \end{equation}
    which is understood as a right-derivative if $Q^s_l=Q^s_{l-1}$ and  a left-derivative if $Q^s_l=Q^s_{l+1}$.
    Similarly,
    \[\partial_{q^s_l} \e_hX^{s'}_0= -(\Lambda \Delta^2)_{ss'} (\zeta_l-\zeta_{l-1})\e_h\e_{1,\dots, l} W^{s'}_1\dots  W^{s'}_l \bigl(\e_{l+1,\dots,r} W^{s'}_{l+1}\dots W^{s'}_{r} \hth  Y^{s'}_r \bigr)^2,  \] which is understood as the right-derivative if $q^s_l=q^s_{l-1}$ and a left-derivative if $q^s_l=q^s_{l+1}$.
\end{lemma}

\begin{proof}
    Fix $s,s'\in\mathscr S$ and $1\le l\le r-1.$
    Assume $Q^s_{l-1}<Q^s_l < Q^s_{l+1}$.
We will use an RPC representation \eqref{eq: standard property of RPC} of $X^{s'}_0$ to take a derivative.
Recalling the notation in \eqref{eq: Gaussian process C},
\begin{equation}
   \partial_{Q_l^s}C^{s'}(\alpha) 
   =  \1_{\{s=s'\}} \Bigl( \frac{ z_{\alpha_l} }{(Q_{l}^{s'} -Q_{l-1}^{s'})^{1/2}} - \frac{ z_{\alpha_{l+1}} }{(Q_{l+1}^{s'} -Q_{l}^{s'})^{1/2}}\Bigr), \label{eq: partial C}
\end{equation}
From \eqref{eq: Gaussian process C} and \eqref{eq: partial C}, a careful verification shows
\begin{equation}\label{eq: IBP covariance}
    \e  \bigl[ C^{s'}(\alpha^2) \cdot \partial_{Q_l^s}C^{s'}(\alpha^1) \bigr] =  \1_{\{s=s'\}} \1_{\{\alpha^1\wedge\alpha^2=l\}}.
\end{equation}
From \eqref{eq: standard property of RPC}, we have \begin{align*}
    \partial_{Q_l^s} \e_h X_0^{s'} &= \sum_{\alpha\in\mathbb N^r}\e v_\alpha  \partial_{Q_l^s} C^{s'}(\alpha) \frac{\sh(h_{s'}+C^{s'}(\alpha))}{\sum_{\alpha\in\mathbb N^r}v_\alpha \ch(h_{s'}+C^{s'}(\alpha))}
    \\&= - \1_{\{s=s'\}} \sum_{\alpha^1\wedge \alpha^2=l}\e  \frac{v_{\alpha^1} v_{\alpha^2}\sh(h_{s'}+C^{s'}(\alpha^1))\sh(h_{s'}+C^{s'}(\alpha^2))}{\bigl(\sum_{\alpha\in\mathbb N^r}v_\alpha \ch(h_{s'}+C^{s'}(\alpha))\bigr)^2},
\end{align*} where we used Gaussian IBP in the second equality along with \eqref{eq: IBP covariance} and the assumption $l\le r-1$.
Now, we can identify the right-hand side of the preceding display with the desired expression by \cite[Proposition 14.3.2]{Talagrand2013vol2}.
To address the boundary case where either $Q^{s}_l=Q^{s}_{l-1}$ or $Q^{s}_l=Q^{s}_{l+1}$, we observe that the right-hand side of \eqref{eq: derivative of X^s_0 in Q^s} is continuous in $Q^s_l$ up to the boundary, so the one-sided derivatives exist and equal to the right-hand side of \eqref{eq: derivative of X^s_0 in Q^s}.

To verify the other equality, from the definition \eqref{eq: definition of Q}, one can check that for $1\le l, l' \le r-1$ and $s,s'\in\mathscr S$, we have $ 
    \partial_{q_l^s}Q_{l'}^{s'}=2\1_{\{l=l'\}} (\Delta^2 \Lambda)_{s's} = 2\1_{\{l=l'\}} (\Lambda \Delta^2 )_{ss'}$ by the symmetry of $\Delta^2$ and $\Lambda$.
    We can then apply the chain rule.

\end{proof}

As a preliminary work, we show that a certain matching between two monotone sets of points has an optimal cost in $L^1$ metric.
Recall the definition of cyclic monotonicity (see, e.g., \cite[Definition 2.4.2]{FG21}).
\begin{lemma}[$L^1$-cyclic monotonicity]\label{lem: cyclically monotone}
Let $r\ge 1$.
For two totally ordered sequences $q_1\lee\dots \lee q_r$ and $\tilde q_1\lee\dots \lee\tilde  q_r$ in $[0,1]^{\mathscr S}$, we have \[\sum_{1\le i\le r}\|q_i-\tilde q_i\|_1 \le \sum_{1\le i\le r}\|q_{i+1}-\tilde q_i\|_1,\] where we identify $q_{r+1}=q_1$.
\end{lemma}
\begin{remark}
    In fact, the proof indicates that we can have an arbitrary permutation of the indices on the right-hand side, instead of a shift.
\end{remark}
\begin{proof}
The statement is trivial for $r=1$.
To check the statement for $r=2,$ since we are considering the $L^1$ norm, it suffices to assume $|\mathscr S|=1$ and verify that for any real numbers $x_1\le x_2$ and $\tilde x_1\le \tilde x_2$, \[|x_1-\tilde x_1| +|x_2-\tilde x_2| \le |x_2-\tilde x_1| +| x_1-\tilde x_2|.\] 
The inequality is completely elementary to verify using the triangle inequality in various cases, and we leave the details to the reader.

Suppose the statement holds for some $r\ge 2.$
Consider two totally ordered sequences $q_1\lee \dots \lee q_{r+1}$ and $\tilde q_1\lee \dots \lee \tilde  q_{r+1}$ in $[0,1]^{\mathscr S}$, where we identify $q_{r+2}=q_1$.
By the induction hypothesis, \begin{align*}
    \sum_{1\le i\le r+1}\|q_i-\tilde q_i\|_1  &=\|q_{r+1}-\tilde q_{r+1}\|+ \sum_{1\le i\le r}\|q_i-\tilde q_i\|_1   
    \\&\le \|q_{r+1}-\tilde q_{r+1}\|_1+ \|q_1-\tilde q_r\|_1+ \sum_{1\le i\le r-1}\|q_{i+1}-\tilde q_i\|_1
    \\&\le \|q_1-\tilde q_{r+1} \|_1+  \|q_{r+1}-\tilde q_{r}\|_1 +\sum_{1\le i\le r-1}\|q_{i+1}-\tilde q_i\|_1
    \\&= \sum_{1\le i\le r+1}\|q_{i+1}-\tilde q_i\|_1,
\end{align*} where the second inequality follows from the $r=2$ case by considering the sequences $q_1\lee q_{r+1}$ and $\tilde q_r \lee \tilde q_{r+1}$.
This finishes the induction.
\end{proof}

\begin{proof}[\bf Proof of Proposition~\ref{prop: Lipschitz continuity of Parisi functional}]
    We may assume that $\mu_1$ and $\mu_2$ have common $(\zeta_l)_{0\le l\le r}$ in \eqref{eq: zeta sequence} by repeating the vectors $(q_l)_{1\le l\le r}$ with entries as in \eqref{eq: q sequence}.
    We may further assume  that $\zeta_0=0$ and $\zeta_{r-1}=1$ so that  $\operatorname{supp}(\mu_1)= \{q_1\lee\dots \lee q_{r-1}\}$ and $\operatorname{supp}(\mu_2)= \{\tilde q_1\lee\dots\lee \tilde q_{r-1}\}$.
    Linearly interpolate $q_l(t)= (1-t) q_l + t \tilde q_l$ for $1\le l\le r-1$ and $0\le t\le 1.$
    Define \[\phi(t)= \mathscr P ( (q_l(t))_{1\le i\le r}, (\zeta_i)_{1\le i\le r}), \quad 0\le t\le 1.\]
            Let $1\le l\le r-1$ and $s \in \mathscr S$.
Recalling \eqref{eq: definition of Q}, we can calculate $ \partial_{q_l^s}Q_{l'}=  2\1_{\{l=l'\}} (\Lambda\Delta^2\Lambda q_{l'})_s$ and \begin{equation*}
      \partial_{q_l^s} \sum_{l' =0}^{r-1}\zeta_{l'} (Q_{l'+1}-Q_{l'})=  2(\zeta_{l-1} -\zeta_l)  (\Lambda\Delta^2\Lambda q_l)_s  .
\end{equation*}
Together with Lemma~\ref{eq: classical calculation for derivative of X^s_0}, we arrive at the derivative of the Parisi functional \eqref{eq: Parisi functional for MSK}  \begin{equation}\label{eq: derivative of MSK Parisi functional}
    \partial_{q_l^s} \mathscr P ((q_l), (\zeta_l)) =(\zeta_l -\zeta_{l-1} ) \bigl(\Lambda \Delta^2 \Lambda (  q_l-  a_l)\bigr)_s,
\end{equation} 
where \[(a_l)_s \colonequals\e_h\e_{1,\dots, l} W^s_1\dots  W^s_l \bigl(\e_{l+1,\dots,r} W^s_{l+1}\dots W^s_{r} \hth  Y^s_r \bigr)^2 .\]
Then, by the chain rule, \[\phi'(t) =  \sum_{1\le l\le r} (\zeta_l-\zeta_{l-1})  (\tilde q_l -q_l)^\intercal \Lambda \Delta^2\Lambda ( q_l(t) - a_l(t) ),\] where $a_l(t)$ is $a_l$ evaluated at $(q_i(t))_{1\le i\le r-1}$.
   Since $ \0 \lee q_l(t), a_l(t)\lee \1$ for all $t\in[0,1]$, we have
    \[ |\mathscr P(\mu_1)-\mathscr P(\mu_2)| = |\phi(1)-\phi(0)|\le \sup_{0\le t\le 1} |\phi'(t)|\le \|\Lambda\Delta^2\Lambda\|_{\infty}\sum_{1\le l\le r} (\zeta_l-\zeta_{l-1}) \| \tilde q_l-q_l\|_1 .\]
The sum on the right-hand side of the preceding display corresponds to the coupling that assigns $\zeta_l-\zeta_{l-1}$ to $(q_l,\tilde q_l) \in ([0,1]^{\mathscr S})^{\otimes 2}$ for each $1\le l\le r-1.$
By Lemma~\ref{lem: cyclically monotone}, the support of this coupling, namely $\{(q_l,\tilde  q_l): 1\le l\le r\}$, is  $L^1$-cyclically monotone. 
Now, the proof is complete by \cite[Corollary 2.6.8]{FG21} which states that the coupling with a cyclically monotone support is an optimal one achieving $W_1(\mu_1,\mu_2)$.
\end{proof}

We collect some properties of the components of \eqref{eq: definition of vector valued map Psi}.
 \begin{lemma}[Adaptation of \cite{AC15_PTRF}, Proposition 3]\label{lem: properties of Psi}
    Let $\nu$ be probability measure on $[0,1]$ and $h$ be a centered Gaussian random variable. 
    Suppose  $\xi_V(x)=2^{-1}Vx^2$ for some $V>0.$    
    The function $[0,V] \ni w\mapsto \e_h \Gamma_{\nu,\xi_V}(h,V^{-1}w)$ is $[0,1]$-valued, continuously differentiable on $(0,V)$, increasing, and $1$-Lipschitz. 
    \end{lemma}
\begin{proof}
Equations (15), (17), and (18) in \cite{AC15_PTRF} immediately implies $0\le \e_h \Gamma_{\nu,\xi_V}(h,V^{-1}w) \le 1$.
It remains to check certain properties of its derivative, provided that it exists.
To this end, note $\xi_V''=V$. 
We can repeat the proof of \cite[Proposition 3]{AC15_PTRF} and show  \begin{enumerate}[label=(\roman*)]
    \item for each $x\in\mathbb R$, the function $ (0,V)\ni w\mapsto \partial_w \Gamma_{\nu,\xi_V}(x,V^{-1}w)$ is well-defined and continuous,
    \item for each $(x,w)\in \mathbb R \times (0,V)$, we have $0< \partial_w \Gamma_{\nu,\xi_V} (x,V^{-1}w)\le 1$ by equations (16), (17), and (20) in \cite{AC15_PTRF}.
\end{enumerate}
The dominated convergence theorem shows, which is applicable due to  (i), (ii), and the mean-value theorem,  that  \[ \frac{d}{dw}\e_h\Gamma_{\nu,\xi_V}(h,V^{-1}w)=  \e_h \partial_w\Gamma_{\nu,\xi_V}(h,V^{-1}w)\in (0,1] \quad \text{for all } w\in (0,V).\]
The continuity of  $w \mapsto \e_h \partial_w \Gamma_{\nu,\xi_V}(h,V^{-1}w)$ also follows from (i), (ii), and the dominated convergence theorem.
This completes the proof.
     
\end{proof}


    \begin{proof}[\bf Proof of Proposition~\ref{prop: derivative of push forward of Parisi functional}]
   This is a routine extension of \cite[Equation (3.34)]{Tal06} (or \cite[Equation (19)]{AC15_PTRF}) to a multi-dimensional setting, and it suffices to check for the case of a discrete measure $\mu\in \mathcal M^\uparrow_d.$
    Suppose $\mu(q_l)= \zeta_l-\zeta_{l-1}$ for some sequences \eqref{eq: zeta sequence} and \eqref{eq: q sequence} with $\zeta_0=0$ and $\zeta_{r-1}=1.$
Combining \eqref{eq: derivative of MSK Parisi functional} and \eqref{eq: Psi for discrete measure}, we can write, for $1\le l\le r-1$,   \begin{equation*}
    \partial_{q_l^s} \mathscr P ((q_l), (\zeta_l)) =(\zeta_l -\zeta_{l-1} ) \bigl(\Lambda \Delta^2 \Lambda (  q_l-  \e_h\Psi_\mu (h, q_l))\bigr)_s,
\end{equation*} from which the proof can be completed by the chain rule in this case.

    Note that, for any $p_1,p_2 \in \operatorname{supp}(\mu)$ with $p_1\lee p_2$, the given assumption on the function $a$ and the fact that the cone $\{p\in \mathbb R^{\mathscr S}: p \gee \0\}$ is convex together imply $\0\lee p_1+ta(p_1) \lee p_2+ta(p_2)\lee \1$ for all $0\le t\le 1$. 
    Consequently, $\mu_t \in\mathcal M^\uparrow$ and $\mathscr P(\mu_t)$ is well-defined for all $ 0\le t\le 1$. 

\end{proof}
We finally proceed to the following proof.
\begin{proof} [\bf Proof of Propostion~\ref{prop: support of Parisi measure}]
First of all, the positive-definiteness of $\Delta^2$ implies that no rows of $\Delta^2$ can be zero and, therefore,  $\Delta^2 \1 \sgee \0$.
     Define a function $a:\operatorname{supp}(\mu) \to \mathbb R^{\mathscr S}$ given by $a(q)= -q+ \e\Psi_\mu(h,  q)$.
    As a consequence of Lemma~\ref{lem: properties of Psi}, we have $a(\1)\lee \0\lee a(\0)$, and for $p_1,p_2 \in \operatorname{supp}(\mu)$ with $p_1\lee p_2 $, \[ p_1 +a(p_1) = \e_h\Psi_\mu(h,  p_1) \lee \e_h\Psi_\mu(h,  p_2) =p_2 +a(p_2).\]
    In particular, $\0 \lee  p_1 +a(p_1) \lee p_2 +a(p_2) \lee \1$ whenever $p_1\lee p_2$ for $p_1,p_2 \in\operatorname{supp}(\mu).$
    As a result, Proposition~\ref{prop: derivative of push forward of Parisi functional} is applicable and yields \[ \frac{d}{dt}\mathscr P(\mu_t)\Bigr|_{t=0^+}= - \int_{\operatorname{supp}(\mu)}   (q-\e_h\Psi_\mu(h,  q))^\intercal \Lambda\Delta^2\Lambda  (q-\e_h\Psi_\mu(h,  q))\,d\mu(q).\]
    Since $\mu$ is a minimizer of $\mathscr P$ and $\Delta^2$ is positive-definite, we deduce from the preceding display that $q=\e_h\Psi_\mu(h, q)$ for all $q\in\operatorname{supp}(\mu).$
   
\end{proof}

\section{Proof of continuity of the Parisi formula in model parameters}\label{sec: Proof of continuity of Parisi formula in model parameters}
We  establish that the limiting free energy and the Parisi functional \eqref{eq: Parisi functional for MSK} are both continuous with respect to the model parameters.
Throughout this appendix, we will always denote the free energy and the Parisi functional corresponding to the parameters $\tilde \Delta^2, \tilde \tau^2$, and $\tilde \Lambda \colonequals \text{diag}(\tilde \lambda_s:s\in\mathscr S)$ by $\tilde F_N$ and $\tilde {\mathscr P}$.

We remark that in the single-species case of the SK model, all the considerations in this section are more than necessary as can be seen from the simple observation that the free energy is jointly convex in the inverse temperature and external field.

\begin{lemma}\label{lem: free energy Lipschitz}
   We have \begin{align*}
        \limsup_{N\to \infty}|\e F_N- \e \tilde F_N| &\le \max_{s,s'\in\mathscr S}|\Delta^2_{ss'}-\tilde \Delta^2_{ss'}| + \max_{s,s'\in\mathscr S} (\Delta^2_{ss'}\vee \tilde \Delta^2_{ss'})  \sum_{s,s'\in\mathscr S} |\lambda_{s}\lambda_{s'}- \tilde \lambda_{s} \tilde \lambda_{s'}|
        \\&\quad +\max_{s\in\mathscr S}|\tau^2_{s}-\tilde \tau^2_{s}| + \max_{s\in\mathscr S} (\tau^2_{s}\vee \tilde \tau^2_{s})  \sum_{s\in\mathscr S} |\lambda_{s} - \tilde \lambda_{s} | .
        \end{align*} 
\end{lemma}
\begin{proof}
    We may assume that $(\tilde J,\tilde h)$ is independent of $(J,h),$ 

    Consider the following matching procedure.
    For each $N$, we have two partitions of $[N]\colonequals \{1,\dots, N\}$, $(I_{s,N})_{s\in\mathscr S}$ and $(\tilde I_{s,N})_{s\in\mathscr S}$, which correspond to the free energies $\e F_N$ and $\e \tilde F_N$, respectively.
    For each $(s,s')\in\mathscr S\times \mathscr S$ and $N\in\mathbb N$, we define $P_{ss',N}\subseteq I_{s,N}\times I_{s',N}$ as follows.
    If $\lambda_{s,N}\lambda_{s',N}\le  \tilde \lambda_{s,N} \tilde \lambda_{s',N}$, then we set $P_{ss',N}= I_{s,N}\times I_{s',N}$. 
    If $\lambda_{s,N}\lambda_{s',N}>  \tilde \lambda_{s,N} \tilde \lambda_{s',N}$, we set $P_{ss',N}$ to be an arbitrary (strict) subset of $ I_{s,N}\times I_{s',N}$ having the same cardinality as  $\tilde I_{s,N}\times \tilde I_{s',N}$.
    With this construction of $(P_{ss',N})_{s,s' \in\mathscr S}$, we can construct a bijection  $f: [N]^2 \to [N]^2$ 
    which satisfies \begin{enumerate}[label=(\roman*)]
        \item $f(P_{ss',N})\subseteq \tilde I_{s,N}\times \tilde I_{s',N}$ for all $s,s'\in\mathscr S$,
        \item $f(P_{ss',N})= \tilde I_{s,N}\times \tilde I_{s',N}$ if and only if $\lambda_{s,N}\lambda_{s',N}>  \tilde \lambda_{s,N} \tilde \lambda_{s',N}$,
        \item All elements in $\cup_{s,s': \lambda_{s,N}\lambda_{s',N}>  \tilde \lambda_{s,N} \tilde \lambda_{s',N}} ((I_{s,N}\times I_{s',N}) \setminus P_{ss',N})$ are bijectively mapped into $\cup_{s,s': \lambda_{s,N}\lambda_{s',N}\le  \tilde \lambda_{s,N} \tilde \lambda_{s',N}} (( \tilde I_{s,N}\times \tilde I_{s',N}) \setminus f(P_{ss',N}))$ by $f.$
    \end{enumerate} 
    Observe that \begin{equation}\label{eq: non-matched pairs}
        \frac{\bigl|[N]^2\setminus \cup_{s,s'\in\mathscr S}P_{ss',N}\bigr|}{N^2}= \sum_{s,s'\in\mathscr S} (\lambda_{s,N}\lambda_{s',N}- \tilde \lambda_{s,N} \tilde \lambda_{s',N})\1_{\{ \lambda_{s,N}\lambda_{s',N}>  \tilde \lambda_{s,N} \tilde \lambda_{s',N}\}}.
    \end{equation}
    We can similarly construct subsets $J_{s,N}\subseteq I_{s,N}$ where the equality holds if and only if $\lambda_{s,N}\le \tilde\lambda_{s,N}$, and from this, a bijection $g:[N]\to [N]$ such that $g(J_{s,N})\subseteq \tilde I_{s,N}$ for all $s\in\mathscr S$, and equality if $\lambda_{s,N}> \tilde \lambda_{s,N}$.
    Moreover, \begin{equation}\label{eq: non-matched sites}
        \frac{\bigl|[N]\setminus \cup_{s\in\mathscr S}J_{s,N}\bigr|}{N}= \sum_{s\in\mathscr S} (\lambda_{s,N}- \tilde \lambda_{s,N} )\1_{\{ \lambda_{s,N}>  \tilde \lambda_{s,N} \}}.
    \end{equation}
    
    Define the interpolating Gaussian random variable $J_{ij}(t) =\sqrt{t}J_{ij}+\sqrt{1-t}\tilde J_{f(i,j)}$ and $h_i(t)=\sqrt{t}h_i+\sqrt{1-t}\tilde h_{g(i)}$ for $0\le t\le 1$ and $1\le i,j\le N$. 
    Define the corresponding Hamiltonian $H_t(\sigma)= N^{-1/2}\sum_{i,j}J_{ij}(t) \sigma_i\sigma_j +\sum_{i}h_i(t)\sigma_i$ and the free energy $F_t= N^{-1}\log \sum_{\sigma} \exp(H_t(\sigma))$ for $0\le t\le 1$.
Denote the Gibbs expectation associated to the Hamiltonian  by $\lla \cdot \rra_{t}$.
 With a slight abuse of notation, denote $s(f(i,j))= (s(\pi_1 \circ f(i,j)), s(\pi_2 \circ f(i,j) ) ) $, where $\pi_i$ is the projection into $i$'th coordinate.
    By the standard Gaussian interpolation (see \cite[Lemma~1.3]{Pan13_SK_book}), \begin{align*}
        \partial_t  \e F_t &= \frac{1}{2N^2}  \sum_{1\le i,j\le N}\bigl(1-\e \lla \sigma_i\sigma_j\rra_t^2 \bigr) \bigl(\Delta^2_{s(i)s(j)}- \tilde \Delta^2_{s(f(i,j))}\bigr) 
        \\& \;  +\frac{1}{2N} \sum_{1\le i\le N} \bigl(1-\e \lla \sigma_i\rra_t^2\bigr)(\tau^2_{s(i)}-\tilde \tau^2_{s(g(i))}).
    \end{align*} 
    Using \eqref{eq: non-matched pairs} and \eqref{eq: non-matched sites}, we can estimate \begin{align*}
        |\partial_t  \e F_t|  &\le \frac{1}{2N^2}  \sum_{1\le i,j\le N} |\Delta^2_{s(i)s(j)}- \tilde \Delta^2_{s(f(i,j))}|   + \frac{1}{2N} \sum_{1\le i\le N}|\tau^2_{s(i)}-\tilde \tau^2_{s(g(i))}|
        \\&\le \max_{s,s'\in\mathscr S}|\Delta^2_{ss'}-\tilde \Delta^2_{ss'}| + \max_{s,s'\in\mathscr S} (\Delta^2_{ss'}\vee \tilde \Delta^2_{ss'})  \sum_{s,s'\in\mathscr S} |\lambda_{s,N}\lambda_{s',N}- \tilde \lambda_{s,N} \tilde \lambda_{s',N}|
        \\&\quad +\max_{s\in\mathscr S}|\tau^2_{s}-\tilde \tau^2_{s}| + \max_{s\in\mathscr S} (\tau^2_{s}\vee \tilde \tau^2_{s})  \sum_{s\in\mathscr S} |\lambda_{s,N} - \tilde \lambda_{s,N}|,
    \end{align*}
    from which the fundamental theorem of calculus and  $\e F_1= \e F_N$, $\e F_0=\e \tilde F_N$ complete the proof. 
\end{proof}

\begin{lemma}
For any $r\ge 2$ and sequences in \eqref{eq: zeta sequence} and \eqref{eq: q sequence}, we have the following: for $a,b, s\in\mathscr S$,
\begin{align*}
    &0\le \partial_{\tau^2_{a}} \e_hX^s_0\le  2^{-1}\1_{\{\tau^2_s>0, \, a=s\}},
    \\&0\le \partial_{\Delta^2_{ab}} \e_hX^s_0\le  \1_{\{a=s\}}\lambda_b,
    \\& 0\le \partial_{\lambda_a} \e_hX^s_0\le  \Delta^2_{sa}.
\end{align*}
\end{lemma}
\begin{proof}
    The calculations will be similar to that of Lemma~\ref{eq: classical calculation for derivative of X^s_0}.
Fix $s\in\mathscr S$.
    Since $h_s \stackrel{d}{=}\sqrt{\tau^2_s}z$ for some standard Gaussian random variable $z$ independent of everything else, we can rewrite  \eqref{eq: standard property of RPC} as \[\e X_{0}^{s}= \e \log\sum_{\alpha\in\mathbb N^{r}} v_\alpha \ch( \sqrt{\tau^2_s} z+ C^s(\alpha)).\] 
    Using Gaussian IBP, \begin{align*}
        \partial_{\tau^2_a}\e X^s_0 &=  \1_{\{\tau^2_s>0, \, a=s\}} \sum_{\alpha\in\mathbb N^r}\e \Bigl( \frac{z}{2\sqrt{\tau^2_s}} \cdot  \frac{v_\alpha  \sh(\sqrt{\tau^2_s} z+C^{s}(\alpha))}{\sum_{\alpha\in\mathbb N^r}v_\alpha \ch(\sqrt{\tau^2_s} z+C^{s}(\alpha))} \Bigr)
        \\&= 2^{-1}\1_{\{\tau^2_s>0, \, a=s\}}  \bigl(1 - \e\lla \hth (\sqrt{\tau^2_s} z+C^{s}(\alpha) \rra^2  \bigr),
    \end{align*} where $\lla \cdot \rra$ denotes the expectation on $\mathbb N^r$ with weights $(v_\alpha)_{\alpha\in\mathbb N^r}.$
    This immediately verifies the  assertion about the derivative in $\tau^2_a.$

We will check the other two inequalities.
From definition \eqref{eq: definition of Q}, we have
    \begin{equation*}
        \partial_{\Delta^2_{ab}} Q^s_l = 2 \1_{\{a=s\}}\lambda_b q^b_l, \quad \partial_{\lambda_a}Q^s_l= 2\Delta^2_{sa}q^a_l.
    \end{equation*}
    By considering smaller $r$ if necessary, we may assume $Q^s_l>Q^s_{l-1}$ for all $1\le l\le r.$
    Recalling the definition \eqref{eq: Gaussian process C},
\begin{equation*}
   \partial_{\Delta^2_{ab}}C^{s}(\alpha) = \sum_{\beta \in p(\alpha)} \frac{z_\beta (\partial_{\Delta^2_{ab}}Q_{|\beta|}^{s} - \partial_{\Delta^2_{ab}}Q_{|\beta|-1}^{s})}{2(Q_{|\beta|}^{s} -Q_{|\beta|-1}^{s})^{1/2}} = \1_{\{a=s\}} \lambda_b \sum_{\beta \in p(\alpha)} \frac{z_\beta ( q^b_{|\beta|} -  q^b_{|\beta|-1})}{(Q_{|\beta|}^{s} -Q_{|\beta|-1}^{s})^{1/2}}.
\end{equation*}
From \eqref{eq: Gaussian process C} and the preceding display, a careful calculation shows
\begin{equation*}
    \e  \bigl[ C^{s}(\alpha^2) \cdot \partial_{\Delta^2_{ab}}C^{s}(\alpha^1) \bigr] =  \1_{\{a=s\}} \lambda_b  q^b_{\alpha^1\wedge \alpha^2},
\end{equation*} from which Gaussian IBP yields 
\begin{align*}
    \partial_{\Delta^2_{ab}} \e_h X_0^{s} &= \sum_{\alpha\in\mathbb N^r}\e v_\alpha  \partial_{\Delta^2_{ab}} C^{s}(\alpha) \frac{\sh(h_{s}+C^{s}(\alpha))}{\sum_{\alpha\in\mathbb N^r}v_\alpha \ch(h_{s}+C^{s}(\alpha))}
    \\&=  \1_{\{a=s\}} \lambda_b   \bigl(1- \e \lla q^b_{\alpha^1\wedge \alpha^2} \hth( h_s +C^s(\alpha^1)\hth( h_s +C^s(\alpha^2)\rra \bigr),
\end{align*}  and the derivative in $\Delta^2_{ab}$ is checked.
We can similarly calculate the derivative in $\lambda_a$, \[\partial_{\lambda_a} \e_h X_0^{s}= \Delta^2_{sa}  \bigl(1- \e \lla q^a_{\alpha^1\wedge \alpha^2} \hth( h_s +C^s(\alpha^1)\hth( h_s +C^s(\alpha^2)\rra \bigr),\] which finishes the proof.

\end{proof}
The preceding lemma implies  \begin{align*}
    |\e_h X_0^s - \e_{\tilde h} \tilde X_0^s| &\le  2^{-1}|\tau^2_s -\tilde \tau^2_s|+ \sum_{b\in\mathscr S} \max (\lambda_b, \tilde \lambda_b) |\Delta^2_{sb}-\tilde \Delta^2_{sb}| + \sum_{a\in\mathscr S} \max(\Delta^2_{sa},\tilde \Delta^2_{sa}) |\lambda_a-\tilde \lambda_a|
    \\&\le 2^{-1}|\tau^2_s -\tilde \tau^2_s|+ \sum_{b\in\mathscr S} |\Delta^2_{sb}-\tilde \Delta^2_{sb}| + \max_{a\in\mathscr S}(\Delta^2_{sa} \vee  \tilde \Delta^2_{sa}) \sum_{a\in\mathscr S} |\lambda_a-\tilde \lambda_a|. 
\end{align*} 
By \eqref{eq: standard property of RPC}, $\e_h X^s_0 \ge 0.$
Moreover, applying Jensen's inequality to \eqref{eq: standard property of RPC} and using \eqref{eq: covariance of Gaussian process C}, \[\e_h X^s_0 \le  \e \log \sum_{\alpha \in\mathbb N^r} v_\alpha \e \ch(h_s+C^s(\alpha))\le  2^{-1}\tau^2_s +(\Delta^2 \Lambda \1)_s.\]
Note \[
    \Bigl| \sum_{s\in\mathscr S} \lambda_s \e_h X_0^s -\sum_{s\in\mathscr S} \tilde \lambda_s \e_{\tilde h} \tilde X_0^s \Bigr| \le \sum_{s\in \mathscr S}|\lambda_s-\tilde\lambda_s|\cdot  \e_h X^s_0 + \sum_{s\in\mathscr S} |\e_h X_0^s - \e_{\tilde h} \tilde X_0^s| \cdot \lambda_s.
\]
The last three displays together show \begin{align*}
    \Bigl| \sum_{s\in\mathscr S} \lambda_s \e_h X_0^s -\sum_{s\in\mathscr S} \tilde \lambda_s \e_{\tilde h} \tilde X_0^s \Bigr| &\le \max_{s\in\mathscr S}\bigl(2^{-1}\tau^2_s +(\Delta^2 \Lambda \1)_s\bigr) \cdot \sum_{s\in \mathscr S}|\lambda_s-\tilde\lambda_s| 
    \\  &+\max_{s\in\mathscr S}\frac{|\tau^2_s -\tilde \tau^2_s|}{2}
    + \max_{s\in\mathscr S}\sum_{a\in\mathscr S} |\Delta^2_{sa}-\tilde \Delta^2_{sa}| 
    \\&+ \max_{s,a\in\mathscr S}(\Delta^2_{sa} \vee  \tilde \Delta^2_{sa}) \sum_{a\in\mathscr S} |\lambda_a-\tilde \lambda_a|,
\end{align*}
which shows the Lipschitz continuity of the second term of the Parisi functional.
On the other hand, using summation by parts, we can write the third term in the Parisi functional 
\[
\sum_{l=0}^{r-1} \zeta_l \,(Q_{l+1} - Q_l)
= \zeta_{r-1} Q_r - \sum_{l=0}^{r-1} (\zeta_l - \zeta_{l-1})\, Q_l,
\]
from which we obtain the uniform upper bound of the difference 
\begin{align*}
    \Bigl| \sum_{l=0}^{r-1} \zeta_l (Q_{l+1} - Q_l) -\sum_{l=0}^{r-1} \zeta_l (\tilde Q_{l+1} - \tilde Q_l)\Bigr|
&\le  \zeta_{r-1} |Q_r-\tilde Q_r| + \sum_{l=0}^{r-1} (\zeta_l - \zeta_{l-1})  |Q_l-\tilde Q_l |
\\&\le \max_{0\le l\le r} |Q_l-\tilde Q_l|
\\&\le \sum_{a,b\in\mathscr S} |(\Lambda\Delta^2\Lambda -\tilde \Lambda \tilde \Delta^2 \tilde \Lambda)_{ab}|.
\end{align*}
We can summarize the above findings as follows.
\begin{lemma}\label{lem: Lipschitz of Parisi functional in model parameters}
    For any $r\ge 2$ and sequences in \eqref{eq: zeta sequence} and \eqref{eq: q sequence}, we have
    \begin{align*}
        |\mathscr P ((q_l),(\zeta_l))- \tilde {\mathscr P} ((q_l),(\zeta_l))| &\le \max_{s\in\mathscr S}\bigl(2^{-1}\tau^2_s +(\Delta^2 \Lambda \1)_s\bigr) \cdot \sum_{s\in \mathscr S}|\lambda_s-\tilde\lambda_s| 
    \\  &+\max_{s\in\mathscr S}\frac{|\tau^2_s -\tilde \tau^2_s|}{2}+ \max_{s\in\mathscr S}\sum_{s'\in\mathscr S} |\Delta^2_{ss'}-\tilde \Delta^2_{ss'}| 
    \\&+ \max_{s,s'\in\mathscr S}(\Delta^2_{ss'} \vee  \tilde \Delta^2_{ss'}) \sum_{s\in\mathscr S} |\lambda_s-\tilde \lambda_s|
    + \frac{1}{2}\sum_{s,s'\in\mathscr S} |(\Lambda\Delta^2\Lambda -\tilde \Lambda \tilde \Delta^2 \tilde \Lambda)_{ss'}|.
    \end{align*}
    In particular, $|\inf_{r,(q_l),(\zeta_l)}\mathscr P ((q_l),(\zeta_l))- \inf_{r,(q_l),(\zeta_l)}\tilde {\mathscr P} ((q_l),(\zeta_l))|$ has the same upper bound.
\end{lemma}


\section{Support Cardinality of Parisi Measures}\label{sec: cardinality of supports of Parisi measures}
For convenience, we denote the continuous map $S=2\Delta^2\Lambda$ on $\mathbb R^{\mathscr S}$ throughout this section.
In view of \cite[Theorem 1.2]{CIM25}, the following lemma readily shows that the supports of any Parisi measures have the same cardinality.
\begin{lemma}
    Suppose  $S_*(\mu)=\nu$ for some $\mu \in\mathcal M^\uparrow$ and a probability measure $\nu$ with totally ordered support in  $[0,\infty)^{\mathscr S}$, where $S_*$ denotes the push-forward by $S.$
    For any  $x\in \operatorname{supp}(\nu)$, we have $|\operatorname{supp}(\mu)\cap S^{-1}(x)|=1.$
    Moreover, $\operatorname{supp}(\mu) \subseteq S^{-1}\operatorname{supp}(\nu)$
    and, therefore, $|\operatorname{supp}(\mu)|=|\operatorname{supp}(\nu)|.$
\end{lemma}
\begin{proof}
    Let us  denote the union of the positive and negative cones by $K=[0,\infty)^{\mathscr S}\cup (-\infty,0]^{\mathscr S}$.
Since $S$ is irreducible, $\operatorname{ker}(S)\cap K=\{\0\}$ by the Perron--Frobenius theorem \cite[(8.3.9)]{Mey00_book}.  
    Fix $x\in\operatorname{supp}(\nu).$
    The uniqueness is simple to check as follows. 
    Suppose $y,z \in \operatorname{supp}(\mu)\cap S^{-1}(x)$ and $y\ne z$.
    This means $y=z+a$ for some $\0\ne a\in \operatorname{ker}(S)$.
    Then, $a \notin K$ and $\{y,z\}$ is not totally ordered, which is a contradiction to $\mu \in\mathcal M^\uparrow$.
    To establish existence, note $0<\nu(\mathsf{B}_{1/n}(x))=  \mu (S^{-1}(\mathsf{B}_{1/n}(x)))\le \mu (S^{-1}(\operatorname{cl}(\mathsf{B}_{1/n}(x))))$ for every $n\ge 1.$ 
    In particular, $(\operatorname{supp}(\mu)\cap S^{-1}(\operatorname{cl}(\mathsf{B}_{1/n}(x))))_{n\ge 1}$ is a decreasing family of nonempty closed sets.
    By the compactness of $\operatorname{supp}(\mu)$, we must have $\emptyset \ne \cap_{n\ge 1}\bigl(\operatorname{supp}(\mu)\cap S^{-1}(\operatorname{cl}(\mathsf{B}_{1/n}(x)))\bigr)= \operatorname{supp}(\mu)\cap S^{-1}(x) $.

    Next, to check  $\operatorname{supp}(\mu) \subseteq S^{-1}\operatorname{supp}(\nu),$ suppose $y\in \operatorname{supp}(\mu)$.
    Then, $\nu(\mathsf{B}_\eps (Sy))= \mu (S^{-1}\mathsf{B}_\eps (Sy)) \\>0,$ where the inequality follows from the fact that $S^{-1}(\mathsf{B}_\eps (Sy))$ is an open set containing $y$ and $y\in\operatorname{supp}(\mu).$
    This means that $Sy \in \operatorname{supp}(\nu)$, as desired.
    
\end{proof}

\section{Simultaneous replica symmetry breaking}\label{sec: Proof of simultaneous RSB}

As a simple byproduct of our proof of  Theorem~\ref{thm: main theorem}, we can address whether all species are simultaneously RSB; that is, the projections of a Parisi measure into each species have supports of the same cardinality.
\begin{proposition}
    Consider MSK with deterministic external field.
    Suppose $\Delta^2$ is irreducible and positive-definite.
Let $(\pi_s)_*$ denote the push-forward  by the projection map $\pi_s$ into the $s$-species.
For every Parisi measure $\mu$ on $[0,1]^{\mathscr S}$, we have $ |\operatorname{supp} ( (\pi_s)_*\mu)|=|\operatorname{supp}( (\pi_{s'})_*\mu)|$ for all $s,s'\in\mathscr S$.
\end{proposition}
\begin{remark}\label{rmk: simultaneous RSB}
    This result is not new and was obtained in Chen--Mourrat \cite[Theorem 6.2]{CM24} for possibly indefinite $\Delta^2$ by  analyzing a critical point equation of a certain functional, which is conceptually similar to Lemma~\ref{prop: support of Parisi measure}.
Simultaneous RSB in the setting of multi-species spherical mixed $p$-spin models with convex interaction was also established in Bates--Sohn \cite{BS22_CMP}.
We note that in all of these results, the support cardinality of a Parisi measure is not  determined, but we expect full-RSB at low temperatures, meaning a nontrivial continuous component of the support; see Zhou~\cite{Zho25} for recent progress on full-RSB in the single-species SK model without external field at low temperature just below the criticality.

\end{remark}
\begin{proof}

Let $\mu\in \mathcal M^\uparrow$ be any Parisi measure and suppose $p,q \in\operatorname{supp}(\mu)$ satisfy $p_s < q_s$ for some species $s\in\mathscr S.$
We aim to show $p_k<q_k$ for any $k\in\mathscr S$ such that $\Delta^2_{ks}>0$.
In spirit of \cite{BS22}, we will achieve this by examining  an identity satisfied by the support of a Parisi measure;
by Lemma~\ref{prop: support of Parisi measure} which also holds for deterministic external fields, it suffices to establish \[\Psi_\mu (h,p)_k < \Psi_\mu (h,q)_k,\] which is equivalent to, in view of definition \eqref{eq: definition of vector valued map Psi}, $f_k(p)<f_k(q)$ where \[f_k(x) \colonequals \Gamma_{\mu\circ   (2\Delta^2\Lambda)^{-1} \circ \pi_k^{-1}\circ S_k, \; \xi_k}\Bigl(h_k, \frac{(2\Delta^2\Lambda x)_k}{(2\Delta^2\Lambda \1)_k}\Bigr), \quad x\in [0,1]^{\mathscr S}.\]
By the chain rule, we have \[\partial_{x_s} f_k(x)= \frac{(\Delta^2\Lambda)_{ks}}{(\Delta^2\Lambda \1)_k} \partial_2 \Gamma_{\mu\circ   (2\Delta^2\Lambda)^{-1} \circ \pi_k^{-1}\circ S_k, \; \xi_k} \Bigl(h_k, \frac{(2\Delta^2\Lambda x)_k}{(2\Delta^2\Lambda \1)_k}\Bigr) \quad  \text{for any } s\in\mathscr S,\] where we denote the derivative in the second coordinate by $\partial_2$.
By the fundamental theorem of calculus, \[f_k(q)-f_k(p) = \sum_{l\in\mathscr S} (q_l-p_l)\frac{(\Delta^2\Lambda)_{kl}}{(\Delta^2\Lambda \1)_k} \int_0^1 \partial_2 \Gamma_{\mu\circ   (2\Delta^2\Lambda)^{-1} \circ \pi_k^{-1}\circ S_k, \; \xi_k} \Bigl(h_k, \frac{(\Delta^2\Lambda (tq+(1-t)p))_k}{(\Delta^2\Lambda \1)_k}\Bigr)   dt.\]
By item (ii) in the proof of Lemma~\ref{lem: properties of Psi}, the integrand hence the integral on the right-hand side of the preceding display is positive.
As a result, $f_k(q)>f_k(p)$ for all $k\in\mathscr S$ such that $(\Delta^2\Lambda)_{ks}>0$, i.e., $\Delta^2_{ks}>0$.
Together with the irreducibility of $\Delta^2$, this establishes bijections between the supports of all the species.

\end{proof}

\bibliographystyle{acm}
\setlength{\bibsep}{0.5pt}   

{\footnotesize\bibliography{references}}

\end{document}